\documentclass[a4, 12pt]{amsart}
\usepackage{amssymb,amstext,amscd,amsthm,amsfonts,mathdots,mathrsfs}
\usepackage{pdflscape}
\usepackage{hyperref}
\usepackage{pdflscape}
\usepackage{setspace}% for spacing
\usepackage[all]{xy}% for quiver
\usepackage{enumerate}% for [(1)] or [(a)]
\usepackage{indentfirst}%for indent
\usepackage{ytableau}
\usepackage{color}% for color
\usepackage{colortbl}
\usepackage{amsmath}% for the equation label
\usepackage{tikz}% for tikz quiver
\usetikzlibrary{shapes.geometric, arrows}
\usetikzlibrary{calc}
\numberwithin{equation}{section}

\usepackage[thicklines]{cancel}

\usepackage{cleveref}
\crefname{alphatheorem}{Theorem}{Theorems}

\newtheorem{theorem}{Theorem}[section]
\newtheorem{lemma}[theorem]{Lemma}
\newtheorem{proposition}[theorem]{Proposition}
\newtheorem{definition}[theorem]{Definition}
\newtheorem{conjecture}[theorem]{Conjecture}
\newtheorem{corollary}[theorem]{Corollary}
\newtheorem{example}[theorem]{Example}

\theoremstyle{remark}
\newtheorem{rem}[theorem]{Remark}

\newtheorem{innercustomthm}{{\bf Theorem}}
\newenvironment{customthm}[1]
{\renewcommand\theinnercustomthm{#1}\innercustomthm}{\endinnercustomthm}

\newcommand{\clr}{rgb:black,2;blue,2;red,0}

\tikzset{anchorbase/.style={baseline={([yshift=-0.5ex]current bounding box.center)}}}
\tikzset{wipe/.style={white,line width=4pt}}

\DeclareFontFamily{OT1}{pzc}{}
\DeclareFontShape{OT1}{pzc}{m}{it}{ <-> s*[1.2] pzcmi7t }{}
\DeclareMathAlphabet{\mathpzc}{OT1}{pzc}{m}{it}
\newcommand{\AH}{\mathpzc{AH}}
\newcommand{\HH}{\mathpzc{H}}
\newcommand{\W}{\mathpzc{Web}}
\newcommand{\AW}{\mathpzc{Web}^\bullet}  
\newcommand{\AWC}{\W^{\bullet\prime}}
\newcommand{\Sch}{\mathpzc{Schur}}
\newcommand{\ASch}{\mathpzc{Schur}^\bullet} 
\newcommand{\WSch}{\mathpzc{wSchur}}
\newcommand{\Mat}{\text{Mat}}
\newcommand{\PMat}{\text{ParMat}}
\newcommand{\Par}{\text{Par}}

\newcommand{\x}{\textsc{x}}
\newcommand{\bfc}{\mathbf{c}}
\newcommand{\bfu}{\mathbf{u}}

\newcommand\C{\mathbb{C}}
\newcommand\Z{\mathbb{Z}}

\newcommand\N{\mathbb{N}}
\newcommand\kk{\Bbbk}
\newcommand\la{\lambda}
\newcommand{\Hom}{{\rm Hom}}
\newcommand{\End}{{\rm End}}
\newcommand{\rot}{\rotatebox[origin=c]{180}}
\newcommand{\arxiv}[1]{\href{http://arxiv.org/abs/#1}{\tt arXiv:\nolinkurl{#1}}}

\def\t{\mathfrak t}
\def\Sc{\mathcal S}

\newcommand{\str}{\begin{tikzpicture}[baseline = 10pt, scale=0.5, color=\clr]
            \draw[-,thick] (0,0.5)to[out=up,in=down](0,1.7);
            \draw (0,0.2) node{$\scriptstyle 1$};
\end{tikzpicture} 
}

\newcommand{\stra}{\begin{tikzpicture}[baseline = 10pt, scale=0.5, color=\clr]
            \draw[-,line width=1.6pt] (0,0.5)to[out=up,in=down](0,1.7);
            \draw (0,0.2) node{$\scriptstyle a$};
\end{tikzpicture} 
}

\newcommand{\strm}{\begin{tikzpicture}[baseline = 10pt, scale=0.5, color=\clr]
            \draw[-,line width=1.6pt] (0,0.5)to[out=up,in=down](0,1.7);
            \draw (0,0.2) node{$\scriptstyle m$};
\end{tikzpicture} 
}
\newcommand{\bdot}{ node[circle, draw, fill=\clr, thick, inner sep=0pt, minimum width=3pt]{}}

\newcommand{\merge}
{\begin{tikzpicture}[baseline = -.5mm,scale=.8, color=\clr]
	\draw[-,line width=1pt] (0.28,-.3) to (0.08,0.04);
	\draw[-,line width=1pt] (-0.12,-.3) to (0.08,0.04);
	\draw[-,line width=1.5pt] (0.08,.4) to (0.08,0);
        \node at (-0.22,-.4) {$\scriptstyle a$};
        \node at (0.35,-.4) {$\scriptstyle b$};\node at (0,.55){$\scriptstyle a+b$};\end{tikzpicture} }
        
\newcommand{\splits}
{\begin{tikzpicture}[baseline = -.5mm,scale=.8, color=\clr]
	\draw[-,line width=1.5pt] (0.08,-.3) to (0.08,0.04);
	\draw[-,line width=1pt] (0.28,.4) to (0.08,0);
	\draw[-,line width=1pt] (-0.12,.4) to (0.08,0);
        \node at (-0.22,.5) {$\scriptstyle a$};
        \node at (0.36,.5) {$\scriptstyle b$};
        \node at (0.1,-.45){$\scriptstyle a+b$};
\end{tikzpicture}}

\newcommand{\dotgen}
{\begin{tikzpicture}[baseline = 3pt, scale=0.5, color=\clr]
\draw[-,thick] (0,0) to[out=up, in=down] (0,1.4);
\draw(0,0.6) \bdot;
\node at (0,-.3) {$\scriptstyle 1$};
\end{tikzpicture}}

\newcommand{\cross}{\begin{tikzpicture}[baseline=-.5mm,color=\clr]
	\draw[-,thick] (-0.3,-.3) to (.3,.4);
	\draw[-,thick] (0.3,-.3) to (-.3,.4);
\end{tikzpicture}}

\newcommand{\crossing}{\begin{tikzpicture}[baseline=-.5mm,scale=.8, color=\clr]
	\draw[-,line width=1.2pt] (-0.3,-.3) to (.3,.4);
	\draw[-,line width=1.2pt] (0.3,-.3) to (-.3,.4);
        \node at (0.3,-.4) {$\scriptstyle b$};
        \node at (-0.3,-.4) {$\scriptstyle a$};
         \node at (0.3,.55) {$\scriptstyle a$};
        \node at (-0.3,.55) {$\scriptstyle b$};
\end{tikzpicture}}

\newcommand{\wkdota}{\begin{tikzpicture}[baseline = 3pt, scale=0.5, color=\clr]
\draw[-,line width=1.2pt] (0,0) to[out=up, in=down] (0,1.4);
\draw(0,0.6) \bdot; 
\draw (0.7,0.6) node {$\scriptstyle \omega_r$};
\node at (0,-.3) {$\scriptstyle a$};
\end{tikzpicture} }

\newcommand{\wkdotaa}{\begin{tikzpicture}[baseline = 3pt, scale=0.5, color=\clr]
\draw[-,line width=1.5pt] (0,0) to[out=up, in=down] (0,1.4);
\draw(0,0.6) \bdot; 
\draw (0.7,0.6) node {$\scriptstyle \omega_a$};
\node at (0,-.3) {$\scriptstyle a$};
\end{tikzpicture} }

\newcommand{\wkdotr}{
\begin{tikzpicture}[baseline = 3pt, scale=0.5, color=\clr]
\draw[-,line width=1.2pt] (0,0) to[out=up, in=down] (0,1.4);
\draw(0,0.6) \bdot; 
\draw (0.7,0.6) node {$\scriptstyle \omega_r$};
\node at (0,-.3) {$\scriptstyle r$};
\end{tikzpicture}
}

\newcommand{\wkdotm}{\begin{tikzpicture}[baseline = 3pt, scale=0.5, color=\clr]
\draw[-,line width=1.2pt] (0,0) to[out=up, in=down] (0,1.4);
\draw(0,0.6) \bdot; 
\draw (0.7,0.6) node {$\scriptstyle \omega_r$};
\node at (0,-.3) {$\scriptstyle m$};
\end{tikzpicture} }

\newcommand{\wdotgen}{
\begin{tikzpicture}[baseline = 3pt, scale=0.5, color=\clr]
\draw[-,line width=1.2pt] (0,0) to[out=up, in=down] (0,1.4);
\draw(0,0.6) \bdot; 
\draw (0.7,0.6) node {$\scriptstyle \omega_1$};
\node at (0,-.3) {$\scriptstyle 1$};
\end{tikzpicture}}

\newcommand{\wxdota}{\begin{tikzpicture}[baseline = 3pt, scale=0.5, color=\clr]
\draw[-,line width=1.2pt] (0,0) to[out=up, in=down] (0,1.4);
\draw(0,0.6) \bdot; 
\draw (0.7,0.6) node {$\scriptstyle \omega_1$};
\node at (0,-.3) {$\scriptstyle a$};
\end{tikzpicture}}
\newcommand{\blue}[1]{{\color{blue}#1}}

\begin{document}
\setlength{\baselineskip}{17pt}
\title{Affine and cyclotomic webs}
%\date{\today}

\author{Linliang Song}
\address{School of Mathematical Science, Key Laboratory of Intelligent Computing and Applications(Ministry of Education),
Tongji University, Shanghai, 200092, China}\email{llsong@tongji.edu.cn}

 \author{Weiqiang Wang}
 \address{Department of Mathematics, University of Virginia,
Charlottesville, VA 22903, USA}\email{ww9c@virginia.edu}



\begin{abstract}
Generalizing the polynomial web category, we introduce a diagrammatic $\Bbbk$-linear monoidal category, {\em the affine web category}, for any commutative ring $\Bbbk$. Integral bases consisting of elementary diagrams are obtained for the affine web category and its cyclotomic quotient categories. Connections between cyclotomic web categories and finite $W$-algebras are established, leading to a diagrammatic presentation of idempotent subalgebras of $W$-Schur algebras introduced by Brundan-Kleshchev. The affine web category will be used as a basic building block of another $\Bbbk$-linear monoidal category, {\em the affine Schur category}, formulated in a sequel. 
\end{abstract}

\maketitle

\tableofcontents
	% 	 \setcounter{tocdepth}{1}

%=====================
% Main body
%=====================
%
\section{Introduction}

\subsection{A categorical Schur duality}

The classical Schur duality relates polynomial representations of $\mathfrak{gl}_n$ and representations of the symmetric group $\mathfrak S_m$. Often one can replace $\mathfrak{gl}_n$ by the corresponding Schur algebras. 

Throughout this paper, all categories and functors will be $\kk$-linear, where $\kk$ is a commutative ring with $1$. The (polynomial) web category $\W$ formulated in \cite{BEEO} is a $\mathfrak{gl}_\infty$-variant of the $\mathfrak{sl}_n$-web category introduced in \cite{CKM}; this variant is simpler and arguably more fundamental at least for our purpose. It is a strict monoidal category with generating objects denoted $\stra$ for $a\in \Z_{\ge 1}$ (that is, with strict compositions as objects) and generating morphisms $\merge, \splits,  \blue{\crossing}
$ subject to the relations \eqref{webassoc}--\eqref{equ:r=0}. 

The Schur category $\Sch$ is a monoidal category with the same objects as for $\W$ and the morphism spaces $\Hom_\Sch (\mu,\nu)$ being $\Hom_{\kk\mathfrak{S}_m} (M^\mu, M^\nu)$, where $M^\mu$ are permutation $\kk \mathfrak{S}_m$-modules.
It was established in \cite{BEEO} that $\W \cong \Sch$ as monoidal categories. This can be viewed as a categorical reformulation and upgrade of the $(\mathfrak{gl}_\infty, \mathfrak S_m)$-Schur duality. Let $\HH$ be the Hecke algebra category, which is the monoidal subcategory of $\W$ with only one generating object $\str$ (that is, with $m$ strands as objects for $m\in \N$); the endomorphism algebra of the $m$ strands in $\HH$ is the group algebra $\kk \mathfrak S_m$. 
%The Schur functor can be realized via the natural functor $\HH \rightarrow \W$. 

\subsection{Goal}

This is the first of a series of papers on constructions of new diagrammatic monoidal categories, called affine web categories and affine Schur categories, and their cyclotomic quotients. In this paper and its companion paper \cite{SW2} we shall deal with the degenerate type A constructions. The results and the approaches developed in these two papers will serve as blueprints for further generalizations and extensions; see \S\ref{subsec:future} for more details.

The goal of this paper is to introduce and study in depth a diagrammatic monoidal category, the {\em affine web category} $\AW$, and its cyclotomic quotient categories $\W_\bfu$, for $\bfu \in \kk^\ell$ and $\ell \ge 1$. We provide integral bases for both $\AW$ and $\W_\bfu$. For $\ell=1$, $\W_\bfu$ reduces to the web category $\W$. 
Recall the path algebra for a small $\kk$-linear category $\mathcal C$ is the locally unital algebra $\oplus_{a,b\in \text{ob}\mathcal C}\Hom_{\mathcal C}(a,b)$ with multiplication induced by the composition.
We then show that $\W_\bfu$ is isomorphic to $\WSch_\bfu$, the $W$-Schur category whose path algebras (for various full subcategories) are distinguished idempotent subalgebras of the $W$-Schur algebras arising from Brundan-Kleshchev's higher level Schur duality between finite $W$-algebras and cyclotomic Hecke algebras $\mathcal{H}_{m,\bfu}$; moreover, these subalgebras are Morita equivalent (over $\C$ and conjectually over any $\kk$) to the $W$-Schur algebras.
This is a notable generalization of the aforementioned isomorphism $\W \cong \Sch$.

\subsection{From $\W$ and $\AH$ to $\AW$}

It is natural to extend $\HH$ to the affine Hecke algebra category $\AH$, a strict monoidal category with the same generating object $\str$ as for $\HH$ and two generating morphisms $\cross, \dotgen$ subject to the slider relations \eqref{dotmovecrossingC}. Just as the algebra $\kk \mathfrak S_m$ is both a subalgebra and a quotient algebra of the degenerate affine Hecke algebra $\widehat{\mathcal H}_m$ (which is the endomorphism algebra of the $m$-strand object in $\AH$), $\HH$ is both a full subcategory and a quotient category of $\AH.$ 

Let $\Lambda_{\text{st}} =\cup_m \Lambda_{\text{st}}(m)$ be the set of all strict compositions. According to Definition~\ref{def-affine-web}, the affine web category $\AW$ has $\Lambda_{\text{st}}$ as the object set (the same as $\W$) and the generating morphisms $\merge, \splits, \crossing$ and 
\begin{align} \label{dota}
    \wkdotaa ,
\end{align} 
for $a,b \in \Z_{\ge 1}$, subject to the relations \eqref{webassoc}--\eqref{intergralballon}. 

These four categories fit well into the following commutative diagram of categories:
\begin{align}
    \label{diag:AHAW}
\xymatrix{
\HH \ar[r]   \ar[d] 
&\AH 
%\ar[r]  
%\ar@{=}[d]
%\ar@/^2.0pc/@[red][r]^{}
 \ar[d]  
  \\
\W \ar[r]  
& \AW~.    
}
\end{align}

For $\kk=\C$, one can replace the generating morphisms \eqref{dota} for $\AW$ by a single one 
$
 \begin{tikzpicture}[baseline = 3pt, scale=0.4, color=\clr]
\draw[-,thick] (0,0) to[out=up, in=down] (0,1.4);
\draw(0,0.6) \bdot;
\node at (0,-.3) {$\scriptstyle 1$};
\end{tikzpicture} \; ,
$ 
thanks to the identities in \eqref{newgendot} and discussions in Section~\ref{sec:C}. The definition of $\AW$ over $\C$ can be much simplified with all relations arising from $\W$ and $\AH$ through the diagram \eqref{diag:AHAW}; see Definition~\ref{def-affine-webC}.

\subsection{A basis theorem for $\AW$}

We first describe the simplest endomorphism algebras, i.e., those of a generating object $\stra$ in $\AW$. It is not a priori clear that the algebra $\End_{\AW}\big(\stra \big)$ is commutative. Define new morphisms $\wkdota$, for $1\le r \le a$, as in \eqref{equ:r=0andr=a}. 

\begin{customthm} {\bf A}  [Theorem~\ref{thm:Da}]
\label{th:A}
    Let $a \ge 1$. The endomorphism $\kk$-algebra   $\End_{\AW}\big(\stra \big)$ is a polynomial algebra in generators $\omega_{a,r} :=\wkdota$, for $1 \le r \le a$. 
\end{customthm}

Theorem~\ref{th:A} can be alternatively stated that $\End_{\AW}\big(\stra\big)$ has a multiplicative basis given by the elementary dot packets $\omega_{a,\nu} =\omega_{a,\nu_1} \omega_{a,\nu_2} \ldots$, parameterized by an arbitrary partition $\nu =(\nu_1, \nu_2,\ldots)$ with each part $\le a$. One can regard $\End_{\AW}\big(\stra\big)$ as the algebra of symmetric polynomials in $a$ variables, where $\omega_{a,\nu}$ has the elementary symmetric polynomial $e_\nu(x_1,\ldots,x_a)$ as its highest degree term. Our proof of Theorem~\ref{th:A} requires a basis result as given in Theorem~\ref{th:B} below. 

It is often one of the most challenging problems to establish a basis for each morphism space of a $\kk$-linear category defined via generators and relations, especially over $\kk=\Z$ (which gives rise to a basis over any commutative ring by base change). For the web category of \cite{CKM}, an integral basis (called double ladder basis) was obtained later by Elias \cite{El15}. 
It was shown \cite{BEEO} that $\Hom_{\W}(\mu,\lambda)$ admits a basis of reduced chicken foot diagrams (CFD) $\Mat_{\la,\mu}$, which can be naturally identified with non-negative integer matrices whose column (and row) sum vector is $\la$ (and $\mu$). In this identification, a matrix entry represents the thickness of a leg in a reduced CFD. 

Denote by $\Par_m$ the set of partitions $\nu$ with all parts  $\le m.$ An elementary CFD means a reduced CFD with each of its thin strands (or legs) of thickness $m$ decorated by an elementary dot packet $\omega_{m,\nu}$, for $\nu \in \Par_m$. 
The set of elementary CFDs from $\mu$ to $\la$ can be identified with  $\PMat_{\lambda,\mu}$ defined in \eqref{dottedreduced}, which stands for partition-enhanced matrices built on $\Mat_{\la,\mu}$.  

\begin{customthm} {\bf B}
[Theorem~\ref{basisAW}]
\label{th:B}
The Hom-space $\Hom_{\AW}(\mu,\lambda)$ has a basis $\PMat_{\lambda,\mu}$ which consists of all elementary CFDs from $\mu$ to $\la$, for any $\mu,\lambda\in \Lambda_{\text{st}}(m)$.  
\end{customthm}

To prove that $\PMat_{\lambda,\mu}$ forms a basis for $\Hom_{\AW}(\mu,\lambda)$, we verify that it is a spanning set and it is linearly independent. 
The spanning part of Theorem~\ref{th:B} is proved by showing that multiplications between any generating morphism and any elementary CFD is a linear combination of elementary CFD's. The proof of the linear independence part of Theorem~\ref{th:B} relies on representations of $\AW$ in the form of functors 
\[
\mathcal F: \AW \longrightarrow \mathfrak{gl}_N\text{-mod}
\]
for various $N$ (see Proposition~\ref{functorofaff}), extending the functors (see \cite{BEEO, CKM})
\begin{align}
 \label{eq:Phi}
 \Phi: \W \longrightarrow \mathfrak{gl}_N\text{-mod}.
\end{align}

\subsection{Cyclotomic web categories $\W_\bfu$}

The affine Hecke algebra $\widehat{\mathcal H}_m$ admits quotient algebras known as cyclotomic Hecke algebras $\mathcal H_{m,\bfu}$, for any $\bfu\in \kk^\ell$ and $\ell \ge 1$, which afford beautiful representation theory and connections to categorification; see \cite{Ar96}. This can be reformulated as a functor from $\AH$ to a cyclotomic Hecke algebra category $\HH_\bfu$. Accordingly, it is natural for us to introduce quotient categories of $\AW$ called cyclotomic web categories and denoted $\W_\bfu$ (see Definition~\ref{def:cycWeb}), which fit into the following commutative diagram (with all vertical arrows as inclusions of full subcategories): 
\[
\xymatrix{
\HH \ar[r]   \ar[d] 
&\AH \ar[r]  
%\ar@{=}[d]
%\ar@/^2.0pc/@[red][r]^{}
 \ar[d]  & \HH_{\bfu}
 \ar[d]  
  \\
\W \ar[r]  
& \AW \ar[r]  
& \W_\bfu
}
\]

We shall describe an integral basis for any cyclotomic web category. 
An $(\ell-1)$-restricted elementary CFD means a reduced CFD with each of its legs of thickness $m$ decorated with an elementary dot packet $\omega_{\nu}$ for $\nu \in \Par_m$ such that $l(\nu)\le \ell-1$. 
%with the Young diagram $\nu$ in the $r\times m$-rectangle. 
The set of level $\ell$ elementary CFDs from $\mu$ to $\la$ can be identified with $\PMat_{\lambda,\mu}^\ell$ defined in \eqref{Def-spansetof-cycweb}, the set of $\ell$-bounded partition-enhanced matrices. 

\begin{customthm} {\bf C}   [Theorem~\ref{basis-cyc-web}]
\label{th:C} 
Suppose that ${\mathbf u} =(u_1, \ldots, u_\ell) \in\kk^\ell$. The Hom-space $\Hom_{\W_{\mathbf u}}(\mu,\lambda)$ has a basis given by $\PMat_{\lambda,\mu}^\ell$, for any $\lambda,\mu \in\Lambda_{\text{st}}(m)$.   
\end{customthm}

It is rather straightforward to prove that $\PMat_{\lambda,\mu}^\ell$ forms a spanning set for the Hom-space $\Hom_{\W_{\mathbf u}}(\mu,\lambda)$. The linear independence will follow most easily from the corresponding result for the cyclotomic Schur category $\Sch_\bfu$ in \cite{SW2} by viewing $\W_\bfu$ as a subcategory of $\Sch_\bfu$. 

For $\ell=1$, we have $\PMat_{\lambda,\mu}^\ell =\Mat_{\lambda,\mu}$, and Theorem~\ref{th:C} reduces to a basis theorem for $\W$ in \cite{BEEO}. Bearing in mind the isomorphism $\W \cong \Sch$, one hopes to identify explicitly the path algebras in $\W_\bfu$. This turns out to have a beautiful (and unexpected to us initially) answer. 

\subsection{The $W$-Schur categories}

Let $\Lambda=(\lambda_1\le \ldots\le \lambda_n)$ be a strict composition of $N$ such that $l(\Lambda)=n$ and  $\lambda_n=\ell$. Let $e\in \mathfrak{gl}_N$ be a nilpotent matrix of Jordan form $\Lambda$. The finite $W$-algebra $W(\Lambda)$ associated to $e$ can be viewed as a subalgebra of $U(\mathfrak p)$ (with a twisting by a fixed vector $\bfc \in \kk^\ell$) for a certain parabolic subalgebra $\mathfrak p$ of $\mathfrak{gl}_N$ and $V$ is the natural $\mathfrak{gl}_N$-module. Generalizing the classic Schur duality (which corresponds to the case for $\ell=1$), Brundan and Kleshchev \cite{BK08} established a $(W(\Lambda), \mathcal{H}_{m,\bfu})$-duality on $V_{\bfc}^{\otimes m}$, where $\bfu$ is specified by $\Lambda$ and $\bfc$ as in \eqref{BKui}. Define the $W$-Schur algebra $W_{m,\bfu}(\Lambda)$ as in \eqref{WSalg}, which reduces to the usual Schur algebra when $\ell=1$ (and hence, $e=0$ and $\Lambda=(1^n)$).

We extend the functor $\Phi$ in \eqref{eq:Phi} to a functor $\hat\Phi: \AW \rightarrow \kk \text{-mod}^\wedge$ and show in Proposition \ref{prop:Phihat} that it factors through the cyclotomic web category $\W_\bfu$ and gives rise to a functor 
\begin{align}  \label{PhiWW}
\hat\Phi : \W_\bfu \longrightarrow W(\Lambda)\text{-mod}. 
\end{align}
We show that the functor $\hat\Phi$ in \eqref{PhiWW} is {\em asymptotically faithful}; see Theorem~\ref{thm:sameHom} for a precise statement. For $\ell=1$ this was established in \cite{CKM} (and also \cite[Remark~4.15]{BEEO}). 

Let $\mathfrak {Web}_{m,\bfu}$, for $m\ge 1$,  be the path algebras in the cyclotomic web category $\W_\bfu$; see \eqref{algWebu}. Given $\Lambda=(\lambda_1,\ldots, \lambda_n)$ with  $\lambda_n=\ell$, we denote by $\Lambda^a$ (for $a\in \N \cup \{\infty\}$) by adding $a$ parts equal to $\ell$ to the end of $\Lambda$. Note the $\bfu$ in \eqref{BKui} attached to $\Lambda^{a}$ is independent of $a$. Associated to the $W$-Schur algebra $W_{m,\bfu}(\Lambda^\infty)$ we form a $W$-Schur category $\WSch_\bfu$ with $\Lambda_{\text{st}}$ as the object set and Hom-spaces as in \eqref{HomWSu}. 

\begin{customthm} {\bf D} [Theorem \ref{thm:CycWeb=WSchur}]
\label{th:D} 
 The $W$-Schur category $\WSch_\bfu$ is isomorphic to $\W_{-\bfu}$. Moreover,  for $\kk=\C$ and for any $m$, the category $W_{m,\bfu}(\Lambda^\infty)$-mod is equivalent to $\mathfrak {Web}_{m,-\bfu}$-mod.
\end{customthm}
We expect that the assumption $\kk=\C$ in Theorem~\ref{th:D} can be removed.
Theorem~\ref{th:D} specializes at $\ell=1$ to the isomorphism of categories $\W \cong \Sch$ in \cite{BEEO}, which can be viewed as a categorical formulation of the Schur duality.

\subsection{Future works}
 \label{subsec:future}

We have initiated a research program in this paper on affine and cyclotomic web categories (and affine and cyclotomic Schur categories in a sequel \cite{SW2}). A simple but key observation is that constructions of variants of affine web categories will follow variants of the diagram \eqref{diag:AHAW} (with variants of affine Hecke algebra category and webs as input) first over $\C$, and then over $\Z$ and $\kk$. 

The web categories in \cite{CKM, BEEO} and affine/cyclotomic web categories studied in this paper are of type $A$ and degenerate. The $q$-version of web categories of type $A$ appeared in \cite{CKM}. The constructions of web categories have been extended to other types: types $Sp/O$ in \cite{ST19} and \cite{BERT21, BT24, BWu23}, type $Q$ in \cite{BKu21}, and type $P$ in \cite{DKM21}. None of the papers in the literature addressed the constructions of affine web categories of any type. However, we were informed that the affine web categories of (degenerate) types $A$ and $Q$ are defined and studied in a unified framework independently in a forthcoming paper by Davidson, Kujawa, Muth and Zhu \cite{DKMZ}. 

More explicitly, our constructions in this paper and \cite{SW2} will be extended to the degenerate types $Sp/O, Q$, $P$ with affine Brauer category and its $Q, P$-variants as additional inputs. The constructions will also be extended to the quantum setting at least for type $A$ and $Sp/O$. Affine/cyclotomic {\em rational} web categories and beyond can also be formulated with oriented Brauer category (or Heisenberg category) as another input.

The affine web category $\AW$ will serve as a key building block in our construction of a new monoidal category called affine Schur category in a sequel \cite{SW2}, where the path algebras of the associated cyclotomic Schur categories are shown to be isomorphic to Dipper-James-Mathas' (degenerate) cyclotomic Schur algebras \cite{DJM98}. 

\subsection{Organization} 
The paper is organized as follows. In Section~\ref{sec:AWeb}, we formulate the affine web category $\AW$ and establish many additional relations which are to be used in later sections. The basis theorem for $\AW$ (Theorems~\ref{th:A} and \ref{th:B}) is established in Section~\ref{sec:AWbasis}.

In Section~\ref{sec:WSchur}, we introduce the cyclotomic web category $\W_\bfu$, formulate its basis theorem (Theorem~\ref{th:C}), and develop its connections to finite $W$-algebra $W(\Lambda)$. We show that $\W_\bfu$ is isomorphic to the $W$-Schur category $\WSch_\bfu$ (Theorem~\ref{th:D}). 

In Section~\ref{sec:C}, for $\kk$ a field of characteristic 0, we give simplified presentations for the affine web category $\AW$ and the cyclotomic web categories $\W_\bfu$. 

\vspace{2mm}

\noindent {\bf Acknowledgement.} 
We thank Robert Muth for helpful explanation and discussion of their ongoing work \cite{DKMZ} during the UGA conference in May 2024. 
LS is partially supported by NSFC (Grant No. 12071346), and he thanks Institute of Mathematical Science and Department of Mathematics at University of Virginia for hospitality and support. WW is partially supported by DMS--2401351. 
\section{The affine web category}
  \label{sec:AWeb}

In this section we introduce a diagrammatic monoidal category, the affine web category.

\subsection{Definition of affine web category} 

Throughout this paper, let $\kk$ be a commutative ring with $1$. All categories and functors will be $\kk$-linear without further mention.

\begin{definition}
\label{def-affine-web}
The affine web category $\AW$ is the strict monoidal category generated by objects $a\in \mathbb Z_{\ge1}$. The object $a$ and its identity morphism  will be drawn as a vertical strand labeled by $a$:
\[
\begin{tikzpicture}[baseline = 10pt, scale=0.4, color=\clr]
            \draw[-,line width=1.5pt] (0,0.5)to[out=up,in=down](0,2.2);
            \draw (0,0.1) node{$ a$};
\end{tikzpicture} \; .
\]
The generating morphisms are the merges, splits and (thick) crossings depicted as
\begin{align}
\label{merge+split+crossing}
\begin{tikzpicture}[baseline = -.5mm,color=\clr]
	\draw[-,line width=1pt] (0.28,-.3) to (0.08,0.04);
	\draw[-,line width=1pt] (-0.12,-.3) to (0.08,0.04);
	\draw[-,line width=1.5pt] (0.08,.4) to (0.08,0);
        \node at (-0.22,-.4) {$\scriptstyle a$};
        \node at (0.35,-.4) {$\scriptstyle b$};\node at (0,.55){$\scriptstyle a+b$};\end{tikzpicture} 
&:(a,b) \rightarrow (a+b),&
\begin{tikzpicture}[baseline = -.5mm,color=\clr]
	\draw[-,line width=1.5pt] (0.08,-.3) to (0.08,0.04);
	\draw[-,line width=1pt] (0.28,.4) to (0.08,0);
	\draw[-,line width=1pt] (-0.12,.4) to (0.08,0);
        \node at (-0.22,.5) {$\scriptstyle a$};
        \node at (0.36,.5) {$\scriptstyle b$};
        \node at (0.1,-.45){$\scriptstyle a+b$};
\end{tikzpicture}
&:(a+b)\rightarrow (a,b),&
\begin{tikzpicture}[baseline=-.5mm,color=\clr]
	\draw[-,line width=1pt] (-0.3,-.3) to (.3,.4);
	\draw[-,line width=1pt] (0.3,-.3) to (-.3,.4);
        \node at (0.3,-.45) {$\scriptstyle b$};
        \node at (-0.3,-.45) {$\scriptstyle a$};
         \node at (0.3,.55) {$\scriptstyle a$};
        \node at (-0.3,.55) {$\scriptstyle b$};
\end{tikzpicture}
&:(a,b) \rightarrow (b,a),
\end{align}
and 
\begin{equation}
\label{dotgenerator}
\begin{tikzpicture}[baseline = 3pt, scale=0.5, color=\clr]
\draw[-,line width=1.5pt] (0,0) to[out=up, in=down] (0,1.4);
\draw(0,0.6) \bdot; 
\draw (0.7,0.6) node {$\scriptstyle \omega_a$};
\node at (0,-.3) {$\scriptstyle a$};
\end{tikzpicture}  
\;,  
\end{equation}
for $a,b \in \Z_{\ge 1}$,
subject to the following relations \eqref{webassoc}--\eqref{intergralballon}, for $a,b,c,d \in \Z_{\ge 1}$ with $d-a=c-b$:
\begin{align}
\label{webassoc}
\begin{tikzpicture}[baseline = 0,color=\clr]
	\draw[-,thick] (0.35,-.3) to (0.08,0.14);
	\draw[-,thick] (0.1,-.3) to (-0.04,-0.06);
	\draw[-,line width=1pt] (0.085,.14) to (-0.035,-0.06);
	\draw[-,thick] (-0.2,-.3) to (0.07,0.14);
	\draw[-,line width=1.5pt] (0.08,.45) to (0.08,.1);
        \node at (0.45,-.41) {$\scriptstyle c$};
        \node at (0.07,-.4) {$\scriptstyle b$};
        \node at (-0.28,-.41) {$\scriptstyle a$};
\end{tikzpicture}
&=
\begin{tikzpicture}[baseline = 0, color=\clr]
	\draw[-,thick] (0.36,-.3) to (0.09,0.14);
	\draw[-,thick] (0.06,-.3) to (0.2,-.05);
	\draw[-,line width=1pt] (0.07,.14) to (0.19,-.06);
	\draw[-,thick] (-0.19,-.3) to (0.08,0.14);
	\draw[-,line width=1.5pt] (0.08,.45) to (0.08,.1);
        \node at (0.45,-.41) {$\scriptstyle c$};
        \node at (0.07,-.4) {$\scriptstyle b$};
        \node at (-0.28,-.41) {$\scriptstyle a$};
\end{tikzpicture}\:,
\qquad
\begin{tikzpicture}[baseline = -1mm, color=\clr]
	\draw[-,thick] (0.35,.3) to (0.08,-0.14);
	\draw[-,thick] (0.1,.3) to (-0.04,0.06);
	\draw[-,line width=1pt] (0.085,-.14) to (-0.035,0.06);
	\draw[-,thick] (-0.2,.3) to (0.07,-0.14);
	\draw[-,line width=1.5pt] (0.08,-.45) to (0.08,-.1);
        \node at (0.45,.4) {$\scriptstyle c$};
        \node at (0.07,.42) {$\scriptstyle b$};
        \node at (-0.28,.4) {$\scriptstyle a$};
\end{tikzpicture}
=\begin{tikzpicture}[baseline = -1mm, color=\clr]
	\draw[-,thick] (0.36,.3) to (0.09,-0.14);
	\draw[-,thick] (0.06,.3) to (0.2,.05);
	\draw[-,line width=1pt] (0.07,-.14) to (0.19,.06);
	\draw[-,thick] (-0.19,.3) to (0.08,-0.14);
	\draw[-,line width=1.5pt] (0.08,-.45) to (0.08,-.1);
        \node at (0.45,.4) {$\scriptstyle c$};
        \node at (0.07,.42) {$\scriptstyle b$};
        \node at (-0.28,.4) {$\scriptstyle a$};
\end{tikzpicture}\:,
\\
\label{mergesplit}
\begin{tikzpicture}[baseline = 7.5pt,scale=.8, color=\clr]
	\draw[-,line width=1.2pt] (0,0) to (.275,.3) to (.275,.7) to (0,1);
	\draw[-,line width=1.2pt] (.6,0) to (.315,.3) to (.315,.7) to (.6,1);
        \node at (0,1.13) {$\scriptstyle b$};
        \node at (0.63,1.13) {$\scriptstyle d$};
        \node at (0,-.1) {$\scriptstyle a$};
        \node at (0.63,-.1) {$\scriptstyle c$};
\end{tikzpicture}
&=
\sum_{\substack{0 \leq s \leq \min(a,b)\\0 \leq t \leq \min(c,d)\\t-s=d-a}}
\begin{tikzpicture}[baseline = 7.5pt,scale=.8, color=\clr]
	\draw[-,thick] (0.58,0) to (0.58,.2) to (.02,.8) to (.02,1);
	\draw[-,thick] (0.02,0) to (0.02,.2) to (.58,.8) to (.58,1);
	\draw[-,thick] (0,0) to (0,1);
	\draw[-,line width=1pt] (0.61,0) to (0.61,1);
        \node at (0,1.13) {$\scriptstyle b$};
        \node at (0.6,1.13) {$\scriptstyle d$};
        \node at (0,-.1) {$\scriptstyle a$};
        \node at (0.6,-.1) {$\scriptstyle c$};
        \node at (-0.1,.5) {$\scriptstyle s$};
        \node at (0.77,.5) {$\scriptstyle t$};
\end{tikzpicture},
\\
  \label{equ:r=0}
\begin{tikzpicture}[baseline = -1mm,scale=.8,color=\clr]
\draw[-,line width=1.5pt] (0.08,-.8) to (0.08,-.5);
\draw[-,line width=1.5pt] (0.08,.3) to (0.08,.6);
\draw[-,thick] (0.1,-.51) to [out=45,in=-45] (0.1,.31);
\draw[-,thick] (0.06,-.51) to [out=135,in=-135] (0.06,.31);
\node at (-.33,-.05) {$\scriptstyle a$};
\node at (.45,-.05) {$\scriptstyle b$};
\end{tikzpicture}
&= 
\binom{a+b}{a}\:
\begin{tikzpicture}[baseline = -1mm,scale=.7,color=\clr]
	\draw[-,line width=1.5pt] (0.08,-.8) to (0.08,.6);
        \node at (.08,-1) {$\scriptstyle a+b$};
\end{tikzpicture}, 
\\
\label{equ:dotmovecrossing-simplify}
\begin{tikzpicture}[baseline = 7.5pt, scale=0.4, color=\clr]
\draw[-,line width=1.2pt] (0,-.2) to  (1,2.2);
\draw[-,line width=1.2pt] (1,-.2) to  (0,2.2);
\draw(0.2,1.6)\bdot;
\draw(-.4,1.6) node {$\scriptstyle \omega_b$};
\node at (0, -.5) {$\scriptstyle a$};
\node at (1, -.5) {$\scriptstyle b$};
\end{tikzpicture}
&= 
 \sum_{0 \leq t \leq \min(a,b)}t!~
\begin{tikzpicture}[baseline = 7.5pt,scale=.8, color=\clr]
\draw[-,thick] (0.58,-.2) to (0.58,0) to (-.48,.8) to (-.48,1);
\draw[-,thick] (-.48,-.2) to (-.48,0) to (.58,.8) to (.58,1);
\draw[-,thick] (-.5,-.2) to (-.5,1);
\draw[-,thick] (0.605,-.2) to (0.605,1);
\draw(.35,0.2)\bdot;
\draw(.1,0) node {$\scriptstyle {\omega_{b-t}}$};
%\node at (0,1.13) {$\scriptstyle b$};
%\node at (0.6,1.13) {$\scriptstyle d$};
\node at (-.5,-.3) {$\scriptstyle a$};
\node at (0.6,-.3) {$\scriptstyle b$};
%\node at (-0.1,.5) {$\scriptstyle s$};
\node at (0.77,.5) {$\scriptstyle t$};
\end{tikzpicture},
\qquad
\begin{tikzpicture}[baseline = 7.5pt, scale=0.4, color=\clr]
\draw[-, line width=1.2pt] (0,-.2) to (1,2.2);
\draw[-,line width=1.2pt] (1,-.2) to(0,2.2);
\draw(.2,0.2)\bdot;
\draw(-.4,.2)node {$\scriptstyle \omega_b$};
\node at (0, -.5) {$\scriptstyle b$};
\node at (1, -.5) {$\scriptstyle a$};
\end{tikzpicture}
=  
 \sum_{0 \leq t \leq \min(a,b)}t!~
\begin{tikzpicture}[baseline = 7.5pt,scale=.8, color=\clr]
\draw[-,thick] (0.58,-.2) to (0.58,0) to (-.48,.8) to (-.48,1);
\draw[-,thick] (-.48,-.2) to (-.48,0) to (.58,.8) to (.58,1);
\draw[-,thick] (-.5,-.2) to (-.5,1);
\draw[-,thick] (0.605,-.2) to (0.605,1);
\draw(.35,0.6)\bdot;
\draw(.1,0.8) node {$\scriptstyle {\omega_{b-t}}$};
%\node at (0,1.13) {$\scriptstyle b$};
%\node at (0.6,1.13) {$\scriptstyle d$};
\node at (-.5,-.3) {$\scriptstyle b$};
\node at (0.6,-.3) {$\scriptstyle a$};
%\node at (-0.1,.5) {$\scriptstyle s$};
\node at (0.77,.5) {$\scriptstyle t$};
\end{tikzpicture},
\\
 \label{dotmovesplits+merge-simplify}
\begin{tikzpicture}[baseline = -.5mm,scale=.8,color=\clr]
\draw[-,line width=1.5pt] (0.08,-.5) to (0.08,0.04);
\draw[-,line width=1pt] (0.34,.5) to (0.08,0);
\draw[-,line width=1pt] (-0.2,.5) to (0.08,0);
\node at (-0.22,.6) {$\scriptstyle a$};
\node at (0.36,.65) {$\scriptstyle b$};
\draw (0.08,-.2) \bdot;
\draw (0.6,-.2) node{$\scriptstyle \omega_{a+b}$};
\end{tikzpicture} 
&=
\begin{tikzpicture}[baseline = -.5mm,scale=.8,color=\clr]
\draw[-,line width=1.5pt] (0.08,-.5) to (0.08,0.04);
\draw[-,line width=1pt] (0.34,.5) to (0.08,0);
\draw[-,line width=1pt] (-0.2,.5) to (0.08,0);
\node at (-0.22,.6) {$\scriptstyle a$};
\node at (0.36,.65) {$\scriptstyle b$};
\draw (-.05,.24) \bdot;
\draw (-0.35,.2) node{$\scriptstyle \omega_{a}$};
\draw (0.6,.2) node{$\scriptstyle \omega_{b}$};
\draw (.22,.24) \bdot;
\end{tikzpicture}, \quad 
\begin{tikzpicture}[baseline = -.5mm, scale=.8, color=\clr]
\draw[-,line width=1pt] (0.3,-.5) to (0.08,0.04);
\draw[-,line width=1pt] (-0.2,-.5) to (0.08,0.04);
\draw[-,line width=1.5pt] (0.08,.6) to (0.08,0);
\node at (-0.22,-.6) {$\scriptstyle a$};
\node at (0.35,-.6) {$\scriptstyle b$};
\draw (0.08,.2) \bdot;
\draw (0.6,.2) node{$\scriptstyle \omega_{a+b}$};  
\end{tikzpicture}
 =~
\begin{tikzpicture}[baseline = -.5mm,scale=.8, color=\clr]
\draw[-,line width=1pt] (0.3,-.5) to (0.08,0.04);
\draw[-,line width=1pt] (-0.2,-.5) to (0.08,0.04);
\draw[-,line width=1.5pt] (0.08,.6) to (0.08,0);
\node at (-0.22,-.6) {$\scriptstyle a$};
\node at (0.35,-.6) {$\scriptstyle b$};
\draw (-.08,-.3) \bdot; \draw (.22,-.3) \bdot;
\draw (-0.41,-.3) node{$\scriptstyle \omega_{a}$};
\draw (0.6,-.3) node{$\scriptstyle \omega_{b}$};
\end{tikzpicture}.
\\
\label{intergralballon}
\begin{tikzpicture}[baseline = 1.5mm, scale=.5, color=\clr]
\draw[-, line width=1.5pt] (0.5,2) to (0.5,2.5);
\draw[-, line width=1.5pt] (0.5,0) to (0.5,-.4);
\draw[-,thin]  (0.5,2) to[out=left,in=up] (-.5,1)
to[out=down,in=left] (0.5,0);
\draw[-,thin]  (0.5,2) to[out=left,in=up] (0,1)
 to[out=down,in=left] (0.5,0);      
\draw[-,thin] (0.5,0)to[out=right,in=down] (1.5,1)
to[out=up,in=right] (0.5,2);
\draw[-,thin] (0.5,0)to[out=right,in=down] (1,1) to[out=up,in=right] (0.5,2);
\node at (0.5,.7){$\scriptstyle\cdots$};
\draw (-0.5,1) \bdot; 
%\node at (-.8,1) {$\scriptstyle \omega_1$}; 
\draw (0,1) \bdot; 
%\node at (0.3,1) {$\scriptstyle \omega_1$};
\node at (0.5,-.6) {$\scriptstyle a$};
\draw (1,1) \bdot;
\draw (1.5,1) \bdot; 
%\node at (1.8,1) {$\scriptstyle \omega_1$};
\node at (-.22,0) {$\scriptstyle 1$};
\node at (1.2,0) {$\scriptstyle 1$};
\node at (.3,0.3) {$\scriptstyle 1$};
\node at (.7,0.3) {$\scriptstyle 1$};
\end{tikzpicture}
&=~a!~
\begin{tikzpicture}[baseline = 1.5mm, scale=.7, color=\clr]
\draw[-,line width=1.5pt] (0,-0.1) to[out=up, in=down] (0,1.4);
\draw(0,0.6) \bdot; 
\draw (0.4,0.6) node {$\scriptstyle \omega_a$};
\node at (0,-.3) {$\scriptstyle a$};
\end{tikzpicture} ,
\end{align}
where we have adopted the simplified notation  $\dotgen :=\wdotgen$. 
\end{definition}
We define the following morphisms in $\AW$
\begin{equation}
\label{equ:r=0andr=a}
\begin{tikzpicture}[baseline = -1mm,scale=.6,color=\clr]
\draw[-,line width=1.5pt] (0.08,-.7) to (0.08,.5);
\node at (.08,-.9) {$\scriptstyle a$};
\draw(0.08,0) \bdot;
\draw(.5,0)node {$\scriptstyle \omega_r$};
\end{tikzpicture}
:=
\begin{tikzpicture}[baseline = -1mm,scale=.8,color=\clr]
\draw[-,line width=1.5pt] (0.08,-.8) to (0.08,-.5);
\draw[-,line width=1.5pt] (0.08,.3) to (0.08,.6);
\draw[-,line width=1pt] (0.1,-.51) to [out=45,in=-45] (0.1,.31);
\draw[-,line width=1pt] (0.06,-.51) to [out=135,in=-135] (0.06,.31);
\draw(-.1,0) \bdot;
\draw(-.4,0)node {$\scriptstyle \omega_r$};
\node at (-.2,-.35) {$\scriptstyle r$};
\node at (.08,-.9) {$\scriptstyle a$};
%\node at (.45,-.35) {$\scriptstyle b$};
\end{tikzpicture} \qquad\quad (1\le r\le a).
\end{equation}
We adopt the convention throughout the paper that
\[ \begin{tikzpicture}[baseline = 10pt, scale=0.5, color=\clr]
            \draw[-,line width=1.5pt] (0,0.5)to[out=up,in=down](0,1.9);
            %\draw (0.3,1.1) node{$\scriptstyle t$};
\node at (0,.2) {$\scriptstyle a$};
\end{tikzpicture}=0 
\text{ unless  } a \ge 0, 
\quad \text{ }
\begin{tikzpicture}[baseline = 3pt, scale=0.5, color=\clr]
\draw[-,line width=1.5pt] (0,0) to[out=up, in=down] (0,1.4);
\draw(0,0.6) \bdot; 
\draw (0.65,0.6) node {$\scriptstyle \omega_r$};
\node at (0,-.3) {$\scriptstyle a$};
\end{tikzpicture}=
0  \text{ unless } 0\le r\le a, \quad \text{ and } 
\begin{tikzpicture}[baseline = 3pt, scale=0.5, color=\clr]
\draw[-,line width=1.5pt] (0,0) to[out=up, in=down] (0,1.4);
\draw(0,0.6) \bdot; 
\draw (0.65,0.6) node {$\scriptstyle \omega_0$};
\node at (0,-.3) {$\scriptstyle a$};
\end{tikzpicture} 
=
\begin{tikzpicture}[baseline = 10pt, scale=0.5, color=\clr]
            \draw[-,line width=1.5pt] (0,0.5)to[out=up,in=down](0,1.9);
            %\draw (0.3,1.1) node{$\scriptstyle t$};
\node at (0,.2) {$\scriptstyle a$};
\end{tikzpicture}.
 \]
These conventions were already implicitly used in \eqref{equ:dotmovecrossing-simplify}.
 
Note that the $a$-fold merges and splits in \eqref{intergralballon} are defined by using \eqref{webassoc} via compositions of 2-fold merges and splits. For example, the following are 3-fold merge and split.
\begin{align}
\begin{tikzpicture}[baseline = 0, color=\clr]
	\draw[-,thick] (0.35,-.3) to (0.09,0.14);
	\draw[-,thick] (0.08,-.3) to (.08,0.1);
	\draw[-,thick] (-0.2,-.3) to (0.07,0.14);
	\draw[-,line width=1.5pt] (0.08,.45) to (0.08,.1);
        \node at (0.45,-.41) {$\scriptstyle c$};
        \node at (0.07,-.4) {$\scriptstyle b$};
        \node at (-0.28,-.41) {$\scriptstyle a$};
\end{tikzpicture}
:=
\begin{tikzpicture}[baseline = 0, color=\clr]
	\draw[-,thick] (0.35,-.3) to (0.08,0.14);
	\draw[-,thick] (0.1,-.3) to (-0.04,-0.06);
	\draw[-,line width=1pt] (0.085,.14) to (-0.035,-0.06);
	\draw[-,thick] (-0.2,-.3) to (0.07,0.14);
	\draw[-,line width=1.5pt] (0.08,.45) to (0.08,.1);
        \node at (0.45,-.41) {$\scriptstyle c$};
        \node at (0.07,-.4) {$\scriptstyle b$};
        \node at (-0.28,-.41) {$\scriptstyle a$};
\end{tikzpicture}
&=
\begin{tikzpicture}[baseline = 0, color=\clr]
	\draw[-,thick] (0.36,-.3) to (0.09,0.14);
	\draw[-,thick] (0.06,-.3) to (0.2,-.05);
	\draw[-,line width=1pt] (0.07,.14) to (0.19,-.06);
	\draw[-,thick] (-0.19,-.3) to (0.08,0.14);
	\draw[-,line width=1.5pt] (0.08,.45) to (0.08,.1);
        \node at (0.45,-.41) {$\scriptstyle c$};
        \node at (0.07,-.4) {$\scriptstyle b$};
        \node at (-0.28,-.41) {$\scriptstyle a$};
\end{tikzpicture}\:,&
\begin{tikzpicture}[baseline = 0, color=\clr]
	\draw[-,thick] (0.35,.3) to (0.09,-0.14);
	\draw[-,thick] (0.08,.3) to (.08,-0.1);
	\draw[-,thick] (-0.2,.3) to (0.07,-0.14);
	\draw[-,line width=1.5pt] (0.08,-.45) to (0.08,-.1);
        \node at (0.45,.41) {$\scriptstyle c$};
        \node at (0.07,.43) {$\scriptstyle b$};
        \node at (-0.28,.41) {$\scriptstyle a$};
\end{tikzpicture}
:=
\begin{tikzpicture}[baseline = 0, color=\clr]
	\draw[-,thick] (0.35,.3) to (0.08,-0.14);
	\draw[-,thick] (0.1,.3) to (-0.04,0.06);
	\draw[-,line width=1pt] (0.085,-.14) to (-0.035,0.06);
	\draw[-,thick] (-0.2,.3) to (0.07,-0.14);
	\draw[-,line width=1.5pt] (0.08,-.45) to (0.08,-.1);
        \node at (0.45,.41) {$\scriptstyle c$};
        \node at (0.07,.43) {$\scriptstyle b$};
        \node at (-0.28,.41) {$\scriptstyle a$};
\end{tikzpicture}
&=
\begin{tikzpicture}[baseline = 0, color=\clr]
	\draw[-,thick] (0.36,.3) to (0.09,-0.14);
	\draw[-,thick] (0.06,.3) to (0.2,.05);
	\draw[-,line width=1pt] (0.07,-.14) to (0.19,.06);
	\draw[-,thick] (-0.19,.3) to (0.08,-0.14);
	\draw[-,line width=1.5pt] (0.08,-.45) to (0.08,-.1);
        \node at (0.45,.41) {$\scriptstyle c$};
        \node at (0.07,.43) {$\scriptstyle b$};
        \node at (-0.28,.41) {$\scriptstyle a$};
\end{tikzpicture}\:.
\end{align}    
 
\begin{rem}
The generators and defining relations for $\AW$ over a field $\kk$ of characteristic zero can be much simplified; see \S \ref{affineweboverC}.  
\end{rem}

\begin{rem}
\label{rem:finiteaffineweb}
%We may define a finite $n$ version of the affine web category. 
For any $n\in \N$, let $\AW_n$ be the quotient of $ \AW$ by the tensor ideal generated by all $1_m$ with $m> n$. It is a monoidal category with the same generating objects and morphisms as $\AW$ with the additional relations 
\[
1_m=0, \text{ for all } m>n.
\]
\end{rem}
For any $ m\in \mathbb N$, a composition $\mu=(\mu_1,\mu_2,\ldots,\mu_n)$ of $m$ is a sequence of non-negative integers such that $\sum_{i\ge 1} \mu_i=m$. The length $l(\mu)$ of $\mu$ is the total number $n$ of parts. Let $\Lambda(m)$ denote the set of all compositions of $m$, and $\Lambda=\cup_m \Lambda(m)$ denote the set of all compositions.
Let $\Par(m)$ denote the set of all partitions of $m$ and $\Par=\cup_m \Par(m)$ denote the set of all partitions. A composition $\mu$ is called strict if  $\mu_i>0$ for any $i\ge 1$.  
We denote by  $\Lambda_{\text{st}}$ (and respectively, $\Lambda_{\text{st}}(m)$) the set of all strict compositions (and respectively, of $m$). 

The set of objects in $\AW$ is $\Lambda_{\text{st}}$ with the empty sequence being the unit object. 
Note that we may omit the labels on stands when it is clear from the context. For example, the label for the top object of the first relation in  \eqref{webassoc} is $ a+b +c$. We also can label the unit object by $0$ but we may often omit this stand.

Any morphism in $\AW$ is a linear combination of diagrams obtained from the tensor product and composition of generators (i.e., splits, merges, crossing, and dotted vertical strands in \eqref{dotgenerator}) together with identity morphism. We call such a diagram a dotted web diagram. In particular, for any 
$\lambda,\mu \in \Lambda_{\text{st}}(m)$, the morphism space 
$\Hom_{\AW}(\lambda,\mu)$ is spanned by all dotted web diagrams with the bottom (and respectively, top) thickness labeled by $\lambda$ (and respectively, $\mu$). In this case, we call such a dotted diagram is of type $\lambda\rightarrow \mu$.
Moreover, we have $\Hom_{\AW}(\lambda,\mu)=0$ if $\lambda\in \Lambda_{\text{st}}(m)$ and $\mu \in\Lambda_{\text{st}}(m')$ with $m\ne m'$.

\subsection{The web category}

 We define the (polynomial) web category $\W$ as the monoidal category with generating objects $a\in \Z_{\ge 1}$ and generating morphisms in \eqref{merge+split+crossing} subject to relations \eqref{webassoc}, \eqref{mergesplit} and \eqref{equ:r=0}. There exists a natural functor from $\W$ to $\AW$ (which will be shown to be faithful). 
%For $\AW$, the relation \eqref{splitmerge} with $r=0$ reads
% \begin{align*}
% \begin{tikzpicture}[baseline = -1mm,scale=.8,color=\clr]
% 	\draw[-,line width=1.5pt] (0.08,-.8) to (0.08,-.5);
% 	\draw[-,line width=1.5pt] (0.08,.3) to (0.08,.6);
% \draw[-,thick] (0.1,-.51) to [out=45,in=-45] (0.1,.31);
% \draw[-,thick] (0.06,-.51) to [out=135,in=-135] (0.06,.31);
%         \node at (-.33,-.05) {$\scriptstyle a$};
%         \node at (.45,-.05) {$\scriptstyle b$};
% \end{tikzpicture}
% &= 
% \binom{a+b}{a}\:
% \begin{tikzpicture}[baseline = -1mm,scale=.8,color=\clr]
% 	\draw[-,line width=1.5pt] (0.08,-.8) to (0.08,.6);
%         \node at (.08,-1) {$\scriptstyle a+b$};
% \end{tikzpicture}.
% \end{align*}
The web category $\W$ appeared in \cite[Definition~ 4.7]{BEEO}, which is a $\mathfrak{gl}_\infty$-version of the original $\mathfrak{sl}_n$-web category introduced in \cite{CKM}. 

We list some well-known implied relations in $\W$ (see \cite{CKM} and \cite[(4.30)--(4.33)]{BEEO}), and hence also in $\AW$, which shall be used frequently:
\begin{align}
\begin{tikzpicture}[baseline = .3mm,scale=.7,color=\clr]
	\draw[-,line width=1.5pt] (0.08,.3) to (0.08,.5);
\draw[-,thick] (-.2,-.8) to [out=45,in=-45] (0.1,.31);
\draw[-,thick] (.36,-.8) to [out=135,in=-135] (0.06,.31);
        \node at (-.3,-.95) {$\scriptstyle a$};
        \node at (.45,-.95) {$\scriptstyle b$};
\end{tikzpicture}
=
\begin{tikzpicture}[baseline = -.6mm,scale=.7,color=\clr]
	\draw[-,line width=1.5pt] (0.08,.1) to (0.08,.5);
\draw[-,thick] (.46,-.8) to [out=100,in=-45] (0.1,.11);
\draw[-,thick] (-.3,-.8) to [out=80,in=-135] (0.06,.11);
        \node at (-.3,-.95) {$\scriptstyle a$};
        \node at (.43,-.95) {$\scriptstyle b$};
\end{tikzpicture},
& \qquad
\begin{tikzpicture}[anchorbase,scale=.7,color=\clr]
	\draw[-,line width=1.5pt] (0.08,-.3) to (0.08,-.5);
\draw[-,thick] (-.2,.8) to [out=-45,in=45] (0.1,-.31);
\draw[-,thick] (.36,.8) to [out=-135,in=135] (0.06,-.31);
        \node at (-.3,.95) {$\scriptstyle a$};
        \node at (.45,.95) {$\scriptstyle b$};
\end{tikzpicture}
=
\begin{tikzpicture}[anchorbase,scale=.7,color=\clr]
	\draw[-,line width=1.5pt] (0.08,-.1) to (0.08,-.5);
\draw[-,thick] (.46,.8) to [out=-100,in=45] (0.1,-.11);
\draw[-,thick] (-.3,.8) to [out=-80,in=135] (0.06,-.11);
        \node at (-.3,.95) {$\scriptstyle a$};
        \node at (.43,.95) {$\scriptstyle b$};
\end{tikzpicture},
\label{swallows}
\\
\begin{tikzpicture}[anchorbase,scale=0.7,color=\clr]
	\draw[-,thick] (0.4,0) to (-0.6,1);
	\draw[-,thick] (0.08,0) to (0.08,1);
	\draw[-,thick] (0.1,0) to (0.1,.6) to (.5,1);
        \node at (0.6,1.13) {$\scriptstyle c$};
        \node at (0.1,1.16) {$\scriptstyle b$};
        \node at (-0.65,1.13) {$\scriptstyle a$};
\end{tikzpicture}
\!\!=\!\!
\begin{tikzpicture}[anchorbase,scale=0.7,color=\clr]
	\draw[-,thick] (0.7,0) to (-0.3,1);
	\draw[-,thick] (0.08,0) to (0.08,1);
	\draw[-,thick] (0.1,0) to (0.1,.2) to (.9,1);
        \node at (0.9,1.13) {$\scriptstyle c$};
        \node at (0.1,1.16) {$\scriptstyle b$};
        \node at (-0.4,1.13) {$\scriptstyle a$};
\end{tikzpicture},\:
\begin{tikzpicture}[anchorbase,scale=0.7,color=\clr]
	\draw[-,thick] (-0.4,0) to (0.6,1);
	\draw[-,thick] (-0.08,0) to (-0.08,1);
	\draw[-,thick] (-0.1,0) to (-0.1,.6) to (-.5,1);
        \node at (0.7,1.13) {$\scriptstyle c$};
        \node at (-0.1,1.16) {$\scriptstyle b$};
        \node at (-0.6,1.13) {$\scriptstyle a$};
\end{tikzpicture}
\!\!=\!\!
\begin{tikzpicture}[anchorbase,scale=0.7,color=\clr]
	\draw[-,thick] (-0.7,0) to (0.3,1);
	\draw[-,thick] (-0.08,0) to (-0.08,1);
	\draw[-,thick] (-0.1,0) to (-0.1,.2) to (-.9,1);
        \node at (0.4,1.13) {$\scriptstyle c$};
        \node at (-0.1,1.16) {$\scriptstyle b$};
        \node at (-0.95,1.13) {$\scriptstyle a$};
\end{tikzpicture},
&
\:\begin{tikzpicture}[baseline=-3.3mm,scale=0.7,color=\clr]
	\draw[-,thick] (0.4,0) to (-0.6,-1);
	\draw[-,thick] (0.08,0) to (0.08,-1);
	\draw[-,thick] (0.1,0) to (0.1,-.6) to (.5,-1);
        \node at (0.6,-1.13) {$\scriptstyle c$};
        \node at (0.07,-1.13) {$\scriptstyle b$};
        \node at (-0.6,-1.13) {$\scriptstyle a$};
\end{tikzpicture}
\!\!=\!\!
\begin{tikzpicture}[baseline=-3.3mm,scale=0.7,color=\clr]
	\draw[-,thick] (0.7,0) to (-0.3,-1);
	\draw[-,thick] (0.08,0) to (0.08,-1);
	\draw[-,thick] (0.1,0) to (0.1,-.2) to (.9,-1);
        \node at (1,-1.13) {$\scriptstyle c$};
        \node at (0.1,-1.13) {$\scriptstyle b$};
        \node at (-0.4,-1.13) {$\scriptstyle a$};
\end{tikzpicture},\:
\begin{tikzpicture}[baseline=-3.3mm,scale=0.7,color=\clr]
	\draw[-,thick] (-0.4,0) to (0.6,-1);
	\draw[-,thick] (-0.08,0) to (-0.08,-1);
	\draw[-,thick] (-0.1,0) to (-0.1,-.6) to (-.5,-1);
        \node at (0.6,-1.13) {$\scriptstyle c$};
        \node at (-0.1,-1.13) {$\scriptstyle b$};
        \node at (-0.6,-1.13) {$\scriptstyle a$};
\end{tikzpicture}
\!\!=\!\!
\begin{tikzpicture}[baseline=-3.3mm,scale=0.7,color=\clr]
	\draw[-,thick] (-0.7,0) to (0.3,-1);
	\draw[-,thick] (-0.08,0) to (-0.08,-1);
	\draw[-,thick] (-0.1,0) to (-0.1,-.2) to (-.9,-1);
        \node at (0.34,-1.13) {$\scriptstyle c$};
        \node at (-0.1,-1.13) {$\scriptstyle b$};
        \node at (-0.95,-1.13) {$\scriptstyle a$};
\end{tikzpicture},
\label{sliders}
\\
\label{symmetric}
\mathord{
\begin{tikzpicture}[baseline = -1mm,scale=0.8,color=\clr]
	\draw[-,thick] (0.28,0) to[out=90,in=-90] (-0.28,.6);
	\draw[-,thick] (-0.28,0) to[out=90,in=-90] (0.28,.6);
	\draw[-,thick] (0.28,-.6) to[out=90,in=-90] (-0.28,0);
	\draw[-,thick] (-0.28,-.6) to[out=90,in=-90] (0.28,0);
        \node at (0.3,-.75) {$\scriptstyle b$};
        \node at (-0.3,-.75) {$\scriptstyle a$};
\end{tikzpicture}
}
&=
\mathord{
\begin{tikzpicture}[baseline = -1mm,scale=0.8,color=\clr]
	\draw[-,thick] (0.2,-.6) to (0.2,.6);
	\draw[-,thick] (-0.2,-.6) to (-0.2,.6);
        \node at (0.2,-.75) {$\scriptstyle b$};
        \node at (-0.2,-.75) {$\scriptstyle a$};
\end{tikzpicture}
}\:,
\\
\label{braid}
\mathord{
\begin{tikzpicture}[baseline = -1mm,scale=0.8,color=\clr]
	\draw[-,thick] (0.45,.6) to (-0.45,-.6);
	\draw[-,thick] (0.45,-.6) to (-0.45,.6);
        \draw[-,thick] (0,-.6) to[out=90,in=-90] (-.45,0);
        \draw[-,thick] (-0.45,0) to[out=90,in=-90] (0,0.6);
        \node at (0,-.77) {$\scriptstyle b$};
        \node at (0.5,-.77) {$\scriptstyle c$};
        \node at (-0.5,-.77) {$\scriptstyle a$};
\end{tikzpicture}
}
&=
\mathord{
\begin{tikzpicture}[baseline = -1mm,scale=0.8,color=\clr]
	\draw[-,thick] (0.45,.6) to (-0.45,-.6);
	\draw[-,thick] (0.45,-.6) to (-0.45,.6);
        \draw[-,thick] (0,-.6) to[out=90,in=-90] (.45,0);
        \draw[-,thick] (0.45,0) to[out=90,in=-90] (0,0.6);
        \node at (0,-.77) {$\scriptstyle b$};
        \node at (0.5,-.77) {$\scriptstyle c$};
        \node at (-0.5,-.77) {$\scriptstyle a$};
\end{tikzpicture}
}\:.
\end{align}

\begin{rem}  [An alternative definition of $\W$]
\label{rem:cross}
We include the crossings as generating morphisms in $\W$, but $\W$ can be defined with the merges and splits as the generating morphisms, subject to \eqref{webassoc} and 
\begin{equation}\label{rung-sawp}
\begin{tikzpicture}[baseline = 2mm,scale=.8,color=\clr]
	\draw[-,thick] (0,0) to (0,1);
	\draw[-,thick] (.015,0) to (0.015,.2) to (.57,.4) to (.57,.6)
        to (.015,.8) to (.015,1);
	\draw[-,line width=1.2pt] (0.6,0) to (0.6,1);
        \node at (0.6,-.1) {$\scriptstyle b$};
        \node at (0,-.1) {$\scriptstyle a$};
        \node at (0.3,.84) {$\scriptstyle c$};
        \node at (0.3,.19) {$\scriptstyle d$};
\end{tikzpicture}
=
\sum_{t \ge 0}
\binom{a\! -\!b\!+\!c\!-\!d}{t}
\begin{tikzpicture}[baseline = 2mm,scale=.8, color=\clr]
	\draw[-,line width=1.2pt] (0,0) to (0,1);
	\draw[-,thick] (0.8,0) to (0.8,.2) to (.03,.4) to (.03,.6)
        to (.8,.8) to (.8,1);
	\draw[-,thin] (0.82,0) to (0.82,1);
        \node at (0.81,-.1) {$\scriptstyle b$};
        \node at (0,-.1) {$\scriptstyle a$};
        \node at (0.4,.9) {$\scriptstyle d-t$};
        \node at (0.4,.13) {$\scriptstyle c-t$};
\end{tikzpicture},
\end{equation} 
where diagrams containing a negatively-labeled strand are regarded as zero (cf. \cite{CKM}, \cite[Remark 4.8]{BEEO}). Then the crossing can be defined via (e.g., \cite{CKM}, \cite[(4.36)]{BEEO})  
\begin{equation}\label{crossgen}
\begin{tikzpicture}[baseline = 2mm,scale=.8, color=\clr]
	\draw[-,line width=1.2pt] (0,0) to (.6,1);
	\draw[-,line width=1.2pt] (0,1) to (.6,0);
        \node at (0,-.1) {$\scriptstyle a$};
        \node at (0.6,-0.1) {$\scriptstyle b$};
\end{tikzpicture}
:=\sum_{t=0}^{\min(a,b)}
(-1)^t
\begin{tikzpicture}[baseline = 2mm,scale=.8, color=\clr]
	\draw[-,thick] (0,0) to (0,1);
	\draw[-,thick] (.015,0) to (0.015,.2) to (.57,.4) to (.57,.6)
        to (.015,.8) to (.015,1);
	\draw[-,line width=1.2pt] (0.6,0) to (0.6,1);
        \node at (0.6,-.1) {$\scriptstyle b$};
        \node at (0,-.1) {$\scriptstyle a$};
         \node at (0.6,1.1) {$\scriptstyle a$};
        \node at (0,1.1) {$\scriptstyle b$};        \node at (-0.13,.5) {$\scriptstyle t$};
\end{tikzpicture}.
\end{equation}
For example, by \eqref{mergesplit}, we have 
\begin{equation}
\begin{tikzpicture}[baseline = 2mm,scale=.8, color=\clr]
	\draw[-,thick] (0,0) to (.28,.3) to (.28,.7) to (0,1);
	\draw[-,thick] (.6,0) to (.31,.3) to (.31,.7) to (.6,1);
        \node at (0,1.13) {$\scriptstyle 1$};
        \node at (0.63,1.13) {$\scriptstyle 1$};
        \node at (0,-.2) {$\scriptstyle 1$};
        \node at (0.63,-.2) {$\scriptstyle 1$};
\end{tikzpicture}
=
\begin{tikzpicture}[baseline = 2mm, scale=.8, color=\clr]
\draw[-,thick] (0,0)to[out=up,in=down](0,1);
\draw[-,thick] (0.5,0)to[out=up,in=down](0.5,1); 
\node at (0,-.2) {$\scriptstyle 1$};
 \node at (0.5,-0.2) {$\scriptstyle 1$};
 \end{tikzpicture}
 + 
 \begin{tikzpicture}[baseline = 2mm,scale=.8, color=\clr]
	\draw[-,thick] (0,0) to (.6,1);
	\draw[-,thick] (0,1) to (.6,0);
        \node at (0,-.2) {$\scriptstyle 1$};
        \node at (0.6,-0.2) {$\scriptstyle 1$};
\end{tikzpicture}.    
\end{equation}
\end{rem}
\subsection{The Schur category}
A Schur category was defined in \cite{BEEO} whose path algebras are certain idempotent subalgebras of the Schur algebras. It was shown that this Schur category is isomorphic to the polynomial web category $\W$ as strict monoidal categories \cite[Theorem~ 4.10]{BEEO}. We recall this result in this subsection.

For any $m\in \mathbb N$, let $\mathcal H_m:= \kk \mathfrak S_m$ be the group algebra of the symmetric group $\mathfrak S_m$.
For any $\lambda=(\lambda_1,\ldots, \lambda_k)\in \Lambda_{\text{st}}(m)$, let $\mathfrak S_\lambda$ be the Young subgroup $\mathfrak S_{\lambda_1}\times \ldots \times \mathfrak S_{\lambda_k}$ and define 
\begin{equation}
    \textsc{x}_{\lambda}=\sum_{w\in \mathfrak S_\lambda} w.
\end{equation}
Then the permutation module of $\mathcal H_m$ is defined to be
$$ M^\lambda: = \x_\lambda \mathcal H_m.$$
The Schur category $\Sch$ is the category with the object set $\Lambda_{\text{st}}$. For any $\lambda\in \Lambda_{\text{st}}(m)$ and $\mu\in\Lambda_{\text{st}}(m')$, the morphism space 
$\Hom_{\Sch}(\lambda,\mu)$ is $0$ unless $m=m'$; in this case, 
$\Hom_{\Sch}(\lambda,\mu)=\Hom_{\mathcal H_m}(M^\lambda,M^\mu)$.

For any $\lambda =(\lambda_1, \ldots,\lambda_{k})\in \Lambda_{\text{st}}$, the {\em $i$-th merge} of $\lambda$ is defined to be     $\la^{\vartriangle_i}$.
$$\la^{\vartriangle_i}:=(\lambda_1, \lambda_2,\ldots, \lambda_{i-1},\lambda_{i}+\lambda_{i+1},\lambda_{i+2},\ldots, \lambda_k),$$ 
for $1\le i\le k-1$. Moreover, we define $\sigma_{\lambda_i, \lambda_{i+1}}\in \mathcal H_m$  such that 
\begin{equation}\label{glambdai}
    \x_{\la^{\vartriangle_i}}=\x_\lambda \sigma_{\lambda_i,\lambda_{i+1}} = \sigma_{\lambda_i,\lambda_{i+1}}^* \x_\lambda,
\end{equation}
where $*$ is the anti-involution of $\mathcal H_m$ fixing all generators $s_i$'s
and each summand of $\sigma_{\lambda_i,\lambda_{i+1}}$ is of minimal length representative element of $\mathfrak S_{\lambda}\backslash \mathfrak S_{\lambda^{\vartriangle_i}}$,
i.e., $\sigma_{\lambda_i,\lambda_{i+1}}=\sum_{w\in (\mathfrak S_\la\backslash\mathfrak S_{\la^{\vartriangle_i}})_{\text{min}} }w$.

\begin{proposition}
[{\cite[Theorem 4.10]{BEEO}}]
\label{cor-isom-schur}
 We have an algebra isomorphism  
\[
\phi: \bigoplus_{\lambda,\mu \in \Lambda_{\text{st}}(m)}\Hom_{\W}(\lambda,\mu) \stackrel{\cong}{\longrightarrow} \End_{\mathcal H_m}(\oplus _{\lambda\in \Lambda_{\text{st}}(m)} M^\lambda)
\]
such that the generators are sent by
$$\begin{aligned}\phi\Big(1_{*}\begin{tikzpicture}[baseline = -.5mm,color=\clr]
	\draw[-,line width=1pt] (0.28,-.3) to (0.08,0.04);
	\draw[-,line width=1pt] (-0.12,-.3) to (0.08,0.04);
	\draw[-,line width=1.5pt] (0.08,.4) to (0.08,0);
        \node at (-0.22,-.4) {$\scriptstyle \lambda_i$};
        \node at (0.4,-.45) {$\scriptstyle \lambda_{i+1}$};\end{tikzpicture}1_{*}\Big) &: M^{\lambda} \rightarrow M^{\la^{\vartriangle_i} }, \quad  \x_\lambda h\mapsto \x_{\la^{\vartriangle_i}}h, \\     
\phi\Big(1_{*}\begin{tikzpicture}[baseline = -.5mm,color=\clr]
	\draw[-,line width=1.5pt] (0.08,-.3) to (0.08,0.04);
	\draw[-,line width=1pt] (0.28,.4) to (0.08,0);
	\draw[-,line width=1pt] (-0.12,.4) to (0.08,0);
        \node at (-0.22,.6) {$\scriptstyle \lambda_i$};
        \node at (0.36,.6) {$\scriptstyle \lambda_{i+1}$};
\end{tikzpicture}1_{*} \Big)&: M^{\la^{\vartriangle_i}}\mapsto M^{\lambda},\quad  \x_{\la^{\vartriangle_i}}h\mapsto \x_{\la^{\vartriangle_i}}h,
\end{aligned}$$
for any $h\in \mathcal H_m$, where $1_*$ stands for suitable identity morphisms.   
\end{proposition}

\begin{proof}
Let $GL_n:=GL_n(\kk)$ be the general linear group and $V$ be its defining representation with basis $\{v_i\mid 1\le i\le n\}$ with $n\ge m$. 
Then $V^{\otimes m}$ is naturally a right $\mathcal H_m$-module. 
Then the result  is implied by \cite[Theorem 4.10]{BEEO} since we have a well-known isomorphism as right $\mathcal H_m$-modules
$v_\lambda \mathcal H_m\stackrel{\simeq}{\longrightarrow} M^\lambda$ sending $v_\lambda \mapsto \x_\lambda$, where 
$v_{\lambda}:=v_1\otimes \ldots\otimes v_1\otimes v_2\otimes v_2\otimes \ldots$, in which each $v_i$ appears  $\lambda_i$ times.
\end{proof}

The explicit isomorphism given in Proposition \ref{cor-isom-schur} will be used in the study of the path algebras in the cyclotomic Schur category.

\subsection{More relations in $\AW$}

Recall $\wkdota$ from \eqref{equ:r=0andr=a}, for $0\le r \le a$.

\begin{lemma}
\label{lem:splitmerge-simplify}
  The following relation holds in $\AW$:
 \begin{align}
\label{splitmerge}
\begin{tikzpicture}[baseline = -1mm,scale=.8,color=\clr]
\draw[-,line width=1.5pt] (0.08,-.8) to (0.08,-.5);
\draw[-,line width=1.5pt] (0.08,.3) to (0.08,.6);
\draw[-,line width=1pt] (0.1,-.51) to [out=45,in=-45] (0.1,.31);
\draw[-,line width=1pt] (0.06,-.51) to [out=135,in=-135] (0.06,.31);
\draw(-.1,0) \bdot;
\draw(-.4,0)node {$\scriptstyle \omega_r$};
\node at (-.33,-.35) {$\scriptstyle a$};
\node at (.45,-.35) {$\scriptstyle b$};
\end{tikzpicture}
&= 
\binom{a+b-r}{b}\:\:
\begin{tikzpicture}[baseline = -1mm,scale=.6,color=\clr]
\draw[-,line width=1.5pt] (0.08,-.7) to (0.08,.5);
\node at (.08,-.9) {$\scriptstyle a+b$};
\draw(0.08,0) \bdot;
\draw(.5,0)node {$\scriptstyle \omega_r$};
\end{tikzpicture}
, \qquad \text{ for } 0\le r \le a. 
\end{align}
\end{lemma}

\begin{proof}
 Using \eqref{equ:r=0} and \eqref{equ:r=0andr=a}  we have 
   \begin{align*}
   \begin{tikzpicture}[baseline = -1mm,scale=.8,color=\clr]
\draw[-,line width=1.5pt] (0.08,-.8) to (0.08,-.5);
\draw[-,line width=1.5pt] (0.08,.3) to (0.08,.6);
\draw[-,line width=1pt] (0.1,-.51) to [out=45,in=-45] (0.1,.31);
\draw[-,line width=1pt] (0.06,-.51) to [out=135,in=-135] (0.06,.31);
\draw(-.1,0) \bdot;
\draw(-.4,0)node {$\scriptstyle \omega_r$};
\node at (-.33,-.35) {$\scriptstyle a$};
\node at (.45,-.35) {$\scriptstyle b$};
\end{tikzpicture}
= \begin{tikzpicture}[baseline = -1mm,scale=.8,color=\clr]
\draw[-,line width=1.5pt] (0.08,-.8) to (0.08,-.5);
\draw[-,line width=1.5pt] (0.08,.3) to (0.08,.6);
\draw[-,thick] (0.1,-.51) to (0.1,.31);
\draw[-,thick] (0.06,-.51) to [out=135,in=-135] (0.06,.31);
\draw[-,line width=1pt] (0.11,-.8) to [out=135,in=down] (.5,-.1) to[out=up, in =-135] (0.11,.6);
\draw(-.1,0) \bdot;
\draw(-.4,0)node {$\scriptstyle \omega_r$};
\node at (-.28,-.35) {$\scriptstyle r$};
\node at (.6,-.35) {$\scriptstyle b$};
\end{tikzpicture}
=
\binom{a+b-r}{b}\:\:
\begin{tikzpicture}[baseline = -1mm,scale=.6,color=\clr]
\draw[-,line width=1.5pt] (0.08,-.7) to (0.08,.5);
\node at (.08,-.9) {$\scriptstyle a+b$};
\draw(0.08,0) \bdot;
\draw(.5,0)node {$\scriptstyle \omega_r$};
\end{tikzpicture}.     
   \end{align*}
   The lemma is proved.
\end{proof}

\begin{lemma}
The following relations hold in $\AW$:
 \begin{equation}
\label{dotmovecrossing}
\begin{tikzpicture}[baseline = 7.5pt, scale=0.4, color=\clr]
\draw[-,line width=1.2pt] (0,-.2) to  (1,2.2);
\draw[-,line width=1.2pt] (1,-.2) to  (0,2.2);
\draw(0.2,1.6)\bdot;
\draw(-.4,1.6) node {$\scriptstyle \omega_r$};
\node at (0, -.5) {$\scriptstyle a$};
\node at (1, -.5) {$\scriptstyle b$};
\end{tikzpicture}
= 
 \sum_{0 \leq t \leq \min(a,b)}t!~
\begin{tikzpicture}[baseline = 7.5pt,scale=.8, color=\clr]
\draw[-,thick] (0.58,-.2) to (0.58,0) to (-.48,.8) to (-.48,1);
\draw[-,thick] (-.48,-.2) to (-.48,0) to (.58,.8) to (.58,1);
\draw[-,thick] (-.5,-.2) to (-.5,1);
\draw[-,thick] (0.605,-.2) to (0.605,1);
\draw(.35,0.2)\bdot;
\draw(.1,0) node {$\scriptstyle {\omega_{r-t}}$};
%\node at (0,1.13) {$\scriptstyle b$};
%\node at (0.6,1.13) {$\scriptstyle d$};
\node at (-.5,-.3) {$\scriptstyle a$};
\node at (0.6,-.3) {$\scriptstyle b$};
%\node at (-0.1,.5) {$\scriptstyle s$};
\node at (0.77,.5) {$\scriptstyle t$};
\end{tikzpicture},
\qquad
\begin{tikzpicture}[baseline = 7.5pt, scale=0.4, color=\clr]
\draw[-, line width=1.2pt] (0,-.2) to (1,2.2);
\draw[-,line width=1.2pt] (1,-.2) to(0,2.2);
\draw(.2,0.2)\bdot;
\draw(-.4,.2)node {$\scriptstyle \omega_r$};
\node at (0, -.5) {$\scriptstyle b$};
\node at (1, -.5) {$\scriptstyle a$};
\end{tikzpicture}
=  
 \sum_{0 \leq t \leq \min(a,b)}t!~
\begin{tikzpicture}[baseline = 7.5pt,scale=.8, color=\clr]
\draw[-,thick] (0.58,-.2) to (0.58,0) to (-.48,.8) to (-.48,1);
\draw[-,thick] (-.48,-.2) to (-.48,0) to (.58,.8) to (.58,1);
\draw[-,thick] (-.5,-.2) to (-.5,1);
\draw[-,thick] (0.605,-.2) to (0.605,1);
\draw(.35,0.6)\bdot;
\draw(.1,0.9) node {$\scriptstyle {\omega_{r-t}}$};
%\node at (0,1.13) {$\scriptstyle b$};
%\node at (0.6,1.13) {$\scriptstyle d$};
\node at (-.5,-.3) {$\scriptstyle b$};
\node at (0.6,-.3) {$\scriptstyle a$};
%\node at (-0.1,.5) {$\scriptstyle s$};
\node at (0.77,.5) {$\scriptstyle t$};
\end{tikzpicture},
\end{equation}
for $0\le r \le b$. 
\end{lemma}

\begin{proof}
We only check the first equality in \eqref{dotmovecrossing} and the second one can be obtained similarly. We have 
\begin{align*}
 \begin{tikzpicture}[baseline = 7.5pt, scale=0.4, color=\clr]
\draw[-,line width=1.2pt] (0,-.2) to  (1,2.2);
\draw[-,line width=1.2pt] (1,-.2) to  (0,2.2);
\draw(0.2,1.6)\bdot;
\draw(-.4,1.6) node {$\scriptstyle \omega_r$};
\node at (0, -.5) {$\scriptstyle a$};
\node at (1, -.5) {$\scriptstyle b$};
\end{tikzpicture}
&\overset{\eqref{equ:r=0andr=a}}{\underset{\eqref{sliders}}{=}} ~
\begin{tikzpicture}[baseline = 7.5pt, scale=0.45, color=\clr]
\draw[-,line width=1.2pt](0,2.2) to (0.3,1.9);
\draw (.6,1.3)\bdot;
%\draw(1.2,1.8) \bdot;
\draw(.4,1)node{$\scriptstyle \omega_{r}$};
\draw[-,thick](0.3,1.9) to[out=-80,in=135](1.5,.5) to[out=-45,in=45](2,.5);
\draw[-,thick](0.3,1.9) to[out=45,in=135](1.5,1.5) to[out=-45,in=45] (2,.5);
\draw[-,line width=1.2pt](2,.5) to (2.5,0);
\draw[-,line width=1.2pt](0,0) to (2.5,2.2);
\node at (1.5, .25) {$\scriptstyle r$};
\node at (0, -.2) {$\scriptstyle a$};
\node at (2.5, -.2) {$\scriptstyle b$};
\end{tikzpicture}
\overset{\eqref{equ:dotmovecrossing-simplify}}{=}
\sum_{0\le t\le \min\{a,r\}}t!
\begin{tikzpicture}[baseline = 7.5pt,scale=.8, color=\clr]
\draw[-,thick] (0.58,-.4) to (0.58,0) to (-.48,.8) to (-.48,1);
\draw[-,thick](-.5,1) to (-.5,1.6);
\draw[-,thick] (-.5,1.4) to [out=-45,in=145](.58,1) to [out=right,in=45] (0.58,-.2);
\draw[-,thick](.58,1) to (.58,1.5);
\draw[-,thick] (-.48,-.2) to (-.48,0) to (.58,.8) to (.58,1);
\draw[-,thick] (-.5,-.2) to (-.5,1);
\draw[-,thick] (0.605,-.2) to (0.605,1);
\draw(.2,0.3)\bdot;
\draw(0.1,0) node {$\scriptstyle {\omega_{r-t}}$};
\node at (-.5,-.5) {$\scriptstyle a$};
\node at (0.6,-.6) {$\scriptstyle b$};
\node at (0.45,.5) {$\scriptstyle t$};
\end{tikzpicture}
\overset{\eqref{equ:r=0andr=a}}{\underset{\eqref{sliders}}{=}}
 \sum_{0 \leq t \leq \min(a,b)}t!~
\begin{tikzpicture}[baseline = 7.5pt,scale=.8, color=\clr]
\draw[-,thick] (0.58,-.2) to (0.58,0) to (-.48,.8) to (-.48,1);
\draw[-,thick] (-.48,-.2) to (-.48,0) to (.58,.8) to (.58,1);
\draw[-,thick] (-.5,-.2) to (-.5,1);
\draw[-,thick] (0.605,-.2) to (0.605,1);
\draw(.35,0.2)\bdot;
\draw(.1,0) node {$\scriptstyle {\omega_{b-t}}$};
%\node at (0,1.13) {$\scriptstyle b$};
%\node at (0.6,1.13) {$\scriptstyle d$};
\node at (-.5,-.3) {$\scriptstyle a$};
\node at (0.6,-.3) {$\scriptstyle b$};
%\node at (-0.1,.5) {$\scriptstyle s$};
\node at (0.77,.5) {$\scriptstyle t$};
\end{tikzpicture}.
\end{align*}
The lemma is proved.
\end{proof}

\begin{lemma}
  The following relations hold in $\AW$:
 \begin{align}\label{dotmovesplitss}
\begin{tikzpicture}[baseline = -.5mm,scale=.8,color=\clr]
\draw[-,line width=1.5pt] (0.08,-.5) to (0.08,0.04);
\draw[-,line width=1pt] (0.34,.5) to (0.08,0);
\draw[-,line width=1pt] (-0.2,.5) to (0.08,0);
\node at (-0.22,.6) {$\scriptstyle a$};
\node at (0.36,.65) {$\scriptstyle b$};
\draw (0.08,-.2) \bdot;
\draw (0.5,-.2) node{$\scriptstyle \omega_r$};
\end{tikzpicture} 
&=
\sum_{c+d+e=r} e!\binom{a-c}{e}\binom{b-d}{e}~
\begin{tikzpicture}[baseline = -.5mm,scale=.8,color=\clr]
\draw[-,line width=1.5pt] (0.08,-.5) to (0.08,0.04);
\draw[-,line width=1pt] (0.34,.5) to (0.08,0);
\draw[-,line width=1pt] (-0.2,.5) to (0.08,0);
\node at (-0.22,.6) {$\scriptstyle a$};
\node at (0.36,.65) {$\scriptstyle b$};
\draw (-.05,.24) \bdot;
\draw (-0.4,.2) node{$\scriptstyle \omega_{c}$};
\draw (0.6,.2) node{$\scriptstyle \omega_{d}$};
\draw (.22,.24) \bdot;
\end{tikzpicture},
\\
\label{dotmovemergess}
\begin{tikzpicture}[baseline = -.5mm, scale=.8, color=\clr]
\draw[-,line width=1pt] (0.3,-.5) to (0.08,0.04);
\draw[-,line width=1pt] (-0.2,-.5) to (0.08,0.04);
\draw[-,line width=1.5pt] (0.08,.6) to (0.08,0);
\node at (-0.22,-.6) {$\scriptstyle a$};
\node at (0.35,-.6) {$\scriptstyle b$};
\draw (0.08,.2) \bdot;
\draw (0.45,.2) node{$\scriptstyle \omega_r$};  
\end{tikzpicture}
 &= 
\sum_{c+d+e=r} e!\binom{a-c}{e}\binom{b-d}{e}~
\begin{tikzpicture}[baseline = -.5mm,scale=.8, color=\clr]
\draw[-,line width=1pt] (0.3,-.5) to (0.08,0.04);
\draw[-,line width=1pt] (-0.2,-.5) to (0.08,0.04);
\draw[-,line width=1.5pt] (0.08,.6) to (0.08,0);
\node at (-0.22,-.6) {$\scriptstyle a$};
\node at (0.35,-.6) {$\scriptstyle b$};
\draw (-.08,-.3) \bdot; \draw (.22,-.3) \bdot;
\draw (-0.4,-.3) node{$\scriptstyle \omega_{c}$};
\draw (0.6,-.3) node{$\scriptstyle \omega_{d}$};
\end{tikzpicture},
\end{align}
for $0\le r \le a+b.$
\end{lemma}

\begin{proof}
    We only show \eqref{dotmovemergess}, as the identity 
    \eqref{dotmovesplitss} follows by symmetry. 
    We have 
    \begin{align*}
\begin{tikzpicture}[baseline = -.5mm,scale=1,color=\clr]
\draw[-,line width=1.5pt] (0.08,-.5) to (0.08,0.04);
\draw[-,line width=1pt] (0.34,.5) to (0.08,0);
\draw[-,line width=1pt] (-0.2,.5) to (0.08,0);
\node at (-0.22,.6) {$\scriptstyle a$};
\node at (0.36,.65) {$\scriptstyle b$};
\draw (0.08,-.2) \bdot;
\draw (0.35,-.2) node{$\scriptstyle \omega_r$};
\end{tikzpicture} 
&
\overset{\eqref{equ:r=0andr=a}}=
\:\:
\begin{tikzpicture}[baseline = -.5mm,scale=1,color=\clr]
\draw[-,line width=1.5pt] (0.08,-.3) to (0.08,-.5);
\draw[-,line width=1.5pt] (0.08,.5) to (0.08,.3);
\draw[-,line width=1pt] (0.34,.8) to (0.08,0.5);
\draw[-,line width=1pt] (-0.2,.8) to (0.08,0.5);
\node at (-0.22,.9) {$\scriptstyle a$};
\node at (0.36,.9) {$\scriptstyle b$};
\draw (-.1,0) \bdot;
\node at (-.15, -.3) {$\scriptstyle r$};
%\draw (.25,0) \bdot;
\draw[-,line width=1pt](0.08,.3) to [out=left,in=left] (0.08, -.3);
\draw[-,line width=1pt](0.08,.3) to [out=right,in=right] (0.08, -.3);
\draw (-0.4,0) node{$\scriptstyle \omega_{r}$};
\end{tikzpicture} 
\overset{\eqref{mergesplit}}{=}
\:\:
\sum_{\substack{0 \leq s \leq \min(a,r)\\0 \leq t \leq \min(a+b-r,b)\\t-s=b-r}}
\begin{tikzpicture}[baseline = 7.5pt,scale=.9, color=\clr]
	\draw[-,thick] %(0.58,0) to 
    (0.58,.2) to (.02,.8) to (.02,1);
	\draw[-,thick] %(0.02,0) to 
    (0.02,.2) to (.58,.8) to (.58,1);
	\draw[-,thick] (0,0.1) to (0,1);
	\draw[-,line width=1pt] (0.61,0.1) to (0.61,1);
 \draw[-,line width=1pt](.3,-.4) to [out=left,in=up] (0,0.1);
 \draw[-,line width=1pt](.3,-.4) to [out=right,in=up] (0.6,0.1);
 \draw (.1,-.2) \bdot;
 \draw (-.2,-.2) node{$\scriptstyle \omega_r$};
 \draw[-,line width=1.5pt](.3,-.6) to (.3,-.4);
        \node at (0,1.16) {$\scriptstyle a$};
        \node at (0.6,1.16) {$\scriptstyle b$};
        \node at (-0.1,.5) {$\scriptstyle s$};
        \node at (0.77,.5) {$\scriptstyle t$};
\end{tikzpicture}
\\
&
\overset{\eqref{dotmovesplits+merge-simplify}}{=}
\sum_{\substack{0 \leq s \leq \min(a,r)\\0 \leq t \leq \min(a+b-r,b)\\t-s=b-r}}
\begin{tikzpicture}[baseline = 7.5pt,scale=1.1, color=\clr]
	\draw[-,thick] %(0.58,0) to 
    (0.58,.2) to (.02,.8) to (.02,1);
	\draw[-,thick] %(0.02,0) to 
    (0.02,.2) to (.58,.8) to (.58,1);
	\draw[-,thick] (0,0.1) to (0,1);
	\draw[-,thick] (0.6,0.1) to (0.6,1);
 \draw[-,line width=1pt](.3,-.4) to [out=left,in=up] (0,0.1);
 \draw[-,line width=1pt](.3,-.4) to [out=right,in=up] (0.6,0.1);
 \draw (0,.5) \bdot;
 \draw(.2,.35) \bdot;
 \draw (.3,.1) node{${\scriptstyle \omega_{r-s}}$};
 \draw[-,line width=1.5pt](.3,-.6) to (.3,-.4);
        \node at (0,1.16) {$\scriptstyle a$};
        \node at (0.6,1.16) {$\scriptstyle b$};
        \node at (-0.3,.5) {$\scriptstyle \omega_s$};
        \node at (0.77,.5) {$\scriptstyle t$};
\end{tikzpicture}
\\
&\overset{\eqref{equ:dotmovecrossing-simplify}}{=}
\sum_{\substack{0 \leq s \leq \min(a,r)\\0 \leq t \leq \min(a+b-r,b)\\t-s=b-r}}
\sum_{0\le e\le \min(a-s,r-s)}e!~\:\:
\begin{tikzpicture}[baseline = 7.5pt,scale=.8, color=\clr]
\draw[-,line width=1.5pt](-.8,1.5) to (-.8,1.2) to (-.48,.8);
\draw[-,thick](-.8,1.2) to (-.8,-.2);
\draw[-,line width=1.5pt] (-.8,-.2) to (-.48,0);
\draw[-,thick]  (0.58,0) to (-.48,.8);
\draw[-,line width=1.5pt](.58,0) to (.9,-.3)to (.9,-.4);
\draw[-,thick]  (-.48,0) to (.58,.8);
\draw[-,thick] (-.5,0) to (-.5,.8);
\draw[-,thick] (0.605,0) to (0.605,.8);
\draw(.35,0.6)\bdot;
\draw(-.8,0.5)\bdot;
\draw(-1.1,.5)node {$\scriptstyle {\omega_{s}}$};
\draw(0.01,.99) node {$\scriptstyle {\omega_{r-s-e}}$};
\node at (-.9,-.5) {$\scriptstyle r$};
%\node at (0.9,-.5) {$\scriptstyle b$};
\node at (0.7,.5) {$\scriptstyle e$};
\node at (1,.5) {$\scriptstyle t$};
\draw[-,line width=1.5pt](-.8,-.4) to (-.8,-.2);
\draw[-,thick](.9,-.3) to (.9,1.2);
\draw[-,line width=1.5pt](.9,1.2) to (.9,1.5) to (.58,.8);
\draw[-,thick](0,-.8) to [out=left,in=down] (-.8,-.2);
\draw[-,thick](0,-.8) to [out=right,in=down] (.9,-.2);
\draw[-,line width=1.5pt] (0,-1) to (0,-.8);
\end{tikzpicture}\\
&\overset{\eqref{equ:r=0},\eqref{equ:r=0andr=a}}{\underset{\eqref{sliders},\eqref{swallows}}{=}}
\sum_{c+d+e=r} e!\binom{a-c}{e}\binom{b-d}{e}~
\begin{tikzpicture}[baseline = -.5mm,scale=1,color=\clr]
\draw[-,line width=1.5pt] (0.08,-.5) to (0.08,0.04);
\draw[-,line width=1pt] (0.34,.5) to (0.08,0);
\draw[-,line width=1pt] (-0.2,.5) to (0.08,0);
\node at (-0.22,.6) {$\scriptstyle a$};
\node at (0.36,.65) {$\scriptstyle b$};
\draw (-.05,.24) \bdot;
\draw (-0.3,.2) node{$\scriptstyle \omega_{c}$};
\draw (0.55,.2) node{$\scriptstyle \omega_{d}$};
\draw (.22,.24) \bdot;
\end{tikzpicture}.
    \end{align*}
    The lemma is proved.
\end{proof}

\begin{rem}
    An alternative definition of $\AW$ would be adding to Definition~\ref{def-affine-web} more generating morphisms $\wkdota$,
 for $1\le r \le a$, and additional relations \eqref{splitmerge}, \eqref{dotmovecrossing}, \eqref{dotmovesplitss} and \eqref{dotmovemergess}. These additional generating morphisms make the basis theorem for $\AW$ (cf. Theorem~\ref{basisAW} and Theorem \ref{thm:Da}) more transparent.
\end{rem}

\subsection{Commutativity and  balloon relations}

By the symmetry of the defining relations in Definition \ref{def-affine-web}, there is an isomorphism of strict monoidal categories
\begin{equation}
\label{anti-auto}
\div : \AW\longrightarrow \mathpzc{Web}^{\bullet\,\text{op}}
%\AW^{\text{op}}    
\end{equation}
by reflecting each diagram along a horizontal axis, i.e., swapping the merge and split and fixing the dotted strand generators. 

Given $a\ge 1$, let $D_a$ be the subalgebra of $\End_{\AW}(a)$ generated by the elements $\wkdota$, for $0\le r \le a$. 

\begin{lemma}
\label{lem:commute}
The following commutativity relation holds in $\AW$:
\begin{equation}
    \label{commutato}
    \begin{tikzpicture}[baseline = 1.5mm, scale=.8, color=\clr]
\draw[-,line width=1.5pt] (0,-.4) to[out=up, in=down] (0,.8);
\draw(0,0) \bdot; 
\draw (0.4,0) node {$\scriptstyle \omega_r$};
\draw(0,0.4) \bdot; 
\draw (0.4,0.4) node {$\scriptstyle \omega_t$};
\node at (0,-.5) {$\scriptstyle a$};
\end{tikzpicture}
~=~
\begin{tikzpicture}[baseline = 1.5mm, scale=.8, color=\clr]
\draw[-,line width=1.5pt] (0,-.4) to[out=up, in=down] (0,.8);
\draw(0,0) \bdot; 
\draw (0.4,0) node {$\scriptstyle \omega_t$};
\draw(0,0.4) \bdot; 
\draw (0.4,0.4) node {$\scriptstyle \omega_r$};
\node at (0,-.5) {$\scriptstyle a$};
\end{tikzpicture},
\end{equation}    
for any admissible $r,t,a$. In particular, $D_a$ is a commutative subalgebra of $\End_{\AW}(a)$. 
\end{lemma}
(It will be shown in Theorem \ref{thm:Da} that $\End_{\AW}(a) =D_a$ and that it is a polynomial algebra.)

\begin{proof}
   We may assume $1\le r\le t\le a$. We prove by induction on $r+t+a$. The case $r=t=1=a$ is trivial. In the following, we may assume that $\le r <t\le a$.
   By \eqref{splitmerge}, 
   we have 
   \[ 
\begin{tikzpicture}[baseline = -1mm,scale=.8,color=\clr]
\draw[-,line width=1.5pt] (0.08,-.7) to (0.08,.5);
\node at (.08,-.8) {$\scriptstyle a$};
\draw(0.08,0) \bdot;
\draw(.45,0)node {$\scriptstyle \omega_r$};
\end{tikzpicture}
= 
\begin{tikzpicture}[baseline = -1mm,scale=.8,color=\clr]
	\draw[-,line width=1.5pt] (0.08,-.8) to (0.08,-.5);
	\draw[-,line width=1.5pt] (0.08,.3) to (0.08,.6);
\draw[-,thick] (0.1,-.51) to [out=45,in=-45] (0.1,.31);
\draw[-,thick] (0.06,-.51) to [out=135,in=-135] (0.06,.31);
\draw(-.1,0) \bdot;
\draw(-.4,0)node {$\scriptstyle \omega_r$};
\node at (-.22,-.35) {$\scriptstyle r$};
\node at (.6,-.35) {$\scriptstyle a-r$};
\end{tikzpicture}. 
\]
Then, we have 
\begin{align*}
  \begin{tikzpicture}[baseline = 1.5mm, scale=.8, color=\clr]
\draw[-,line width=1.5pt] (0,-.4) to[out=up, in=down] (0,.8);
\draw(0,0) \bdot; 
\draw (0.4,0) node {$\scriptstyle \omega_t$};
\draw(0,0.4) \bdot; 
\draw (0.4,0.4) node {$\scriptstyle \omega_r$};
\node at (0,-.5) {$\scriptstyle a$};
\end{tikzpicture}&= 
\begin{tikzpicture}[baseline = -1mm,scale=.8,color=\clr]
\draw[-,line width=1.5pt] (0.08,-1) to (0.08,-.5);
\draw[-,line width=1.5pt] (0.08,.3) to (0.08,.6);
\draw[-,thick] (0.1,-.51) to [out=45,in=-45] (0.1,.31);
\draw[-,thick] (0.06,-.51) to [out=135,in=-135] (0.06,.31);
\draw(-.1,0) \bdot;
\draw(-.45,0)node {$\scriptstyle \omega_r$};
\node at (-.22,-.35) {$\scriptstyle r$};
\node at (.6,-.35) {$\scriptstyle a-r$};
\draw(.08,-.7) \bdot;
\draw(0.45,-.7)node {$\scriptstyle \omega_t$};
\end{tikzpicture}
\overset{\eqref{dotmovesplitss}} =
\sum_{c+d+e=t} e!\binom{r-c}{e}\binom{a-r-d}{e}~
\begin{tikzpicture}[baseline = -1mm,scale=.8,color=\clr]
\draw[-,line width=1.5pt] (0.08,-1) to (0.08,-.5);
\draw[-,line width=1.5pt] (0.08,.3) to (0.08,.6);
\draw[-,thick] (0.1,-.51) to [out=45,in=-45] (0.1,.31);
\draw[-,thick] (0.06,-.51) to [out=135,in=-135] (0.06,.31);
\draw(-.1,0) \bdot;
\draw(-.45,0)node {$\scriptstyle \omega_r$};
\node at (-.43,-.35) {$\scriptstyle \omega_{c}$};
\node at (.6,-.35) {$\scriptstyle \omega_{d}$};
\draw(.2,-.34) \bdot;
\draw(-.1,-.34) \bdot;
\end{tikzpicture}.
\end{align*}
Then the identity \eqref{commutato} follows since each diagram in the summation above is fixed under $\div$ by the induction hypothesis.
\end{proof}

For any $r\ge 1$ and $u\in \kk$, we define 
\begin{equation}
    \label{def-gau}
 g_{r,u}:=\sum_{0\le i\le r}(-1)^i\prod_{0\le j\le i-1}(u+j)\:\:
 \begin{tikzpicture}[baseline = -1mm,scale=.8,color=\clr]
\draw[-,line width=1.5pt] (0.08,-.6) to (0.08,.5);
\node at (.08,-.8) {$\scriptstyle r$};
\draw(0.08,0) \bdot;
\draw(.65,0)node {$\scriptstyle \omega_{r-i}$};
\end{tikzpicture}
\in \End_{\AW}(r).
\end{equation}
The element $g_{r,u}$ and the following lemma will  be used for the cyclotomic web category later and also for the affine Schur category \cite{SW2}.

\begin{lemma}
 \label{lem:gru}
 For any $r\ge 1$ and $u\in \kk$, the following relation (called a {\em balloon relation}) holds in $\AW$:
\[
\begin{tikzpicture}[baseline = 5pt, scale=.5, color=\clr]
\draw[-, line width=1.2pt] (0.5,2) to (0.5,2.3);
\draw[-, line width=1.2pt] (0.5,0) to (0.5,-.3);
\draw[-,thin](0.5,2) to[out=left,in=up] (-.5,1) to[out=down,in=left] (0.5,0);
\draw[-,thin]  (0.5,2) to[out=left,in=up] (0,1) to[out=down,in=left] (0.5,0);   
\draw[-,thin] (0.5,0)to[out=right,in=down] (1.5,1)to[out=up,in=right] (0.5,2);
\draw[-,thin] (0.5,0)to[out=right,in=down] (1,1)
 to[out=up,in=right] (0.5,2);
\node at (0.5,.7){$\scriptstyle \cdots$};
\draw (-0.5,1) \bdot; 
 \node at (-.75,1) {$\scriptstyle g$}; 
\draw (0,1) \bdot; 
\node at (0.3,1) {$\scriptstyle g$};
\draw (1,1) \bdot;
\draw (1.5,1) \bdot; 
\node at (1.8,1) {$\scriptstyle g$};
\node at (-.4,0) {$\scriptstyle 1$};
 \node at (.2,0.3) {$\scriptstyle 1$};
\node at (.7,0.3) {$\scriptstyle 1$};
\node at (1.2,0) {$\scriptstyle 1$};
\draw (.5,-.5) node{$\scriptstyle {r}$};
\end{tikzpicture}
=
r!\, g_{r,u},
\qquad \text{ where }
 \begin{tikzpicture}[baseline = 3pt, scale=0.5, color=\clr]
\draw[-,thin] (0,0) to[out=up, in=down] (0,1.4);
\draw(0,0.6) \bdot; 
\draw (0.65,0.6) node {$\scriptstyle g$};
\node at (0,-.3) {$\scriptstyle 1$};
\end{tikzpicture} 
=
\begin{tikzpicture}[baseline = 3pt, scale=0.5, color=\clr]
\draw[-,thin] (0,0) to[out=up, in=down] (0,1.4);
\draw(0,0.6) \bdot; 
\draw (0.65,0.6) node {$\scriptstyle \omega_1$};
\node at (0,-.3) {$\scriptstyle 1$};
\end{tikzpicture}
-u 
\begin{tikzpicture}[baseline = 3pt, scale=0.5, color=\clr]
\draw[-,thin] (0,0) to[out=up, in=down] (0,1.4);
%\draw(0,0.6) \bdot; 
%\draw (0.65,0.6) node {$\scriptstyle g$};
\node at (0,-.3) {$\scriptstyle 1$};
\end{tikzpicture} .
\]
\end{lemma}

\begin{proof}
We prove by induction on $r$.
The case $r=1$ is trivial by definition.
Suppose $r\ge2$. Write $a_{i}=\prod_{0\le j\le i-1}(u+j)$ for $1\le i\le r$. By inductive assumption on $r-1$, we have 
\begin{align*}
\begin{tikzpicture}[baseline = 5pt, scale=.5, color=\clr]
\draw[-, line width=1.2pt] (0.5,2) to (0.5,2.3);
\draw[-, line width=1.2pt] (0.5,0) to (0.5,-.3);
\draw[-,thin](0.5,2) to[out=left,in=up] (-.5,1) to[out=down,in=left] (0.5,0);
\draw[-,thin]  (0.5,2) to[out=left,in=up] (0,1) to[out=down,in=left] (0.5,0);   
\draw[-,thin] (0.5,0)to[out=right,in=down] (1.5,1)to[out=up,in=right] (0.5,2);
\draw[-,thin] (0.5,0)to[out=right,in=down] (1,1)
 to[out=up,in=right] (0.5,2);
\node at (0.5,.7){$\scriptstyle \cdots$};
\draw (-0.5,1) \bdot; 
 \node at (-.75,1) {$\scriptstyle g$}; 
\draw (0,1) \bdot; 
\node at (0.3,1) {$\scriptstyle g$};
\draw (1,1) \bdot;
\draw (1.5,1) \bdot; 
\node at (1.8,1) {$\scriptstyle g$};
\node at (-.4,0) {$\scriptstyle 1$};
 \node at (.2,0.3) {$\scriptstyle 1$};
\node at (.7,0.3) {$\scriptstyle 1$};
\node at (1.2,0) {$\scriptstyle 1$};
\draw (.5,-.5) node{$\scriptstyle {r}$};
\end{tikzpicture}
&
= 
 \sum_{0\le i\le r-1}(-1)^i  (r-1)!a_{i} \:\:
 \begin{tikzpicture}[baseline = -1mm,scale=1,color=\clr]
	\draw[-,line width=1.5pt] (0.08,-.8) to (0.08,-.5);
	\draw[-,line width=1.5pt] (0.08,.3) to (0.08,.6);
\draw[-,thin] (0.1,-.51) to [out=45,in=-45] (0.1,.31);
\draw[-,thick] (0.06,-.51) to [out=135,in=-135] (0.06,.31);
\draw(-.1,0) \bdot;
\draw(.25,0) \bdot;
\node at (.45,0) {$\scriptstyle g$};
\draw(-.8,0)node {$\scriptstyle \omega_{r-1-i}$};
\node at (-.4,-.45) {$\scriptstyle r-1$};
\node at (.4,-.45) {$\scriptstyle 1$};
\end{tikzpicture}
\\
&
\overset{\eqref{splitmerge}}=
\sum_{0\le i\le r-1}(-1)^i (r-1)! a_{i} \:\:
\begin{tikzpicture}[baseline = -1mm,scale=.8,color=\clr]
\draw[-,line width=1.5pt] (0,-.8) to (0.,-.5);
\draw[-,line width=1.5pt] (0.,.5) to (0,.8);
\draw[-,thick] (0,-.5) to [out=45,in=-45] (0.,.5);
\draw[-,thin] (0,-.5) to [out=right,in=down] (0.5,0)to [out=up,in=right](0,.5);
\draw[-,thick] (0,-.5) to [out=left,in=down] (-0.5,0)to [out=up,in=left](0,.5);
\draw(-.5,0) \bdot;
\draw(.5,0) \bdot;
\node at (.7,0) {$\scriptstyle g$};
\draw(-1.2,0)node {$\scriptstyle \omega_{r-1-i}$};
\node at (0,-.2) {$\scriptstyle i$};
\node at (.4,-.45) {$\scriptstyle 1$};
\end{tikzpicture}
\\
&
\overset{\eqref{webassoc}}{\underset{\eqref{swallows}}{=}}
\sum_{0\le i\le r-1}(-1)^i (r-1)! a_i \:\:
\begin{tikzpicture}[baseline = -1mm,scale=.7,color=\clr]
\draw[-,line width=1.5pt] (0,-1.2) to (0.,-1);
\draw[-,line width=1.5pt] (0.,1) to (0,1.2);
\draw[-,thick] (0,-1) to [out=right,in=down] (0.25,-.75);
\draw[-,thick] (0.25,-.75) to [out=45,in=down] (0.5,-.5);
\draw[-,thick] (0.25,-.75) to [out=135,in=down] (0,-.5);
\draw[-,thick] (0,1) to [out=right,in=up] (0.25,.75);
\draw[-,thick] (0.25,.75) to [out=135,in=up] (0.5,.5);
\draw[-,thick] (0.25,.75) to [out=45,in=up] (0,.5);
\draw[-,thick] (0,-1) to [out=left,in=down] (-0.5,0)to [out=up,in=left](0,1);
\draw[-,thick](0,-.5) to (.5,.5);
\draw[-,thick](0,.5) to (.5,-.5);
\draw(-.5,0) \bdot;
\draw(.4,0.35) \bdot;
\node at (.6,.3) {$\scriptstyle g$};
\draw(-1.3,0)node {$\scriptstyle \omega_{r-1-i}$};
\node at (0.5,-.7) {$\scriptstyle i$};
\node at (-.1,-.7) {$\scriptstyle 1$};
\end{tikzpicture}\\
&
\overset{\eqref{dotmovecrossing}}{=}
\sum_{0\le i\le r-1}(-1)^i (r-1)! a_i \:\:
\left( 
\begin{tikzpicture}[baseline = -1mm,scale=.7,color=\clr]
\draw[-,line width=1.5pt] (0,-1.2) to (0.,-1);
\draw[-,line width=1.5pt] (0.,1) to (0,1.2);
\draw[-,thick] (0,-1) to [out=right,in=down] (0.25,-.75);
\draw[-,thick] (0.25,-.75) to [out=45,in=down] (0.5,-.5);
\draw[-,thick] (0.25,-.75) to [out=135,in=down] (0,-.5);
\draw[-,thick] (0,1) to [out=right,in=up] (0.25,.75);
\draw[-,thick] (0.25,.75) to [out=135,in=up] (0.5,.5);
\draw[-,thick] (0.25,.75) to [out=45,in=up] (0,.5);
\draw[-,thick] (0,-1) to [out=left,in=down] (-0.5,0)to [out=up,in=left](0,1);
\draw[-,thick](0,-.5) to (.5,.5);
\draw[-,thick](0,.5) to (.5,-.5);
\draw(-.5,0) \bdot;
\draw(0.05,-.4) \bdot;
\node at (-.2,-.35) {$\scriptstyle g$};
\draw(-1.3,0)node {$\scriptstyle \omega_{r-1-i}$};
\node at (0.5,-.7) {$\scriptstyle i$};
\node at (-.1,-.7) {$\scriptstyle 1$};
\end{tikzpicture}
-
\begin{tikzpicture}[baseline = -1mm,scale=.7,color=\clr]
\draw[-,line width=1.5pt] (0,-1.2) to (0.,-1);
\draw[-,line width=1.5pt] (0.,1) to (0,1.2);
\draw[-,thick] (0,-1) to [out=right,in=down] (0.25,-.75);
\draw[-,thick] (0.25,-.75) to [out=45,in=down] (0.5,-.5);
\draw[-,thick] (0.25,-.75) to [out=135,in=down] (0,-.5);
\draw[-,thick] (0,1) to [out=right,in=up] (0.25,.75);
\draw[-,thick] (0.25,.75) to [out=135,in=up] (0.5,.5);
\draw[-,thick] (0.25,.75) to [out=45,in=up] (0,.5);
\draw[-,thick] (0,-1) to [out=left,in=down] (-0.5,0)to [out=up,in=left](0,1);
%\draw[-,thick](0,-.5) to (.5,.5);
\draw[-,thick](0,.5) to (.5,-.5);
\draw[-,thick](0,-.5) to (0,.5);
\draw[-,thick](0.5,.5) to (.5,-.5);
\draw(-.5,0) \bdot;
%\draw(0.05,-.4) \bdot;
%\node at (-.2,-.35) {$\scriptstyle g$};
\draw(-1.3,0)node {$\scriptstyle \omega_{r-1-i}$};
\node at (0.5,-.7) {$\scriptstyle i$};
\node at (-.1,-.7) {$\scriptstyle 1$};
\node at (.6,.5) {$\scriptstyle 1$};
\end{tikzpicture}
\right)\\
&
\overset{\eqref{webassoc}}{\underset{\eqref{swallows}\eqref{splitmerge}}{=}}
\sum_{0\le i\le r-1}(-1)^i (r-1)! a_i \:\:
\left( 
\begin{tikzpicture}[baseline = -1mm,scale=.8,color=\clr]
\draw[-,line width=1.5pt] (0,-.8) to (0.,-.5);
\draw[-,line width=1.5pt] (0.,.5) to (0,.8);
\draw[-,thick] (0,-.5) to  (0.,.5);
\draw[-,thick] (0,-.5) to [out=right,in=down] (0.5,0)to [out=up,in=right](0,.5);
\draw[-,thick] (0,-.5) to [out=left,in=down] (-0.5,0)to [out=up,in=left](0,.5);
\draw(-.5,0) \bdot;
\draw(0,0) \bdot;
\node at (0.2,0) {$\scriptstyle g$};
\draw(-1.3,0)node {$\scriptstyle \omega_{r-1-i}$};
\node at (.4,-.45) {$\scriptstyle i$};
%\node at (0.1,-.2) {$\scriptstyle 1$};
\end{tikzpicture}
-
i(i+1)~~\:\:
 \begin{tikzpicture}[baseline = -1mm,scale=.8,color=\clr]
\draw[-,line width=1.5pt] (0.08,-.6) to (0.08,.5);
\node at (.08,-.8) {$\scriptstyle r$};
\draw(0.08,0) \bdot;
\draw(.9,0)node {$\scriptstyle \omega_{r-1-i}$};
\end{tikzpicture}
\right)
\\
&\overset{\eqref{dotmovesplits+merge-simplify}}{\underset{\eqref{splitmerge}}{=}}
\sum_{0\le i\le r-1}(-1)^i (r-1)! a_i \:\:
\left( 
\begin{tikzpicture}[baseline = -1mm,scale=.8,color=\clr]
\draw[-,line width=1.5pt] (0,-.8) to (0.,-.3);
\draw[-,line width=1.5pt] (0.,.5) to (0,1);
%\draw[-,thick] (0,-.5) to  (0.,.5);
\draw[-,thick] (0,-.5) to [out=right,in=down] (0.6,0.1)to [out=up,in=right](0,.85);
\draw[-,line width=1pt] (0,-.3) to [out=left,in=down] (-0.4,.1)to [out=up,in=left](0,.5);
\draw[-,thick] (0,-.3) to [out=right,in=down] (0.4,0.1)to [out=up,in=right](0,.5);
\draw(0,.7) \bdot;
%\draw(0,0) \bdot;
%\node at (0.2,0) {$\scriptstyle g$};
\draw(-.6,0.7)node {$\scriptstyle \omega_{r-i}$};
\node at (.4,-.45) {$\scriptstyle i$};
\node at (-1,-.2) {$\scriptstyle r-1-i$};
\node at (0,-.95) {$\scriptstyle r$};
\end{tikzpicture}
-
(i+1)(u+i)~~\:\:
 \begin{tikzpicture}[baseline = -1mm,scale=.8,color=\clr]
\draw[-,line width=1.5pt] (0.08,-.6) to (0.08,.5);
\node at (.08,-.8) {$\scriptstyle r$};
\draw(0.08,0) \bdot;
\draw(.9,0)node {$\scriptstyle \omega_{r-1-i}$};
\end{tikzpicture}
\right)\\
&\overset{\eqref{equ:r=0}}{\underset{\eqref{splitmerge}}{=}}
\sum_{0\le i\le r-1}(-1)^i (r-1)! a_i \:\:
\left( 
(r-i)~
\begin{tikzpicture}[baseline = -1mm,scale=.8,color=\clr]
\draw[-,line width=1.5pt] (0.08,-.6) to (0.08,.5);
\node at (.08,-.8) {$\scriptstyle r$};
\draw(0.08,0) \bdot;
\draw(.7,0)node {$\scriptstyle \omega_{r-i}$};
\end{tikzpicture}
-
(i+1)(u+i)~~\:\:
 \begin{tikzpicture}[baseline = -1mm,scale=.8,color=\clr]
\draw[-,line width=1.5pt] (0.08,-.6) to (0.08,.5);
\node at (.08,-.8) {$\scriptstyle r$};
\draw(0.08,0) \bdot;
\draw(.9,0)node {$\scriptstyle \omega_{r-1-i}$};
\end{tikzpicture}
\right)\\
&=
\sum_{0\le i\le r}(-1)^i r! a_i \:\: 
\begin{tikzpicture}[baseline = -1mm,scale=.8,color=\clr]
\draw[-,line width=1.5pt] (0.08,-.6) to (0.08,.5);
\node at (.08,-.8) {$\scriptstyle r$};
\draw(0.08,0) \bdot;
\draw(.7,0)node {$\scriptstyle \omega_{r-i}$};
\end{tikzpicture}= r! \, g_{r,u}.
\end{align*}
The lemma is proved.
\end{proof}

\subsection{Relations modulo lower degree terms}

By assigning degree $r$ to $\wkdota$ and $0$ to other generating morphisms, we define the degree of any dotted web diagram $D$ additively by summing up degrees of local diagrams of $D$. Let $\Hom_{\AW}(\lambda,\mu)_{\le k}$ (and respectively, $\Hom_{\AW}(\lambda,\mu)_{< k}$) be the $\kk$-span of all dotted web diagram of type $\lambda\rightarrow \mu$ with degree $\le k$ (and respectively, $<k$). For any two dotted web diagrams $D$ and $D'$  in  $\Hom_{\AW}(\lambda,\mu)_{\le k}$, we write 
\begin{align}  \label{modLOT}
D\equiv D' \quad \text{ if } D=D'  \mod \Hom_{\AW}(\lambda,\mu)_{< k}.
\end{align}

It follows from relations \eqref{dotmovecrossing}, \eqref{dotmovesplitss} and \eqref{dotmovemergess} that we may move dots through crossings, splits and merges freely up to lower degree terms, as formulated in the following lemma. This fact will be used frequently without any explicit reference.

\begin{lemma} In $\AW$, the following $\equiv$-relations hold:
\label{dotmovefreely}
 \begin{enumerate}
  \item 
    $\begin{tikzpicture}[baseline = 7.5pt, scale=0.4, color=\clr]
\draw[-,line width=1.2pt] (0,-.2) to  (1,2.2);
\draw[-,line width=1.2pt] (1,-.2) to  (0,2.2);
\draw(0.2,1.6)\bdot;
\draw(-.4,1.6) node {$\scriptstyle \omega_r$};
\node at (0, -.5) {$\scriptstyle a$};
\node at (1, -.5) {$\scriptstyle b$};
\end{tikzpicture}
~\equiv~ 
\begin{tikzpicture}[baseline = 7.5pt, scale=0.4, color=\clr]
\draw[-, line width=1.2pt] (0,-.2) to (1,2.2);
\draw[-,line width=1.2pt] (1,-.2) to (0,2.2);
\draw(.8,0.3)\bdot;
\draw(1.5,0.3) node {$\scriptstyle \omega_r$};
\node at (0, -.5) {$\scriptstyle a$};
\node at (1, -.5) {$\scriptstyle b$};
\end{tikzpicture}, 
\quad 
\begin{tikzpicture}[baseline = 7.5pt, scale=0.4, color=\clr]
\draw[-, line width=1.2pt] (0,-.2) to (1,2.2);
\draw[-,line width=1.2pt] (1,-.2) to(0,2.2);
\draw(.2,0.2)\bdot;
\draw(-.4,.2)node {$\scriptstyle \omega_r$};
\node at (0, -.5) {$\scriptstyle b$};
\node at (1, -.5) {$\scriptstyle a$};
\end{tikzpicture}
~\equiv~
\begin{tikzpicture}[baseline = 7.5pt, scale=0.4, color=\clr]
\draw[-,line width=1.2pt] (0,-.2) to  (1,2.2);
\draw[-,line width=1.2pt] (1,-.2) to  (0,2.2);
\draw(0.8,1.6)\bdot;
\draw(1.5,1.6)node {$\scriptstyle \omega_r$};
\node at (0, -.5) {$\scriptstyle b$};
\node at (1, -.5) {$\scriptstyle a$};
\end{tikzpicture}
$;
\item
$
\begin{tikzpicture}[baseline = -.5mm,scale=.85,color=\clr]
\draw[-,line width=1.5pt] (0.08,-.5) to (0.08,0.04);
\draw[-,line width=1pt] (0.34,.5) to (0.08,0);
\draw[-,line width=1pt] (-0.2,.5) to (0.08,0);
\node at (-0.22,.6) {$\scriptstyle a$};
\node at (0.36,.65) {$\scriptstyle b$};
\draw (0.08,-.2) \bdot;
\draw (0.45,-.2) node{$\scriptstyle \omega_r$};
\end{tikzpicture} 
\equiv~
\sum\limits_{c+d=r} ~
\begin{tikzpicture}[baseline = -.5mm,scale=.85,color=\clr]
\draw[-,line width=1.5pt] (0.08,-.5) to (0.08,0.04);
\draw[-,line width=1pt] (0.34,.5) to (0.08,0);
\draw[-,line width=1pt] (-0.2,.5) to (0.08,0);
\node at (-0.22,.6) {$\scriptstyle a$};
\node at (0.36,.65) {$\scriptstyle b$};
\draw (-.05,.24) \bdot;
\draw (-0.4,.2) node{$\scriptstyle \omega_{c}$};
\draw (0.6,.2) node{$\scriptstyle \omega_{d}$};
\draw (.22,.24) \bdot;
\end{tikzpicture}~
$;
\item 
$
\begin{tikzpicture}[baseline = -.5mm, scale=.85, color=\clr]
\draw[-,line width=1pt] (0.3,-.5) to (0.08,0.04);
\draw[-,line width=1pt] (-0.2,-.5) to (0.08,0.04);
\draw[-,line width=1.5pt] (0.08,.6) to (0.08,0);
\node at (-0.22,-.6) {$\scriptstyle a$};
\node at (0.35,-.6) {$\scriptstyle b$};
\draw (0.08,.2) \bdot;
\draw (0.45,.2) node{$\scriptstyle \omega_r$};  
\end{tikzpicture}
 ~\equiv~
\sum\limits_{c+d=r}~
\begin{tikzpicture}[baseline = -.5mm,scale=.85, color=\clr]
\draw[-,line width=1pt] (0.3,-.5) to (0.08,0.04);
\draw[-,line width=1pt] (-0.2,-.5) to (0.08,0.04);
\draw[-,line width=1.5pt] (0.08,.6) to (0.08,0);
\node at (-0.22,-.6) {$\scriptstyle a$};
\node at (0.35,-.6) {$\scriptstyle b$};
\draw (-.08,-.3) \bdot; \draw (.22,-.3) \bdot;
\draw (-0.4,-.3) node{$\scriptstyle \omega_{c}$};
\draw (0.6,-.3) node{$\scriptstyle \omega_{d}$};
\end{tikzpicture}
$. 
\end{enumerate}  
\end{lemma}

For any two strict  partitions  $\lambda=(\lambda_1,\ldots, \lambda_k)$ and $\mu=(\mu_1,\ldots,\mu_l)$ with  $1\le \lambda_i\le a$ and $ 1\le \mu_j \le b$, 
let $b_{\lambda,\mu}\in \End_{\AW}(a+b)$ be given as 
\begin{equation}
\label{blambdamu}
\begin{tikzpicture}[baseline = 5pt, scale=.7, color=\clr] 
\node at (-2,1){$ b_{\lambda,\mu}:=$};
\draw[-, line width=1.5pt] (0.5,2) to (0.5,2.5);
\draw[-, line width=1.5pt] (0.5,0) to (0.5,-.5);
\draw[-,line width=1pt]  (0.5,2) to[out=left,in=up] (0,1)to[out=down,in=left] (0.5,0);      
\draw[-,line width=1pt] (0.5,0)to[out=right,in=down] (1,1)to[out=up,in=right] (0.5,2);
\node at (-.2,.9){$\scriptstyle \vdots$};
\draw (0.1,1.7) \bdot;
\node at (-.4,1.7) {$\scriptstyle \omega_{\lambda_1}$}; 
\draw (0,1.2) \bdot; 
\node at (-0.45,1.2) {$\scriptstyle \omega_{\lambda_2}$};
\draw (0.05,.4) \bdot; 
\node at (-0.45,.4) {$\scriptstyle \omega_{\lambda_k}$};
\draw (1,1.2) \bdot;
\draw (.95,.4) \bdot; 
\node at (1.5,.4) {$\scriptstyle \omega_{\mu_l}$};
\draw (.9,1.7) \bdot; 
\node at (1.45,1.7) {$\scriptstyle \omega_{\mu_1}$};
\node at (1.5,1.2) {$\scriptstyle \omega_{\mu_2}$};
\node at (1.2,.9){$\scriptstyle \vdots$};
\node at (0.2,-.1) {$\scriptstyle a$};
\node at (.9,-0.1) {$\scriptstyle b$};
\end{tikzpicture}.  
\end{equation}
Note the degree 
$\deg (b_{\lambda,\mu})= |\lambda|+|\mu|$.
\begin{lemma}
\label{wkdotsspan}
For any $\lambda$ and $\mu$, we have $b_{\lambda,\mu}\in D_{a+b}$.
    \end{lemma}

\begin{proof}
We will proceed by using four layers of induction. The first is induction on the total degree $|\lambda|+|\mu|$ of \eqref{blambdamu}.
Using Lemma \ref{dotmovefreely} (3) repeatedly we may move the dots from the right of the balloon to the left up to moving them outside the balloon (modulo lower degree terms), reducing to the case for $\mu= \emptyset$. Thus, by induction on degree $|\lambda|+|\mu|$, it suffices to prove that $b_{\lambda,\emptyset}\in D_{a+b}$ for all $\lambda$. 

Keeping the simple induction on the total degree in the background (which allows us to use $\equiv$ modulo lower terms whenever it is needed), we now proceed by a second induction on the thickness $b$, and then further apply double induction on $k$ (\# dot generators) and on $\lambda_k$ for a given $b$.

Suppose $b=1$ first. We proceed by double induction on $k$ and on $\lambda_k$. If $k=1$, then by \eqref{splitmerge},
\[
b_{\lambda,\emptyset} =(a+1-\lambda_1)
\begin{tikzpicture}[baseline = -1mm,scale=.8,color=\clr]
\draw[-,line width=1.5pt] (0.08,-.7) to (0.08,.5);
\node at (.08,-.8) {$\scriptstyle a+1$};
\draw(0.08,0) \bdot;
\draw(.5,0)node {$\scriptstyle \omega_{\lambda_1}$};
\end{tikzpicture}
\in D_{a+1}.
\]
Suppose $k\ge 2$, first with $\lambda_k=1$. Using Lemma \ref{dotmovefreely}(2)
we have 
\begin{align}
\label{b=1lambdak=1}
b_{\lambda,\emptyset}= 
\begin{tikzpicture}[baseline = 5pt, scale=.7, color=\clr] 
\draw[-, line width=1.5pt] (0.5,1.5) to (0.5,1.8);
\draw[-, line width=1.5pt] (0.5,-.5) to (0.5,-1);
\draw[-,line width=1pt]  (0.5,1.5) to[out=left,in=up] (0,0.5)to[out=down,in=left] (0.5,-.5);      
\draw[-,line width=1pt] (0.5,-.5)to[out=right,in=down] (1,0.5)to[out=up,in=right] (0.5,1.5);
\node at (-.2,.4){$\scriptstyle \vdots$};
\draw (0.1,1.2) \bdot;
\node at (-.4,1.2) {$\scriptstyle \omega_{\lambda_1}$}; 
\draw (0,.7) \bdot; 
\node at (-0.45,.7) {$\scriptstyle \omega_{\lambda_2}$};
\draw (0.05,-.1) \bdot; 
\node at (-0.45,-.1) {$\scriptstyle \omega_{\lambda_k}$};
\node at (0.2,-.6) {$\scriptstyle a$};
\node at (.9,-0.6) {$\scriptstyle 1$};
\end{tikzpicture}
~\equiv~
\begin{tikzpicture}[baseline = 5pt, scale=.7, color=\clr] 
\draw[-, line width=1.5pt] (0.5,1.5) to (0.5,1.8);
\draw[-, line width=1.5pt] (0.5,-.5) to (0.5,-1.2);
\draw[-,line width=1pt]  (0.5,1.5) to[out=left,in=up] (0,0.5)to[out=down,in=left] (0.5,-.5);      
\draw[-,line width=1pt] (0.5,-.5)to[out=right,in=down] (1,0.5)to[out=up,in=right] (0.5,1.5);
\node at (-.2,.4){$\scriptstyle \vdots$};
\draw (0.1,1.2) \bdot;
\node at (-.4,1.2) {$\scriptstyle \omega_{\lambda_1}$}; 
\draw (0,.7) \bdot; 
\node at (-0.45,.7) {$\scriptstyle \omega_{\lambda_2}$};
\draw (0.05,-.1) \bdot; 
\node at (-0.6,-.1) {$\scriptstyle \omega_{\lambda_{k-1}}$};
\draw (0.5,-.8) \bdot; 
\node at (1.1,-.8) {$\scriptstyle \omega_{\lambda_k}$};
\node at (0.2,-.6) {$\scriptstyle a$};
%\node at (.9,-0.6) {$\scriptstyle 1$};
\end{tikzpicture}
-
\begin{tikzpicture}[baseline = 5pt, scale=.7, color=\clr] 
\draw[-, line width=1.5pt] (0.5,1.5) to (0.5,1.8);
\draw[-, line width=1.5pt] (0.5,-.5) to (0.5,-1);
\draw[-,line width=1pt]  (0.5,1.5) to[out=left,in=up] (0,0.5)to[out=down,in=left] (0.5,-.5);      
\draw[-,line width=1pt] (0.5,-.5)to[out=right,in=down] (1,0.5)to[out=up,in=right] (0.5,1.5);
\node at (-.2,.4){$\scriptstyle \vdots$};
\draw (0.1,1.2) \bdot;
\node at (-.4,1.2) {$\scriptstyle \omega_{\lambda_1}$}; 
\draw (0,.7) \bdot; 
\node at (-0.5,.7) {$\scriptstyle \omega_{\lambda_2}$};
\draw (0.05,-.1) \bdot; 
\node at (-0.6,-.1) {$\scriptstyle \omega_{\lambda_{k-1}}$};
\draw (1,.7) \bdot; 
\node at (1.4,.7) {$\scriptstyle \omega_{1}$};
\node at (0.2,-.6) {$\scriptstyle a$};
%\node at (.9,-0.6) {$\scriptstyle 1$};
\end{tikzpicture}.
\end{align}
The first summand is already in $D_{a+1}$ by the inductive assumption on $k-1$. The second summand can be handled using  Lemma \ref{dotmovefreely}(2) again:
\begin{align}
\label{b=1}
\begin{tikzpicture}[baseline = 5pt, scale=.7, color=\clr] 
\draw[-, line width=1.5pt] (0.5,1.5) to (0.5,1.8);
\draw[-, line width=1.5pt] (0.5,-.5) to (0.5,-1);
\draw[-,line width=1pt]  (0.5,1.5) to[out=left,in=up] (0,0.5)to[out=down,in=left] (0.5,-.5);      
\draw[-,line width=1pt] (0.5,-.5)to[out=right,in=down] (1,0.5)to[out=up,in=right] (0.5,1.5);
\node at (-.2,.4){$\scriptstyle \vdots$};
\draw (0.1,1.2) \bdot;
\node at (-.4,1.2) {$\scriptstyle \omega_{\lambda_1}$}; 
\draw (0,.7) \bdot; 
\node at (-0.5,.7) {$\scriptstyle \omega_{\lambda_2}$};
\draw (0.05,-.1) \bdot; 
\node at (-0.6,-.1) {$\scriptstyle \omega_{\lambda_{k-1}}$};
\draw (1,.7) \bdot; 
\node at (1.4,.7) {$\scriptstyle \omega_{1}$};
\node at (0.2,-.6) {$\scriptstyle a$};
%\node at (.9,-0.6) {$\scriptstyle 1$};
\end{tikzpicture}
\equiv
\begin{tikzpicture}[baseline = 5pt, scale=.7, color=\clr] 
\draw[-, line width=1.5pt] (0.5,1.5) to (0.5,1.8);
\draw[-, line width=1.5pt] (0.5,-.5) to (0.5,-1.2);
\draw[-,line width=1pt]  (0.5,1.5) to[out=left,in=up] (0,0.5)to[out=down,in=left] (0.5,-.5);      
\draw[-,line width=1pt] (0.5,-.5)to[out=right,in=down] (1,0.5)to[out=up,in=right] (0.5,1.5);
\node at (-.2,.4){$\scriptstyle \vdots$};
\draw (0.1,1.2) \bdot;
\node at (-.4,1.2) {$\scriptstyle \omega_{\lambda_1}$}; 
\draw (0,.7) \bdot; 
\node at (-0.5,.7) {$\scriptstyle \omega_{\lambda_2}$};
\draw (0.05,-.1) \bdot; 
\node at (-0.65,-.1) {$\scriptstyle \omega_{\lambda_{k-2}}$};
\draw (0.5,-.8) \bdot; 
\node at (1.5,-.8) {$\scriptstyle \omega_{\lambda_{k-1}+1}$};
\node at (0.2,-.6) {$\scriptstyle a$};
%\node at (.9,-0.6) {$\scriptstyle 1$};
\end{tikzpicture}
-
\begin{tikzpicture}[baseline = 5pt, scale=.7, color=\clr] 
\draw[-, line width=1.5pt] (0.5,1.5) to (0.5,1.8);
\draw[-, line width=1.5pt] (0.5,-.5) to (0.5,-1);
\draw[-,line width=1pt]  (0.5,1.5) to[out=left,in=up] (0,0.5)to[out=down,in=left] (0.5,-.5);      
\draw[-,line width=1pt] (0.5,-.5)to[out=right,in=down] (1,0.5)to[out=up,in=right] (0.5,1.5);
\node at (-.2,.4){$\scriptstyle \vdots$};
\draw (0.1,1.2) \bdot;
\node at (-.4,1.2) {$\scriptstyle \omega_{\lambda_1}$}; 
\draw (0,.7) \bdot; 
\node at (-0.5,.7) {$\scriptstyle \omega_{\lambda_2}$};
\draw (0.05,-.1) \bdot; 
\node at (-.9,-.1) {$\scriptstyle \omega_{\lambda_{k-1}+1}$};
\node at (0.2,-.6) {$\scriptstyle a$};
%\node at (.9,-0.6) {$\scriptstyle 1$};
\end{tikzpicture} \in D_{a+1}
\end{align}
 by the induction hypothesis  on $< k$ dot generators. This settles the case for $\lambda_k=1$.
Now let $\lambda_k\ge 2$.
Then we have 
 \[
b_{\lambda,\emptyset}= 
\begin{tikzpicture}[baseline = 5pt, scale=.7, color=\clr] 
\draw[-, line width=1.5pt] (0.5,1.5) to (0.5,1.8);
\draw[-, line width=1.5pt] (0.5,-.5) to (0.5,-1);
\draw[-,line width=1pt]  (0.5,1.5) to[out=left,in=up] (0,0.5)to[out=down,in=left] (0.5,-.5);      
\draw[-,line width=1pt] (0.5,-.5)to[out=right,in=down] (1,0.5)to[out=up,in=right] (0.5,1.5);
\node at (-.2,.4){$\scriptstyle \vdots$};
\draw (0.1,1.2) \bdot;
\node at (-.4,1.2) {$\scriptstyle \omega_{\lambda_1}$}; 
\draw (0,.7) \bdot; 
\node at (-0.5,.7) {$\scriptstyle \omega_{\lambda_2}$};
\draw (0.05,-.1) \bdot; 
\node at (-0.5,-.1) {$\scriptstyle \omega_{\lambda_k}$};
\node at (0.2,-.6) {$\scriptstyle a$};
\node at (.9,-0.6) {$\scriptstyle 1$};
\end{tikzpicture}
~\equiv~
\begin{tikzpicture}[baseline = 5pt, scale=.7, color=\clr] 
\draw[-, line width=1.5pt] (0.5,1.5) to (0.5,1.8);
\draw[-, line width=1.5pt] (0.5,-.5) to (0.5,-1.2);
\draw[-,line width=1pt]  (0.5,1.5) to[out=left,in=up] (0,0.5)to[out=down,in=left] (0.5,-.5);      
\draw[-,line width=1pt] (0.5,-.5)to[out=right,in=down] (1,0.5)to[out=up,in=right] (0.5,1.5);
\node at (-.2,.4){$\scriptstyle \vdots$};
\draw (0.1,1.2) \bdot;
\node at (-.4,1.2) {$\scriptstyle \omega_{\lambda_1}$}; 
\draw (0,.7) \bdot; 
\node at (-0.5,.7) {$\scriptstyle \omega_{\lambda_2}$};
\draw (0.05,-.1) \bdot; 
\node at (-0.6,-.1) {$\scriptstyle \omega_{\lambda_{k-1}}$};
\draw (0.5,-.8) \bdot; 
\node at (1.1,-.8) {$\scriptstyle \omega_{\lambda_k}$};
\node at (0.2,-.6) {$\scriptstyle a$};
%\node at (.9,-0.6) {$\scriptstyle 1$};
\end{tikzpicture}
-
\begin{tikzpicture}[baseline = 5pt, scale=.7, color=\clr] 
\draw[-, line width=1.5pt] (0.5,1.5) to (0.5,1.8);
\draw[-, line width=1.5pt] (0.5,-.5) to (0.5,-1);
\draw[-,line width=1pt]  (0.5,1.5) to[out=left,in=up] (0,0.5)to[out=down,in=left] (0.5,-.5);      
\draw[-,line width=1pt] (0.5,-.5)to[out=right,in=down] (1,0.5)to[out=up,in=right] (0.5,1.5);
\node at (-.2,.4){$\scriptstyle \vdots$};
\draw (0.1,1.2) \bdot;
\node at (-.4,1.2) {$\scriptstyle \omega_{\lambda_1}$}; 
\draw (0,.7) \bdot; 
\node at (-0.5,.7) {$\scriptstyle \omega_{\lambda_2}$};
\draw (0.05,0) \bdot; 
\node at (-0.6,0) {$\scriptstyle \omega_{\lambda_{k-1}}$};
\draw (0.15,-.3) \bdot; 
\node at (-0.6,-.4) {$\scriptstyle \omega_{\lambda_{k}-1}$};
\draw (1,.7) \bdot; 
\node at (1.4,.7) {$\scriptstyle \omega_{1}$};

\end{tikzpicture}.
\]
The first summand is in $D_{a+1}$ by inductive assumption on $k-1$ dot generators. For the second summand, we have 
\begin{align*}
\begin{tikzpicture}[baseline = 5pt, scale=.7, color=\clr] 
\draw[-, line width=1.5pt] (0.5,1.5) to (0.5,1.8);
\draw[-, line width=1.5pt] (0.5,-.5) to (0.5,-1);
\draw[-,line width=1pt]  (0.5,1.5) to[out=left,in=up] (0,0.5)to[out=down,in=left] (0.5,-.5);      
\draw[-,line width=1pt] (0.5,-.5)to[out=right,in=down] (1,0.5)to[out=up,in=right] (0.5,1.5);
\node at (-.2,.4){$\scriptstyle \vdots$};
\draw (0.1,1.2) \bdot;
\node at (-.4,1.2) {$\scriptstyle \omega_{\lambda_1}$}; 
\draw (0,.7) \bdot; 
\node at (-0.5,.7) {$\scriptstyle \omega_{\lambda_2}$};
\draw (0.05,0) \bdot; 
\node at (-0.6,0) {$\scriptstyle \omega_{\lambda_{k-1}}$};
\draw (0.15,-.3) \bdot; 
\node at (-0.6,-.4) {$\scriptstyle \omega_{\lambda_{k}-1}$};
\draw (1,.7) \bdot; 
\node at (1.4,.7) {$\scriptstyle \omega_{1}$};
\end{tikzpicture}
\equiv
\begin{tikzpicture}[baseline = 5pt, scale=.7, color=\clr] 
\draw[-, line width=1.5pt] (0.5,1.5) to (0.5,1.8);
\draw[-, line width=1.5pt] (0.5,-.5) to (0.5,-1.2);
\draw[-,line width=1pt]  (0.5,1.5) to[out=left,in=up] (0,0.5)to[out=down,in=left] (0.5,-.5);      
\draw[-,line width=1pt] (0.5,-.5)to[out=right,in=down] (1,0.5)to[out=up,in=right] (0.5,1.5);
\node at (-.2,.4){$\scriptstyle \vdots$};
\draw (0.1,1.2) \bdot;
\node at (-.4,1.2) {$\scriptstyle \omega_{\lambda_1}$}; 
\draw (0,.7) \bdot; 
\node at (-0.5,.7) {$\scriptstyle \omega_{\lambda_2}$};
\draw (0.05,0) \bdot; 
\node at (-0.65,0) {$\scriptstyle \omega_{\lambda_{k-2}}$};
\draw (0.5,-.8) \bdot; 
\node at (1.5,-.8) {$\scriptstyle \omega_{\lambda_{k-1}+1}$};
\draw (0.15,-.3) \bdot; 
\node at (-0.6,-.4) {$\scriptstyle \omega_{\lambda_{k}-1}$};
%\draw (1,.7) \bdot;
%\node at (0.2,-.6) {$\scriptstyle a$};
%\node at (.9,-0.6) {$\scriptstyle 1$};
\end{tikzpicture}
-
\begin{tikzpicture}[baseline = 5pt, scale=.7, color=\clr] 
\draw[-, line width=1.5pt] (0.5,1.5) to (0.5,1.8);
\draw[-, line width=1.5pt] (0.5,-.5) to (0.5,-1);
\draw[-,line width=1pt]  (0.5,1.5) to[out=left,in=up] (0,0.5)to[out=down,in=left] (0.5,-.5);      
\draw[-,line width=1pt] (0.5,-.5)to[out=right,in=down] (1,0.5)to[out=up,in=right] (0.5,1.5);
\node at (-.2,.4){$\scriptstyle \vdots$};
\draw (0.1,1.2) \bdot;
\node at (-.4,1.2) {$\scriptstyle \omega_{\lambda_1}$}; 
\draw (0,.7) \bdot; 
\node at (-0.5,.7) {$\scriptstyle \omega_{\lambda_2}$};
\draw (0.05,0) \bdot; 
\node at (-.8,0) {$\scriptstyle \omega_{\lambda_{k-1}+1}$};
\draw (0.15,-.3) \bdot; 
\node at (-0.6,-.4) {$\scriptstyle \omega_{\lambda_{k}-1}$};
%\node at (0.2,-.6) {$\scriptstyle a$};
%\node at (.9,-0.6) {$\scriptstyle 1$};
\end{tikzpicture} \in D_{a+1}
\end{align*}
since the first (and respectively, second) summand of the RHS of the above equation is in $D_{a+1}$ by the inductive assumption on $k-1$ dot generators (and respectively, on $\lambda_{k}-1$). 
 This completes the proof when $b=1$.

 Suppose that $b\ge 2$. We proceed by induction on $k$ and on $\lambda_k$ as above. The case $k=1$ again is clear by \eqref{splitmerge}. The case $k\ge 2$ with $\lambda_k=1$ can be argued in the same way as in \eqref{b=1lambdak=1}--\eqref{b=1}. Now let $\lambda_k\ge 2$. Using Lemma \ref{dotmovefreely}(2) again we have 
\[
b_{\lambda,\emptyset}= 
\begin{tikzpicture}[baseline = 5pt, scale=.7, color=\clr] 
\draw[-, line width=1.5pt] (0.5,1.5) to (0.5,1.8);
\draw[-, line width=1.5pt] (0.5,-.5) to (0.5,-1);
\draw[-,line width=1pt]  (0.5,1.5) to[out=left,in=up] (0,0.5)to[out=down,in=left] (0.5,-.5);      
\draw[-,line width=1pt] (0.5,-.5)to[out=right,in=down] (1,0.5)to[out=up,in=right] (0.5,1.5);
\node at (-.2,.4){$\scriptstyle \vdots$};
\draw (0.1,1.2) \bdot;
\node at (-.4,1.2) {$\scriptstyle \omega_{\lambda_1}$}; 
\draw (0,.7) \bdot; 
\node at (-0.5,.7) {$\scriptstyle \omega_{\lambda_2}$};
\draw (0.05,-.1) \bdot; 
\node at (-0.5,-.1) {$\scriptstyle \omega_{\lambda_k}$};
\node at (0.2,-.6) {$\scriptstyle a$};
\node at (.9,-0.6) {$\scriptstyle b$};
\end{tikzpicture}
~\equiv~
\begin{tikzpicture}[baseline = 5pt, scale=.7, color=\clr] 
\draw[-, line width=1.5pt] (0.5,1.5) to (0.5,1.8);
\draw[-, line width=1.5pt] (0.5,-.5) to (0.5,-1.2);
\draw[-,line width=1pt]  (0.5,1.5) to[out=left,in=up] (0,0.5)to[out=down,in=left] (0.5,-.5);      
\draw[-,line width=1pt] (0.5,-.5)to[out=right,in=down] (1,0.5)to[out=up,in=right] (0.5,1.5);
\node at (-.2,.4){$\scriptstyle \vdots$};
\draw (0.1,1.2) \bdot;
\node at (-.4,1.2) {$\scriptstyle \omega_{\lambda_1}$}; 
\draw (0,.7) \bdot; 
\node at (-0.5,.7) {$\scriptstyle \omega_{\lambda_2}$};
\draw (0.05,-.1) \bdot; 
\node at (-0.6,-.1) {$\scriptstyle \omega_{\lambda_{k-1}}$};
\draw (0.5,-.8) \bdot; 
\node at (1.1,-.8) {$\scriptstyle \omega_{\lambda_k}$};
\node at (0.2,-.6) {$\scriptstyle a$};
\node at (1,-0.3) {$\scriptstyle b$};
\end{tikzpicture}
-\sum_{1\le r\le \min \{b,\lambda_k\}}
\begin{tikzpicture}[baseline = 5pt, scale=.7, color=\clr] 
\draw[-, line width=1.5pt] (0.5,1.5) to (0.5,1.8);
\draw[-, line width=1.5pt] (0.5,-.5) to (0.5,-1);
\draw[-,line width=1pt]  (0.5,1.5) to[out=left,in=up] (0,0.5)to[out=down,in=left] (0.5,-.5);      
\draw[-,line width=1pt] (0.5,-.5)to[out=right,in=down] (1,0.5)to[out=up,in=right] (0.5,1.5);
\node at (-.2,.4){$\scriptstyle \vdots$};
\draw (0.1,1.2) \bdot;
\node at (-.4,1.2) {$\scriptstyle \omega_{\lambda_1}$}; 
\draw (0,.7) \bdot; 
\node at (-0.5,.7) {$\scriptstyle \omega_{\lambda_2}$};
\draw (0.05,0) \bdot; 
\node at (-0.6,0) {$\scriptstyle \omega_{\lambda_{k-1}}$};
\draw (0.15,-.3) \bdot; 
\node at (-0.55,-.45) {$\scriptstyle \omega_{\lambda_{k}-r}$};
\draw (1,.7) \bdot; 
\node at (1.4,.7) {$\scriptstyle \omega_{r}$};
%\node at (0.2,-.6) {$\scriptstyle a$};
\node at (1.1,-0.3) {$\scriptstyle b$};
\end{tikzpicture} . 
\]
The first summand on the right-hand side above is in $D_{a+b}$ by the inductive assumption on $k-1$ dot generators. For the summand in the above summation with $1\le r<b$, we have 
\begin{equation}
\label{equ:lessb}
    \begin{tikzpicture}[baseline = 5pt, scale=.7, color=\clr] 
\draw[-, line width=1.5pt] (0.5,1.5) to (0.5,1.8);
\draw[-, line width=1.5pt] (0.5,-.5) to (0.5,-1);
\draw[-,line width=1pt]  (0.5,1.5) to[out=left,in=up] (0,0.5)to[out=down,in=left] (0.5,-.5);      
\draw[-,line width=1pt] (0.5,-.5)to[out=right,in=down] (1,0.5)to[out=up,in=right] (0.5,1.5);
\node at (-.2,.4){$\scriptstyle \vdots$};
\draw (0.1,1.2) \bdot;
\node at (-.4,1.2) {$\scriptstyle \omega_{\lambda_1}$}; 
\draw (0,.7) \bdot; 
\node at (-0.5,.7) {$\scriptstyle \omega_{\lambda_2}$};
\draw (0.05,0) \bdot; 
\node at (-0.6,0) {$\scriptstyle \omega_{\lambda_{k-1}}$};
\draw (0.15,-.3) \bdot; 
\node at (-0.55,-.45) {$\scriptstyle \omega_{\lambda_{k}-r}$};
\draw (1,.7) \bdot; 
\node at (1.4,.7) {$\scriptstyle \omega_{r}$};
%\node at (0.2,-.6) {$\scriptstyle a$};
\node at (1.1,-0.3) {$\scriptstyle b$};
\end{tikzpicture}
\overset{\eqref{splitmerge}}{=}
\begin{tikzpicture}[baseline = 5pt, scale=.7, color=\clr] 
\draw[-, line width=1.5pt] (0.5,1.5) to (0.5,1.8);
\draw[-, line width=1.5pt] (0.5,-.5) to (0.5,-1);
\draw[-,line width=1pt]  (0.5,1.5) to[out=left,in=up] (0,0.5)to[out=down,in=left] (0.5,-.5);      
\draw[-,line width=1pt] (0.5,-.5)to [out=right,in=down](.8,.5);
\draw[-,line width=1pt](0.5,1.5)to[out=right,in=up] (.8,.5);
\draw[-,line width=1pt] (0.5,-.6)to[out=right,in=down] (1.3,0.5)to[out=up,in=right] (0.5,1.6);
\node at (-.2,.4){$\scriptstyle \vdots$};
\draw (0.1,1.2) \bdot;
\node at (-.4,1.2) {$\scriptstyle \omega_{\lambda_1}$}; 
\draw (0,.7) \bdot; 
\node at (-0.5,.7) {$\scriptstyle \omega_{\lambda_2}$};
\draw (0.05,0) \bdot; 
\node at (-0.6,0) {$\scriptstyle \omega_{\lambda_{k-1}}$};
\draw (0.15,-.3) \bdot; 
\node at (-0.55,-.45) {$\scriptstyle \omega_{\lambda_{k}-r}$};
\draw (.8,.7) \bdot; 
\node at (.4,.7) {$\scriptstyle \omega_{r}$};
\node at (.6,-.1) {$\scriptstyle r$};
%\node at (0.2,-.6) {$\scriptstyle a$};
\node at (1.5,-0.3) {$\scriptstyle b-r$};
\end{tikzpicture} ,
\end{equation}
which is in $D_{a+b}$ by the induction hypothesis on the thickness $r<b$ and $b-r<b$, twice.
 For the last summand  with $r=b$ (and $b\le \lambda_k$), using Lemma \ref{dotmovefreely}(2) again we have  
\begin{align*}  
\begin{tikzpicture}[baseline = 5pt, scale=.7, color=\clr] 
\draw[-, line width=1.5pt] (0.5,1.5) to (0.5,1.8);
\draw[-, line width=1.5pt] (0.5,-.5) to (0.5,-1);
\draw[-,line width=1pt]  (0.5,1.5) to[out=left,in=up] (0,0.5)to[out=down,in=left] (0.5,-.5);      
\draw[-,line width=1pt] (0.5,-.5)to[out=right,in=down] (1,0.5)to[out=up,in=right] (0.5,1.5);
\node at (-.2,.4){$\scriptstyle \vdots$};
\draw (0.1,1.2) \bdot;
\node at (-.4,1.2) {$\scriptstyle \omega_{\lambda_1}$}; 
\draw (0,.7) \bdot; 
\node at (-0.5,.7) {$\scriptstyle \omega_{\lambda_2}$};
\draw (0.05,0) \bdot; 
\node at (-0.7,0) {$\scriptstyle \omega_{\lambda_{k-1}}$};
\draw (0.15,-.3) \bdot; 
\node at (-0.55,-.45) {$\scriptstyle \omega_{\lambda_{k}-b}$};
\draw (1,.7) \bdot; 
\node at (1.4,.7) {$\scriptstyle \omega_{b}$};
%\node at (0.2,-.6) {$\scriptstyle a$};
\node at (1.1,-0.3) {$\scriptstyle b$};
\end{tikzpicture}
~\equiv~
\begin{tikzpicture}[baseline = 5pt, scale=.7, color=\clr] 
\draw[-, line width=1.5pt] (0.5,1.5) to (0.5,1.8);
\draw[-, line width=1.5pt] (0.5,-.5) to (0.5,-1.2);
\draw[-,line width=1pt]  (0.5,1.5) to[out=left,in=up] (0,0.5)to[out=down,in=left] (0.5,-.5);      
\draw[-,line width=1pt] (0.5,-.5)to[out=right,in=down] (1,0.5)to[out=up,in=right] (0.5,1.5);
\node at (-.2,.4){$\scriptstyle \vdots$};
\draw (0.1,1.2) \bdot;
\node at (-.4,1.2) {$\scriptstyle \omega_{\lambda_1}$}; 
\draw (0,.7) \bdot; 
\node at (-0.5,.7) {$\scriptstyle \omega_{\lambda_2}$};
\draw (0.05,0) \bdot; 
\node at (-0.7,0) {$\scriptstyle \omega_{\lambda_{k-2}}$};
\draw (0.15,-.3) \bdot; 
\node at (-0.55,-.45) {$\scriptstyle \omega_{\lambda_{k}-b}$};
\draw (0.5,-.8) \bdot; 
\node at (1.5,-.8) {$\scriptstyle \omega_{\lambda_{k-1}+b}$};
%\node at (0.2,-.6) {$\scriptstyle a$};
\node at (1.1,-0.3) {$\scriptstyle b$};
\end{tikzpicture}
- 
\begin{tikzpicture}[baseline = 5pt, scale=.7, color=\clr] 
\draw[-, line width=1.5pt] (0.5,1.5) to (0.5,1.8);
\draw[-, line width=1.5pt] (0.5,-.5) to (0.5,-1.2);
\draw[-,line width=1pt]  (0.5,1.5) to[out=left,in=up] (0,0.5)to[out=down,in=left] (0.5,-.5);      
\draw[-,line width=1pt] (0.5,-.5)to[out=right,in=down] (1,0.5)to[out=up,in=right] (0.5,1.5);
\node at (-.2,.4){$\scriptstyle \vdots$};
\draw (0.1,1.2) \bdot;
\node at (-.4,1.2) {$\scriptstyle \omega_{\lambda_1}$}; 
\draw (0,.7) \bdot; 
\node at (-0.5,.7) {$\scriptstyle \omega_{\lambda_2}$};
\draw (0.05,0) \bdot; 
\node at (-0.9,0) {$\scriptstyle \omega_{\lambda_{k-1}+b}$};
\draw (0.15,-.3) \bdot; 
\node at (-0.55,-.45) {$\scriptstyle \omega_{\lambda_{k}-b}$};
%\node at (0.2,-.6) {$\scriptstyle a$};
\node at (1.1,-0.3) {$\scriptstyle b$};
\end{tikzpicture}
-
\sum_{1\le t< b}
\begin{tikzpicture}[baseline = 5pt, scale=.7, color=\clr] 
\draw[-, line width=1.5pt] (0.5,1.5) to (0.5,1.8);
\draw[-, line width=1.5pt] (0.5,-.5) to (0.5,-1);
\draw[-,line width=1pt]  (0.5,1.5) to[out=left,in=up] (0,0.5)to[out=down,in=left] (0.5,-.5);      
\draw[-,line width=1pt] (0.5,-.5)to[out=right,in=down] (1,0.5)to[out=up,in=right] (0.5,1.5);
\node at (-.2,.4){$\scriptstyle \vdots$};
\draw (0.1,1.2) \bdot;
\node at (-.4,1.2) {$\scriptstyle \omega_{\lambda_1}$}; 
\draw (0,.7) \bdot; 
\node at (-0.5,.7) {$\scriptstyle \omega_{\lambda_2}$};
\draw (0.05,0) \bdot; 
\node at (-1.1,-.1) {$\scriptstyle \omega_{\lambda_{k-1}+b-t}$};
\draw (0.15,-.35) \bdot; 
\node at (-0.55,-.55) {$\scriptstyle \omega_{\lambda_{k}-b}$};
\draw (1,.7) \bdot; 
\node at (1.4,.7) {$\scriptstyle \omega_{t}$};
%\node at (0.2,-.6) {$\scriptstyle a$};
\node at (1.1,-0.3) {$\scriptstyle b$};
\end{tikzpicture}.
\end{align*}
 Then the first two summands on the RHS of the above identity are in $D_{a+b}$ by the inductive assumption on $<k$ dot generators and on $\lambda_k-b$. Finally, each summand in the last summation of the above identity also belongs to $D_{a+b}$ by the same reasoning as for \eqref{equ:lessb} since $t<b$. This completes the proof for general $b$ and then for the lemma.
\end{proof}

\section{Basis for the affine web category}
 \label{sec:AWbasis}
In this section, we establish a basis theorem for the category $\AW$ by constructing a basis for each of its morphism spaces.

\subsection{A basis theorem for $\AW$}

To give a basis for an arbitrary morphism space in $\AW$, we shall use chicken foot diagrams. 
 
Let $\lambda,\mu\in \Lambda_{\text{st}}(m)$. A $\lambda\times \mu$ \emph{chicken foot diagram} (following \cite{BEEO}) is a  web diagram in $\Hom_{\W}(\mu,\lambda)$ consisting of three  horizontal  parts such that
\begin{itemize}
    \item the bottom part consists of only splits,
    \item the top part consists of only merges,
    \item the middle part consists of only crossings of the thinner strands (i.e., legs).
\end{itemize}
A chicken foot diagram (CFD) is called \emph{reduced} if there is at most one
intersection or one join between every pair of the legs. 

\begin{example}
Let $\lambda=(4,5)$ and $\mu=(2,3,4)$. The following are two $\lambda\times \mu$ chicken foot diagrams, where the first one is unreduced while the second one is reduced:
\begin{align}\label{ex-of-chickenfootd}
  \begin{tikzpicture}[anchorbase,scale=2,color=\clr]
\draw[-,line width=.6mm] (.212,.5) to (.212,.39);
\draw[-,line width=.75mm] (.595,.5) to (.595,.39);
\draw[-,line width=.15mm] (0.0005,-.396) to (.2,.4);
\draw[-,line width=.15mm]  (.2,.4)to (.4,-.4);
\draw[-,line width=.15mm](.2,.4) to (.2,.1) to(.4,-.4);
\draw[-,line width=.3mm] (0.01,-.4) to (.59,.4);
\draw[-,line width=.3mm] (.4,-.4) to (.607,.4);
\draw[-,line width=.45mm] (.79,-.4) to (.214,.4);
\draw[-,line width=.15mm] (.8035,-.398) to (.614,.4);
\draw[-,line width=.3mm] (.4006,-.5) to (.4006,-.395);
\draw[-,line width=.6mm] (.788,-.5) to (.788,-.395);
\draw[-,line width=.45mm] (0.011,-.5) to (0.011,-.395);
\node at (0.05,0.15) {$\scriptstyle 1$};
\node at (0.79,0.05) {$\scriptstyle 3$};
\node at (0.35,0.37) {$\scriptstyle 1$};
\node at (0.15,-0.32) {$\scriptstyle 1$};
%\node at (0.3,-.3) {$\scriptstyle 2$};
\node at (0.5,-0.3) {$\scriptstyle 1$};
\end{tikzpicture},
\qquad \quad   
\begin{tikzpicture}[anchorbase,scale=2,color=\clr]
\draw[-,line width=.6mm] (.212,.5) to (.212,.39);
\draw[-,line width=.75mm] (.595,.5) to (.595,.39);
\draw[-,line width=.15mm] (0.0005,-.396) to (.2,.4);
\draw[-,line width=.15mm]  (.2,.4)to (.4,-.4);
\draw[-,line width=.3mm] (0.01,-.4) to (.59,.4);
\draw[-,line width=.3mm] (.4,-.4) to (.607,.4);
\draw[-,line width=.45mm] (.79,-.4) to (.214,.4);
\draw[-,line width=.15mm] (.8035,-.398) to (.614,.4);
\draw[-,line width=.3mm] (.4006,-.5) to (.4006,-.395);
\draw[-,line width=.6mm] (.788,-.5) to (.788,-.395);
\draw[-,line width=.45mm] (0.011,-.5) to (0.011,-.395);
\node at (0.05,0.15) {$\scriptstyle 1$};
\node at (0.79,0.05) {$\scriptstyle 2$};
\node at (0.35,0.37) {$\scriptstyle 2$};
\node at (0.17,-0.3) {$\scriptstyle 1$};
\node at (0.3,-0.3) {$\scriptstyle 1$};
\node at (0.5,-0.3) {$\scriptstyle 2$};
\end{tikzpicture} .
\end{align}    
\end{example}

For $\lambda,\mu\in \Lambda_{\text{st}}(m)$, let 
$\Mat_{\lambda,\mu}$ be the set of all $l(\lambda)\times l(\mu)$
non-negative integer matrices $A=(a_{ij})$ with row sum vector $\la$ and column sum vector $\mu$, i.e.,  $\sum_{1\le h\le l(\mu)}a_{ih}=\lambda_i$ and $\sum_{1\le h\le l(\lambda)}a_{hj}=\mu_j$ for all $1\le i\le l(\lambda) $ and $1\le j\le l(\mu)$. 
The information of each reduced chicken foot diagram is encoded in a matrix $A=(a_{ij})\in \Mat_{\lambda,\mu}$ such that 
$a_{ij}$ is the thickness of the unique strand connecting $\lambda_i$ and $\mu_j$. In this case, we say the reduced chicken foot diagram is of shape $A$.
For example, the shape for the second chicken foot diagram in \eqref{ex-of-chickenfootd} is 
$\begin{pmatrix} 1&1&2\\1&2&2\end{pmatrix}$.

We introduce a shorthand notation for morphisms of the form, called {\em the $r$th elementary dot packet (of thickness $a$)},
\[
\omega_{a,r}:= \wkdota.   
\]
We may also write $\omega_r=\omega_{a,r}$ if $a$ is clear in the context.
Let $\Par_a$ be the set of  partitions $\nu=(\nu_1,\ldots,\nu_k)$ such that each part  $\nu_i\le a$.
For any partition $\nu=(\nu_1,\nu_2,\ldots, \nu_k)\in \Par_a$, the {\em elementary dot packet (of thickness $a$)} is defined to be 
\[
\omega_{a, \nu}:= \omega_{a,\nu_1}\omega_{a,\nu_2}\cdots \omega_{a,\nu_k}\in \End_{\AW}(a).
\]
We write $\omega_\nu=\omega_{a,\nu}$ if $a$ is clear from the context and draw it as 
$
\begin{tikzpicture}[baseline = 3pt, scale=0.5, color=\clr]
\draw[-,line width=1.5pt] (0,-.2) to[out=up, in=down] (0,1.2);
\draw(0,0.5) \bdot; \node at (0.6,0.5) {$ \scriptstyle \nu$};
\node at (0,-.4) {$\scriptstyle a$};
\end{tikzpicture}$. 
There are many other morphisms in $\End_{\AW}(a)$ which are not of the form $\omega_{a, \nu}.$

\begin{definition} \label{def:diagram}
    For $\lambda,\mu\in \Lambda_{\text{st}}(m)$, a $\lambda\times \mu$ elementary chicken foot diagram (or simply elementary diagram) is a $\lambda\times \mu$ reduced chicken foot diagram with an elementary dot packet $\omega_{\nu}$, for some partition $\nu\in \Par_a$, attached at the bottom of each leg with thickness $a$. 
\end{definition}

Just as a reduced chicken diagram of shape $A$ is encoded by $A\in \Mat_{\lambda,\mu}$, the elementary chicken foot diagrams of shape $A$ are encoded in the matrix $A$ enriched by certain partitions. Denote  
\begin{equation}\label{dottedreduced}
 \PMat_{\lambda,\mu}:=\{ (A, P))\mid  A=(a_{ij})\in \Mat_{\lambda,\mu}, P=(\nu_{ij}), \nu_{ij}\in \Par_{a_{ij}} \}.   
\end{equation}
We will identify the set of all elementary chicken foot diagrams from $\mu$ to $\lambda$ with $\PMat_{\lambda,\mu}$.

\begin{example}
We have the following elementary chicken foot diagram:
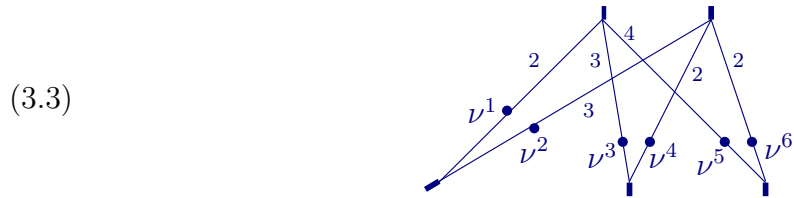
\begin{align}
    \label{ex-of-dotchickenfootd}
\begin{tikzpicture}[anchorbase,scale=1.8,color=\clr]
\draw[-,line width=.6mm] (.212,.9) to (.212,.8);
\draw[-,line width=.75mm] (1,.9) to (1,.8);
\draw[-,line width=.15mm] (-1,-.396) to (.2,.8);
\draw[-,line width=.75mm](-1,-.39) to (-1.1,-.45);
\draw[-,line width=.15mm]  (.2,.8)to (.4,-.4);
\draw[-,line width=.75mm](.4,-.4) to (.4,-.5);
\draw[-,line width=.15mm]  (.2,.8)to (1.4,-.4);
\draw[-,line width=.75mm](1.4,-.4) to (1.4,-.5);
\draw[-,line width=.15mm] (1,.8) to (-1,-.39);
\draw[-,line width=.15mm] (1,.8) to (.4,-.39);
\draw[-,line width=.15mm] (1,.8) to (1.4,-.39);
\node at (-0.3,0.5) {$\scriptstyle 2$}; 
\draw (-.5,0.13)\bdot;
\node at (-.68,0.13) {$\nu^1$};
\draw (-.3,0)\bdot; \node at (-.3,-0.16) {$\nu^2$};
\draw (.35,-0.1)\bdot;\node at (.2,-0.2) {$\nu^3$};
\draw (.55,-0.1)\bdot;\node at (.65,-0.2) {$\nu^4$};
\draw (1.1,-0.1)\bdot;\node at (1,-0.25) {$\nu^5$};
\draw (1.3,-0.1)\bdot;\node at (1.5,-0.13) {$\nu^6$};
\node at (1.2,0.5) {$\scriptstyle 2$};
\node at (0.15,0.5) {$\scriptstyle 3$};
\node at (0.1,0.13) {$\scriptstyle 3$};
\node at (0.4,0.7) {$\scriptstyle 4$};
\node at (.9,.4) {$\scriptstyle 2$};
\end{tikzpicture} 
\end{align} 
where $\nu^i\in \Par_{a_i}$, with the thickness of legs are $(a_1,\ldots,a_6)=(2,3,3,2,4,2)$. 
\end{example}

\begin{theorem} [A basis theorem for $\Hom_{\AW}(\mu,\lambda)$]
 \label{basisAW}
For any $\mu,\lambda\in \Lambda_{\text{st}}(m)$, $\Hom_{\AW}(\mu,\lambda)$ has a basis $\PMat_{\lambda,\mu}$ which consists of all elementary chicken foot diagrams from $\mu$ to $\la$.  
\end{theorem}

We single out the following special result as a consequence of Theorem~\ref{basisAW}; it improves Lemma~\ref{lem:commute} on the commutative $\kk$-algebra $D_m \subseteq \End_{\AW}(m)$. The commutativity of the algebra $\End_{\AW}\big(\strm \big)$ is not a priori clear.

\begin{theorem}
    \label{thm:Da}
    Let $m \ge 1$. The $\kk$-algebra $\End_{\AW}\big(\strm \big)$ is a polynomial algebra in generators $\wkdotm$, for $1 \le r \le m$. In particular, $D_m = \End_{\AW}(m)$.
\end{theorem}

\begin{proof}
Follows directly from Theorem \ref{basisAW} and Lemma \ref{lem:commute}.
\end{proof}

The remainder of this section is devoted to the proof of Theorem \ref{basisAW}. We shall first prove in \S\ref{subsec:span} that $\PMat_{\lambda,\mu}$ spans $\Hom_{\AW}(\mu,\lambda)$. To prove the linear independence of $\PMat_{\lambda,\mu}$, it suffices to prove it over $\Z$ (by working over $\C$) and then do a base change to $\kk$ in \S\ref{subsec:independence}. To that end we shall construct in \S\ref{subsec:rep} a representation of $\AW$ on a module category of $\mathfrak{gl}_N$.

\subsection{Proof of the span part of Theorem \ref{basisAW}}
 \label{subsec:span}

\begin{proposition} \label{prospaningset}
For any $\lambda,\mu\in \Lambda_{\text{st}}(m)$, $\Hom_{\AW}(\mu,\lambda)$ is spanned by $\PMat_{\la,\mu}$ consisting of all $\lambda \times \mu$ elementary chicken foot diagrams from $\mu$ to $\la$. 
\end{proposition}

\begin{proof}
We prove by induction on the degree $k$ for $\Hom_{\AW}(\mu,\lambda)_{\le k}$. By considering the obvious functor from $ \W $ to $\AW$, the degree 0 part follows from  \cite[Lemma~4.9]{BEEO} for the web category $\W$ since the degree 0 part is generated just by merges and splits.

Note that by \eqref{crossgen}, the generators of $\AW$
consist just of splits, merges, and dotted strands. It suffices to show that $fg$ can be written as a linear combination of elementary chicken foot diagrams up to lower degree terms (i.e, with degree $<\deg f +\deg g$)  for an arbitrary elementary chicken foot diagram $g$ and a specific generating morphism $f$ of the following three types: 
\begin{itemize}
    \item[(1)] Type I:  $1_* \wkdotaa 1_*$, i.e.,   a dotted strand tensored on the left and right by certain identity morphisms,
    \item[(2)]Type  $\rot{Y}$: $1_*\merge 1_*$,
    \item[(3)]Type Y: $1_*\splits 1_*$.
\end{itemize}
Here $1_*$ denotes suitable identity morphisms. 

We divided the proof according to the type of $f$. First note that by  Lemma \ref{dotmovefreely}, we can move dots freely through the crossing and move the dot from the thick strand to each of its legs up to lower degree terms. We will use this observation frequently below. 

(1) 
Suppose that $f$ is of type I with a dot on the $i$th strand.
Then $fg$ is obtained from $g$ by adding a dot to the thick strand at the $i$th vertex. Suppose that $g$ has a $s$-fold merge at its $i$th vertex.
Then we  use   \eqref{dotmovesplits+merge-simplify} to move the dot  to each leg of the $s$-fold merge, 
resulting a diagram $h$ obtained from $g$ by adding a thick dot $\omega_{r_j}$ on the $j$th leg of the $s$-fold merge at its $i$th vertex, where $r_j$ is the thickness of the $j$th leg of the $s$-fold merge.
Note that the diagram of $h$ without dots is the same as $g$ and hence is still a reduced chicken foot diagram.

For example, for $g$ as in \eqref{ex-of-dotchickenfootd} 
 and 
 $f=
 \begin{tikzpicture}[baseline = 3pt, scale=0.5, color=\clr]
\draw[-,line width=1.5pt] (0,-.2) to[out=up, in=down] (0,1.2);
\draw[-,line width=1.5pt] (0.5,-.2) to[out=up, in=down] (0.5,1.2);
\draw(0,0.5) \bdot; 
\draw(-.6,0.5) node{$\scriptstyle \omega_9$};
\node at (0,-.5) {$\scriptstyle 9$};
\node at (0.5,-.5) {$\scriptstyle 7$};
\end{tikzpicture} $, 
we have  
\begin{align*}
fg 
&=
\begin{tikzpicture}[anchorbase,scale=1.5,color=\clr]
\draw[-,line width=.6mm] (.212,1.1) to (.212,.8);
\draw(.212,.9) \bdot;
\draw(0,0.9) node{$\scriptstyle \omega_9$};
\draw[-,line width=.75mm] (1,1) to (1,.8);
\draw[-,line width=.15mm] (-1,-.396) to (.2,.8);
\draw[-,line width=.75mm](-1,-.39) to (-1.1,-.45);
\draw[-,line width=.15mm]  (.2,.8)to (.4,-.4);
\draw[-,line width=.75mm](.4,-.4) to (.4,-.5);
\draw[-,line width=.15mm]  (.2,.8)to (1.4,-.4);
\draw[-,line width=.75mm](1.4,-.4) to (1.4,-.5);
\draw[-,line width=.15mm] (1,.8) to (-1,-.39);
\draw[-,line width=.15mm] (1,.8) to (.4,-.39);
\draw[-,line width=.15mm] (1,.8) to (1.4,-.39);
\node at (-0.3,0.5) {$\scriptstyle 2$}; 
\draw (-.5,0.13)\bdot;
\node at (-.7,0.13) {$\nu^1$};
\draw (-.3,0)\bdot; \node at (-.3,-0.15) {$\nu^2$};
\draw (.35,-0.1)\bdot;\node at (.15,-0.18) {$\nu^3$};
\draw (.55,-0.1)\bdot;\node at (.7,-0.2) {$\nu^4$};
\draw (1.1,-0.1)\bdot;\node at (1,-0.3) {$\nu^5$};
\draw (1.3,-0.1)\bdot;\node at (1.5,-0.13) {$\nu^6$};
\node at (1.2,0.5) {$\scriptstyle 2$};
\node at (0.2,0.1) {$\scriptstyle 3$};
\node at (0.85,0.75) {$\scriptstyle 3$};
\node at (0.55,0.35) {$\scriptstyle 4$};
\node at (.9,.4) {$\scriptstyle 2$};
\end{tikzpicture}
=
\begin{tikzpicture}[anchorbase,scale=1.5,color=\clr]
\draw[-,line width=.6mm] (.212,.9) to (.212,.8);
\draw(0,.6) \bdot;
\draw(-.2,0.7) node{$\scriptstyle \omega_2$};
\draw(.212,.62) \bdot;
\draw(0.5,0.7) node{$\scriptstyle \omega_4$};
\draw(.3,.7) \bdot;
\draw(0.15,0.47) node{$\scriptstyle \omega_3$};
\draw[-,line width=.75mm] (1,.9) to (1,.8);
\draw[-,line width=.15mm] (-1,-.396) to (.2,.8);
\draw[-,line width=.75mm](-1,-.39) to (-1.1,-.45);
\draw[-,line width=.15mm]  (.2,.8)to (.4,-.4);
\draw[-,line width=.75mm](.4,-.4) to (.4,-.5);
\draw[-,line width=.15mm]  (.2,.8)to (1.4,-.4);
\draw[-,line width=.75mm](1.4,-.4) to (1.4,-.5);
\draw[-,line width=.15mm] (1,.8) to (-1,-.39);
\draw[-,line width=.15mm] (1,.8) to (.4,-.39);
\draw[-,line width=.15mm] (1,.8) to (1.4,-.39);
\node at (-0.4,0.4) {$\scriptstyle 2$}; 
\draw (-.5,0.13)\bdot;
\node at (-.7,0.13) {$\nu^1$};
\draw (-.3,0)\bdot; 
\node at (-.3,-0.15) {$\nu^2$};
\draw (.35,-0.1)\bdot;
\node at (.15,-0.18) {$\nu^3$};
\draw (.55,-0.1)\bdot;
\node at (.7,-0.2) {$\nu^4$};
\draw (1.1,-0.1)\bdot;
\node at (1,-0.3) {$\nu^5$};
\draw (1.3,-0.1)\bdot;
\node at (1.5,-0.13) {$\nu^6$};
\node at (1.2,0.5) {$\scriptstyle 2$};
\node at (0.2,0.1) {$\scriptstyle 3$};
\node at (0.85,0.75) {$\scriptstyle 3$};
\node at (0.55,0.35) {$\scriptstyle 4$};
\node at (.9,.4) {$\scriptstyle 2$};
\end{tikzpicture}.
\end{align*}
Next, we use Lemma \ref{dotmovefreely}(1) to move the new dot of $h$  through crossing down (up to lower degree terms) until it meets the dots on the $j$th leg such that it becomes a new elementary chicken foot diagram (recall all dots in a leg are commutative by Lemma \ref{lem:commute}). This completes the proof for $f$ of type I.

(2) Suppose that $f$ is of type $\rot{Y}$, a two-fold merge joining to the $i$th and $(i+1)$th strand at the top of $g$.
Suppose $g$ has an $s$-fold merge and $t$-fold merge on its $i$th and $(i+1)$st vertices, respectively. Then by \eqref{webassoc} $fg$
is a (dotted) chicken foot diagram obtained from $g$ by merge the $s$-fold merge and the $s$-fold merge to a $(s+t)$-fold merge.
Note that the resulting chicken foot diagram is not reduced in general since there may exist two legs joining twice. 
For example, for $g$ given in \eqref{ex-of-dotchickenfootd} and $f=
$$ 
Next for each pair of strands which join twice (i.e., they join both on the top and the bottom), we use relations \eqref{webassoc},\eqref{swallows}-\eqref{braid} to pull one of the strands freely such that the non-reduced pair of strands becomes 
disjoint balloons. Moreover, using \eqref{swallows} we can move the balloon down such that no other strand passing through the thin stands of the balloons.
For example, the second diagram above becomes
$$%
$$
Finally, we use Lemma 
\ref{wkdotsspan} to write each balloon of the form 
%
as a linear combination of elementary chicken foot diagrams.  

(3) Suppose that $f$ is of type Y, a 2-fold split joining to the $i$th vertex at the top of $g$.
Suppose $g$ has a $h$-fold merge in the $i$th vertex.
Then we first  use \eqref{webassoc}, \eqref{mergesplit} and \eqref{sliders}
to rewrite the composition of the split in $f$ and the merge of at the $i$th vertex of  $g$ as a sum of reduced chicken foot diagrams. For example, for $g$ given in \eqref{ex-of-dotchickenfootd} and
%
, 
the composition of $%
 $ and the leftmost merge of $g$ is 
\begin{equation}\label{examplenewsplit}
    \begin{aligned}
    %
.
\end{aligned}
\end{equation}
Then compose these diagrams with the remainder of the diagram $fg$. The resulting diagrams may not be reduced since there are new splits in the top part (e.g., \eqref{examplenewsplit}).
Next, we use \eqref{sliders} to pull each new split down  until it meets the elementary dot packet at the bottom.

For example, we pull the split in the first summand of \eqref{examplenewsplit} to obtain the following diagram
$$  %
.$$
We then use Lemma \ref{dotmovefreely}(2) to move the dots to the thinner legs of the new split (up to lower degree terms). Finally, we use \eqref{webassoc} to obtain a linear combination of elementary chicken foot diagrams. This completes the proof for $f$ of type Y. The proposition is proved. 
\end{proof}

\subsection{A representation of $\AW$}
 \label{subsec:rep}

We shall construct a representation of $\AW$ on a module category of the Lie algebra $\mathfrak {gl}_N(\C)$. 

Let $V$ be the natural representation of the general linear Lie algebra $\mathfrak{g} =\mathfrak{gl}_N(\C)$ with standard basis $\{v_1,v_2,\ldots, v_N\}$. The matrix units $e_{ij}$, for $ i,j\in [N]:=\{1,2,\ldots,N\}$, form a basis for $\mathfrak{g}$. For any $a\in \N$, let $\bigwedge^a V$ be the $a$th exterior power of $V$ with basis 
$$
v_{i_1}\wedge v_{i_2}\wedge \ldots\wedge v_{i_a}, 1\le i_1<i_2<\ldots< i_a\le N.
$$
For any $\lambda=(\lambda_1,\lambda_2,\ldots,\lambda_k)\in \Lambda_{\text{st}}$, we define 
$\bigwedge ^\lambda V:=\bigwedge ^{\lambda_1}V\otimes \ldots \otimes \bigwedge^{\lambda_k}V.$

We recall the functor which relates $\W$ and $\mathfrak g$-mod (cf. \cite{CKM, BEEO}) by using the following three $\mathfrak g$-homomorphisms:
$$\begin{aligned}
\text{Y}_{a,b} : &
\bigwedge\nolimits^{a+b} V  
\rightarrow 
\bigwedge\nolimits ^a V\otimes \bigwedge\nolimits^b V,
\\
  v_{i_1}\wedge \ldots&\wedge v_{i_{a+b}} 
\mapsto  
\sum_{w\in (\mathfrak S_{a+b}/\mathfrak S_a\times \mathfrak S_b)_{\text{min}} }(-1)^{\ell(w)}v_{i_{w(1)}}\wedge\ldots\wedge v_{i_{w(a)}}\otimes v_{i_{w(a+1)}}\wedge \ldots\wedge v_{i_{w(a+b)}}, 
\\
\rot{Y}_{a,b}: &
\bigwedge\nolimits ^aV\otimes \bigwedge\nolimits ^bV
 \rightarrow 
 \bigwedge\nolimits^{a+b} V,
 \quad v_{i_1}\wedge \ldots\wedge v_{i_a}\otimes v_{i_{a+1}}\wedge \ldots\wedge v_{i_{a+b}} 
 \mapsto v_{i_1}\wedge \ldots \wedge v_{i_{a+b}},
 \\
 \text{X}_{a,b}: &
 \bigwedge\nolimits ^aV\otimes \bigwedge\nolimits ^b V
 \rightarrow 
 \bigwedge\nolimits ^b V\otimes \bigwedge\nolimits^a V, \quad v\otimes w \mapsto (-1)^{ab} w\otimes v,
\end{aligned}
  $$
where $(\mathfrak S_{a+b}/\mathfrak S_a\times\mathfrak S_b)_{\text{min}}$ denotes the set of minimal length coset representatives.

The following result is valid over any $\kk$. Note that throughout the paper the diagrams (as morphisms in the representation category) should be read from bottom to top and from left to right.
\begin{proposition} [{\cite[Theorem 4.14]{BEEO}}]
\label{functorofweb} 
   There is a strict monoidal functor $\Phi: \W\rightarrow \mathfrak g$-mod, sending the generating object $a$ to $\bigwedge ^aV$, and the morphisms $\merge$, $\splits$ and $\crossing$ to \rm{\rot{Y}}$_{a,b}$, 
   ${\rm Y}_{a,b}$ and ${\rm X}_{a,b}$, respectively. 
\end{proposition}

Let $\End(\mathfrak g\text{-mod})$ be the category whose objects are all endofunctors of $\mathfrak g$-mod with transformation as morphisms. It is a strict monoidal category with identity functor $\text{Id}$ as the unit object. The tensor product  $\otimes$ is the horizontal composition, and $\circ $ is the vertical composition.

There is a triangular decomposition of $\mathfrak g=\mathfrak n^-\oplus \mathfrak h\oplus \mathfrak n^+$ such that $\mathfrak n^-$ (and respectively, $\mathfrak n^+$) consists of all strictly lower triangular (and respectively, upper triangular) matrices. Let $\mathfrak b=\mathfrak h\oplus \mathfrak n^+$. 
Moreover, the Cartan subalgebra $\mathfrak h$ has basis $\{h_i:= e_{i,i}\mid i\in [N] \}$. Let 
 \begin{equation}
     \Omega=\sum_{i,j\in[N]} e_{ij}\otimes e_{ji}.
 \end{equation}

Note that by \eqref{intergralballon}, the generator \begin{tikzpicture}[baseline = 3pt, scale=0.5, color=\clr]
\draw[-,line width=1.5pt] (0,0) to[out=up, in=down] (0,1.4);
\draw(0,0.6) \bdot; 
\draw (0.7,0.6) node {$\scriptstyle \omega_a$};
\node at (0,-.3) {$\scriptstyle a$};
\end{tikzpicture} 
is generated by $\dotgen$, splits and merges over a field of characteristic zero. Let us extend the functor in the previous proposition to $\AW.$

\begin{proposition} \label{functorofaff}
Suppose $\kk=\C$. There is a strict monoidal functor $\mathcal F:\AW\rightarrow \text{End}(\mathfrak {g}\text{-mod})$ sending the generating object $a$ to the functor $-\otimes \bigwedge^aV$, and sending the generating morphisms to
\begin{align*}
    \mathcal F\big(\splits\big)_M&=  \text{Id}_M\otimes\Phi \big(\splits\big),
    \\
    \mathcal F\big(\merge\big)_M&= \text{Id}_M\otimes\Phi \big(\merge\big),
    \\
    \mathcal F\big(\crossing\big)_M&=  \text{Id}_M\otimes\Phi \big(\crossing\big),
    \\
    \mathcal F\big(\dotgen\big)_M(m\otimes v) &=\Omega(m\otimes v), \quad\forall v\in V,m\in M.
\end{align*}
%where $\Phi$ is the functor given in Proposition~\ref{functorofweb}.   
\end{proposition}

\begin{proof}
It suffices to check the relations \eqref{webassoc}--\eqref{equ:r=0} and \eqref{dotmovecrossingC} according to the equivalent and simpler definition of $\AW$ over $\kk=\C$ in \S \ref{affineweboverC}. It follows from  Proposition~\ref{functorofweb} that 
$\mathcal F\big(\splits\big), \mathcal F\big(\merge\big), \mathcal F\big(\crossing\big)$ satisfy the defining relations \eqref{webassoc}--\eqref{equ:r=0}. We have
\begin{align*}
\mathcal F\big(
\begin{tikzpicture}[baseline = 7.5pt, scale=0.4, color=\clr]
\draw[-, line width=1pt] (0,-.2) to (1,1.6);
\draw[-,thick] (1,-.2) to (0,1.6);
\draw(.25,0.3)\bdot;
\node at (0, -.5) {$\scriptstyle 1$};
\node at (1, -.5) {$\scriptstyle 1$};
\end{tikzpicture}
    \big)_M(m\otimes v_i\otimes v_j)
&
=
-\sum_{k\in [N]} e_{ik}m\otimes v_j\otimes v_k,
\\
\mathcal F\big(
\begin{tikzpicture}[baseline = 7.5pt, scale=0.4, color=\clr]
\draw[-, line width=1pt] (0,-.2) to (1,1.6);
\draw[-,thick] (1,-.2) to (0,1.6);
\draw(.7,1)\bdot;
\node at (0, -.5) {$\scriptstyle 1$};
\node at (1, -.5) {$\scriptstyle 1$};
\end{tikzpicture} 
\big)_M(m\otimes v_i\otimes v_j)
&
=
-\sum_{k\in [N]} 
e_{ik}m\otimes v_j\otimes v_k -m\otimes v_i\otimes v_j,
\end{align*}
   proving the second relation in \eqref{dotmovecrossingC}. The first relation in \eqref{dotmovecrossingC} can be checked similarly.   
   \end{proof}
   
 \begin{rem}
 \label{remactfundot}
One can check that $\mathcal F\big(\wxdota \big)_M= \Omega(M\otimes \bigwedge^a V)$, for $a\ge 1$. 
 \end{rem}
 
 By evaluation at any module $M\in \mathfrak g$-mod, we obtain a functor $\mathcal F_M: \AW\rightarrow \mathfrak g$-mod. We specialize $M$ to be the generic Verma module $M^{\text{gen}}$ below. In this case, we shall work out the leading terms of the action of various morphisms on $M^{\text{gen}}\otimes \bigwedge^a V.$
 
 Let $U(\mathfrak g)$ be the universal enveloping algebra of $\mathfrak g$. We define the generic Verma module 
 \begin{equation}\label{defofgenericverma}
 M^{\text{gen}}:= U(\mathfrak g)\otimes _{U(\mathfrak b)}U(\mathfrak h),
 \end{equation}
 where $U(\mathfrak h)$ is the $U(\mathfrak b)$-module by inflation.
 Then $M^{\text{gen}}$ is a free right $U(\mathfrak h)$-module with a monomial basis in $\{e_{i,j}\mid 1\le j<i\le N\}$ (by fixing an arbitrary ordering of these elements). 
 %$$\{f_1^{a_1}\ldots f_n^{a_n}\otimes 1\mid a_1,\ldots,a_n\in \N\}, $$ where $\{f_1,\ldots, f_n\}=\{e_{i,j}\mid 1\le j<i\le N\}$. 
 We have an isomorphism of $\C$-vector spaces $M^{\text{gen}} \cong U(\mathfrak{n}^-) \otimes U(\mathfrak{h})$, and assigning $\deg h=1$ and $\deg f=0$ for $h\in \mathfrak h$ and $f\in \mathfrak n^-$ provides a $\Z$-grading on $M^{\text{gen}}$. This induces a $\Z$-grading on $M^{\text{gen}}\otimes \bigwedge^\lambda V$ by further assigning degree 0 to any element in $\bigwedge^\lambda V.$
 %such that $$ \deg (f_1^{a_1}\ldots f_n^{a_n}\otimes g\otimes v)=\deg (g), \quad\text{ for any } g\in U(\mathfrak h), v\in \bigwedge\nolimits^\lambda V. $$
 
 \begin{lemma}
 \label{lem:actionofdotlead}
 Suppose that $M=M^{\text{gen}}\otimes \bigwedge^b V$ for any $b\in \N$.
Then the action of $\mathcal F_M (\wxdota)$
 on $M\otimes \bigwedge^a V$ is of filtered degree 1. Moreover, up to lower degree terms, we have 
 \begin{equation*}
     \wxdota(1\otimes h\otimes v\otimes  v_{i_1}\wedge v_{i_2}\wedge\ldots \wedge v_{i_a})
     \equiv 1\otimes \sum_{1\le j\le a}h_{i_j}h\otimes v\otimes  v_{i_1}\wedge v_{i_2}\wedge \ldots \wedge v_{i_a},
 \end{equation*}  
 for any $h\in U(\mathfrak h)$, $v\in \bigwedge^b V$ and $1\le i_1, \ldots, i_a\le N$.
 \end{lemma}
 
 \begin{proof}
 Note that the linear map defined by the left action of  $ e_{ij} $ on $M^{\text{gen}}$ is of filtered degree 1 (and respectively, of degree 0) for $i\le j$ (and respectively, $i>j$). This together with Remark \ref{remactfundot} implies the first statement. On the other hand, we have 
    $$
    \begin{aligned}
       \wxdota(1\otimes h\otimes v\otimes  v_{i_1}\wedge v_{i_2}\wedge\ldots\wedge v_{i_a})
       &=\Omega(1\otimes h\otimes v\otimes  v_{i_1}\wedge v_{i_2}\wedge\ldots\wedge v_{i_a} )\\
    &\equiv \sum_{1\le i\le j \le N} e_{ij}(1\otimes h\otimes v)\otimes  e_{ji} v_{i_1}\wedge v_{i_2}\wedge\ldots\wedge v_{i_a}\\
    & \equiv \sum_{1\le i\le  N} e_{ii}(1\otimes h\otimes v)\otimes  e_{ii} v_{i_1}\wedge v_{i_2}\wedge\ldots\wedge v_{i_a}\\
    &= 1\otimes \sum_{1\le j\le a}h_{i_j}h\otimes v\otimes v_{i_1}\wedge v_{i_2}\wedge \ldots \wedge v_{i_a}.
    \end{aligned}  $$
This proves the second statement. 
 \end{proof}
 
\begin{corollary}
\label{actionofelempoly}
Suppose that $M=M^{\text{gen}}\otimes \bigwedge^b V$ for any $b\in \N$.
    For any $v\in \bigwedge^b V$ and $\nu\in \Par_a$, we have 
 \[
 \omega_{a,\nu} ((1\otimes h\otimes v) \otimes v_{i_1}\wedge \ldots\wedge v_{i_a})
 \equiv
 \big( 1\otimes e_{\nu}(h_{i_1},\ldots, h_{i_a})h \otimes v\big) \otimes 
 v_{i_1}\wedge \ldots\wedge v_{i_a},
 \]
 where $e_{\nu}(x_1,\ldots, x_a)$ is the elementary symmetric polynomial associated with partition $\nu$.
\end{corollary}

\begin{proof}
    Since $\omega_{a,\nu}=\omega_{a,\nu_1}\cdots \omega_{a,\nu_k}$, it suffices to prove the case for $\nu=(r)$.  By Lemma ~\ref{lem:actionofdotlead} and induction on $r$, we have 
\begin{align*}
    \omega_{a,r}& (1\otimes h\otimes v\otimes  v_{i_1}\wedge \ldots\wedge v_{i_a})
   \\
&\overset{\eqref{intergralballon}}=
 \frac{1}{r}
\begin{tikzpicture}[baseline = 5pt, scale=0.4, color=\clr]
\draw[-,line width=1.5pt](0.5,1.5) to (.5,1.8);
\draw[-,line width=1.5pt](0.5,-.5) to (.5,-.8);
\draw[-,thin]  (0.5,1.5) to[out=left,in=up] (-.5,.5)to[out=down,in=left] (0.5,-0.5);
\draw[-,line width=1pt] (0.5,-0.5) to[out=right,in=down] (1.5,.5)
to[out=up,in=right] (0.5,1.5);   
\draw (1.5,.5) \bdot; 
\node at (3.3,.5) {$\scriptstyle \omega_{a-1,r-1}$};
\node at (-.3,-.5) {$\scriptstyle 1$};
\draw (-0.5,.5) \bdot; 
\node at (-1.1,.5) {$\scriptstyle \omega_1$};
\node at (0.5,-1.1) {$\scriptstyle a$};
\end{tikzpicture} (1\otimes h\otimes v\otimes v_{i_1}\wedge \ldots\wedge v_{i_a})
   \\
   &
   =
   \frac{1}{r} (
   \begin{tikzpicture}[baseline = 5pt, scale=0.4, color=\clr]
\draw[-,line width=1.5pt](0.5,1.5) to (.5,1.8);
%\draw[-,line width=1.5pt](0.5,-.5) to (.5,-.8);
\draw[-,thin]  (0.5,1.5) to[out=left,in=up] (-.5,-.2);
\draw[-,line width=1pt]  (1.5,-.2)
to[out=up,in=right] (0.5,1.5);   
\draw (1.5,.5) \bdot; 
\node at (3.3,.5) {$\scriptstyle \omega_{a-1,r-1}$};
\node at (-.5,-.5) {$\scriptstyle 1$};
\draw (-0.5,.5) \bdot; 
\node at (-1.1,.5) {$\scriptstyle \omega_1$};
\node at (1.5,-.5) {$\scriptstyle a-1$};
\end{tikzpicture}
)(1\otimes h\otimes v\otimes (\sum_{1\le j\le a}(-1)^{j-1}v_{i_j}\otimes v_{i_1}\wedge \ldots \wedge\hat v_{i_j}\wedge \ldots \wedge v_{i_a})
\\
 &
\equiv \frac{1}{r}( \begin{tikzpicture}[baseline = 5pt, scale=0.4, color=\clr]
\draw[-,line width=1.5pt](0.5,1.5) to (.5,1.8);
%\draw[-,line width=1.5pt](0.5,-.5) to (.5,-.8);
\draw[-,thin]  (0.5,1.5) to[out=left,in=up] (-.5,-.2);
\draw[-,line width=1pt]  (1.5,-.2)
to[out=up,in=right] (0.5,1.5);   
\draw (1.5,.5) \bdot; 
\node at (3.3,.5) {$\scriptstyle \omega_{a-1,r-1}$};
\node at (-.5,-.5) {$\scriptstyle 1$};
%\draw (-0.5,.5) \bdot; 
%\node at (-1.1,.5) {$\scriptstyle \omega_1$};
\node at (1.5,-.5) {$\scriptstyle a-1$};
\end{tikzpicture} )(\sum_{1\le j\le a}(1\otimes h_{i_j}h\otimes v)\otimes (-1)^{j-1}v_{i_j}\otimes v_{i_1}\wedge \ldots \wedge\hat v_{i_j}\wedge \ldots \wedge v_{i_a})
\\  &\equiv \frac{1}{r}
\sum_{1\le j\le a}(1\otimes h_{i_j}e_{r-1}(h_{i_1},\ldots, h_{i_{j-1}},\hat h_{i_j}, \ldots, h_{i_a}) h\otimes v\otimes  v_{i_1}\wedge \ldots\wedge v_{i_a})\\
&=\big( 1\otimes e_{r}(h_{i_1},\ldots, h_{i_a})h \big) \otimes v\otimes
 v_{i_1}\wedge \ldots\wedge v_{i_a}
\end{align*}
 where $e_{r-1}(h_{i_1},\ldots, h_{i_{j-1}},\hat h_{i_j}, \ldots, h_{i_a})$  is the $(r-1)$th elementary symmetric polynomial of variables $h_{i_1}, \ldots, h_{i_{j-1}}, h_{i_{j+1}},\ldots, h_{i_a}$.
This completes the proof. 
\end{proof}

\subsection{Linear independence of $\PMat_{\lambda,\mu}$}
\label{subsec:independence}

In this subsection we prove that $\PMat_{\lambda,\mu}$ is linearly independent in $\Hom_{\AW}(\mu,\lambda)$. Together with Proposition~\ref{prospaningset} this completes the proof of Theorem \ref{basisAW}.

Suppose $\lambda,\mu\in \Lambda_{\text{st}}(m)$ with $l(\lambda)=h$ and $l(\mu)=t$.
Recall for any reduced chicken foot diagram of shape $A \in\Mat_{\lambda,\mu}$, each $\mu_j$ splits first at the bottom to $a_{1j}, \ldots, a_{hj}$ from left to right, then the thin strands $a_{i1},\ldots, a_{it}$ merge back to  $\lambda_i$ at the top.
For any elementary chicken foot diagram of shape $A$, there is an elementary dot packet,  say $\omega_{\eta_{ij}}\in \Par_{a_{ij}}$ on each $a_{ij}$ at the bottom. 

We first assume that $\kk=\C$. Suppose that 
\begin{equation}
    L:=\sum_{A_\wp\in S}c_{A_\wp} A_\wp=0,
\end{equation}
for some finite subset $S\subset \PMat_{\lambda,\mu}$ with each $c_{A_\wp}\in \C^*$. 
%For $A_\wp= (A,P) \in S$, we define $\deg (A_\wp) :=\deg (A)$. 
Let $\hbar =\max\{\deg (A_\wp) \mid A_\wp\in S\}$ and $S_\hbar=\{A_\wp\in S\mid \deg (A_\wp)=\hbar \}$. 

It remains to prove $S_\hbar=\emptyset$. To prove this, we use the functor $\mathcal F_{M^{\text{gen}}}$ and compare the highest degree terms of $A_\wp\in S_\hbar$ on certain element of $M^{\text{gen}}\otimes \bigwedge ^\mu V$.

We choose $N\gg 0$ (e.g., $N>d$) and $u:=1\otimes 1\otimes w_\mu$, where 
$$w_\mu:= v_1\wedge \ldots\wedge v_{\mu_1} \otimes v_{\mu_1+1}\wedge \ldots \wedge v_{\mu_1+\mu_2}\otimes \ldots \otimes v_{\mu_1+\ldots+\mu_{t-1}+1}\wedge \ldots\wedge v_{d}.  $$
That is,   the indices of the vector in   $\bigwedge ^{\mu_j}V $ is  $I_{\mu_j}=\{\sum_{1\le s\le j-1}\mu_s+1, \ldots, \sum_{1\le s\le j}\mu_s \}$ and hence are disjoint for different $j$, $1\le j\le l(\mu)=t$.
First we claim that any two different CFD's $A,A'\in \Mat_{\lambda,\mu}$ (i.e., the CFD without dot packets) would act differently on $w_\mu$. In fact, write the action of $A$ on $w_\mu $ as  
\begin{equation}
\label{equ:actwithoutdot}
 A  w_\mu=\sum_{\mathbf i} c_{\mathbf i} v_{\mathbf i_{\lambda_1}}\otimes \ldots \otimes v_{\mathbf i_{\lambda_h}} 
\end{equation}
for some non-zero scalars $c_{\mathbf i}$, 
where $\mathbf i_{\lambda_j}=(i_{j,1}, i_{j,2},\ldots, i_{j,\lambda_j})$ and 
$v_{\mathbf i_{\lambda_j}}=v_{i_{j,1}}\wedge \ldots\wedge v_{i_{j,\lambda_j}}\in \bigwedge^{\lambda_j} V$. 
By the actions of merges, splits and thick crossings on the wedge space in Proposition \ref{functorofweb}, 
for each $v_{\mathbf i_{\lambda_j}}$, we must have 
\[\sharp \{ i_{j,k}\in I_{\mu_i}\mid 1\le k\le \lambda_j \}= a_{j,i}.\]
This shows that any summands in \eqref{equ:actwithoutdot} are different for different $A$ and $A'$ since  $I_{\mu_j}\cap I_{\mu_i}=\emptyset$ if $i\ne j$.

Now we consider the degree $\hbar$ component of $\mathcal F_{M^{\text{gen}}}(L)(u)$, which are contributed by $\mathcal F_{M^{\text{gen}}}(A_\wp)(u)$, $A_\wp\in S_\hbar$.
 Note that by Corollary \ref{actionofelempoly}, the wedge part of 
the highest degree component does not change under the elementary dot packet action. This together with the claim in the previous paragraph imply that the wedge part of the highest degree component of $F_{M^{\text{gen}}}(A_\wp)(u)$ is different from  $F_{M^{\text{gen}}}(A_\wp')(u)$ if $A\neq A'$,
where 
    $A_\wp=(A,P),  A_\wp'=(A',P')\in S_\hbar$.
Thus, we may assume that all $A_\wp$ in $S_\hbar$ have the same shape $A\in \Mat_{\lambda,\mu}$.

Next, we choose a unique element $z_{\lambda}$ appearing in the action of reduced chicken foot diagram $A$ on $w_\mu $, where   
 $z_{\lambda}=z_{\lambda_1}\otimes \ldots \otimes z_{\lambda_h} $ and 
  $z_{\lambda_i}= v_{d_1}\wedge \ldots\wedge v_{d_{\lambda_i}}$ with $(d_1,\ldots, d_{\lambda_i})=(\mathbf i_1, \ldots, \mathbf i_t )$, where 
  $$ \mathbf i_j=
  \Big(\sum_{1\le m\le j-1}\mu_m+\sum_{1\le s\le i-1}a_{sj}+1,\ldots, \sum_{1\le m\le j-1}\mu_m+\sum_{1\le s\le i}a_{sj} \Big). $$
    
Now use  Corollary \ref{actionofelempoly}  again to see that the  unique term of degree $\hbar$ component of $\mathcal F_{M^{\text{gen}}}(L)(u)$ with   wedge part  $z_\lambda$  must be  
$$
\sum_{B\in S_\hbar}c_{A_\wp} \prod_{i,j} e_{\eta_{ij}} \big(h_{\sum_{1\le s\le j-1}\mu_s+\sum_{1\le m\le i-1}a_{mj}+1}, \ldots,h_{\sum_{1\le s\le j-1}\mu_s+\sum_{1\le m\le i}a_{mj}} \big) z_\lambda.   
$$
This cannot be zero if $S_\hbar\neq \emptyset$ since the partial symmetric functions appearing above are linearly independent. Hence we conclude that $S_\hbar =\emptyset$.

This proves that $\PMat_{\lambda,\mu}$ is linearly independent over $\mathbb C$ and hence over $\Z$. 

Now suppose that $\kk$ is a commutation ring with $1$. 
Write $\AW_{\kk}$  the affine web category over $\kk$ and let
$\mathcal C:=\AW_\Z\otimes _\Z\kk$. Then  $\Hom_{\mathcal C}(\mu,\lambda)$ has basis given by   $\PMat_{\lambda,\mu}$ by base change property. Hence,  the linear independence of $\PMat_{\lambda,\mu}$ for $\AW_\kk $ over $\kk$ follows from the linear independence for $\mathcal C$ and the obvious functor from $\AW_\kk$ to $\mathcal C$ sending the required basis element of $\Hom_{\AW_\kk}(\mu,\lambda)$ to the corresponding basis element in  $\Hom_{\mathcal C}(\mu,\lambda)$.  This completes the proof of the linear independence of $\PMat_{\lambda,\mu}$ over $\kk$. 

For $k \ge 0$, we denote
\begin{align}  \label{Dak}
    D_{a}^k=\kk\text{-span} \{ \omega_{a, \eta} \in D_{a}\mid l(\eta)\le k \}.
\end{align}
The following refinement of Lemma \ref{wkdotsspan} will be used later.
\begin{corollary}
\label{cor:blambdamu}
Recall $b_{\lambda,\mu}\in D_{a+b}$ in \eqref{blambdamu}. Then we have 
$b_{\lambda,\mu}\in D_{a+b}^k$ with 
$k=\max\{l(\lambda),l(\mu)\}$.    
\end{corollary}

\begin{proof}
 Write $b_{\lambda,\mu}$ as a $\Z$-linear combination  
 $b_{\lambda,\mu}= \sum_{\nu}a_{\nu} \omega_{a+b,\nu}.$
 Then we use Corollary \ref{actionofelempoly} to compare the leading term of both sides of the above equation on $1\otimes 1\otimes v_1\wedge \ldots \wedge v_{a+b}$, and we may conclude that $a_{\nu}\neq0$ only if 
 $l(\nu)\le k$.
\end{proof}

\subsection{Affine Hecke algebras}
\label{affinehecke}

Recall (see e.g., \cite{Kle05}) the degenerate affine Hecke algebra $\widehat{\mathcal H}_m$ is the associative unital $\kk$-algebra generated by $x_i$ and $s_j$, for $1\le i\le m$, $1\le j\le m-1$, subject to the relations for all admissible $i,j,l$:
\begin{align}
        s_i^2=1, \quad s_is_{i+1}s_i &=s_{i+1}s_is_{i+1}, \quad s_js_i =s_is_j \; (|i-j|>1),
        \label{ss} \\
         x_ix_l& =x_lx_i,  \quad     s_jx_i=x_is_j\; (j\neq i,i+1),
        \label{xx} \\
        x_{i+1}s_i &=s_ix_i-1. 
        %\quad s_ix_{i+1}=x_is_i-1,
        \label{xs}
\end{align}  

\begin{proposition}
\label{isotodeaffine}
We have an algebra isomorphism $\End_{\AW}(1^m) \cong \widehat {\mathcal H}_m$.
\end{proposition}

\begin{proof}
We claim to have an algebra homomorphism 
   $\phi: \widehat {\mathcal H}_m\rightarrow \End_{\AW}(1^m) $ such that 
   $$ 
   \begin{aligned}
x_j &\mapsto 
\begin{tikzpicture}[baseline = 3pt, scale=0.5, color=\clr]
\draw[-,line width=1pt] (-1,-.2) to (-1,1.2);
\node at (-1,-.5) {$\scriptstyle 1$};
\node at (-.4,.4) {$\ldots$};
\draw[-,line width=1pt] (.5,-.2) to (.5,1.2);
\node at (.5,-.5) {$\scriptstyle 1$};
\draw[-,line width=1pt] (1,-.2) to (1,1.2);
\draw(1,0.5) \bdot;
\node at (1,-.5) {$\scriptstyle 1$};
\draw[-,line width=1pt] (1.5,-.2) to (1.5,1.2);
\node at (1.5,-.5) {$\scriptstyle 1$};
\node at (2.1,.4) {$\ldots$};
\draw[-,line width=1pt] (3,-.2) to (3,1.2);
\node at (3,-.5) {$\scriptstyle 1$};
\end{tikzpicture}, \\
s_i& \mapsto 
\begin{tikzpicture}[baseline = 3pt, scale=0.5, color=\clr]
\draw[-,line width=1pt] (-1,-.2) to (-1,1.2);
\node at (-1,-.5) {$\scriptstyle 1$};
\node at (-.4,.4) {$\ldots$};
\draw[-,line width=1pt] (.5,-.2) to (.5,1.2);
\node at (.5,-.5) {$\scriptstyle 1$};
%\draw[-,line width=1pt] (1,-.2) to (1,1.2);
\draw[-,line width=1pt] (1,-.2) to (2,1.2); 
\node at (1,-.5) {$\scriptstyle 1$};
\draw[-,line width=1pt] (2,-.2) to (1,1.2);
\node at (2,-.5) {$\scriptstyle 1$};
\draw[-,line width=1pt] (2.5,-.2) to (2.5,1.2);
\node at (2.5,-.5) {$\scriptstyle 1$};
\node at (3.2,.4) {$\ldots$};
\draw[-,line width=1pt] (4,-.2) to (4,1.2);
\node at (4,-.5) {$\scriptstyle 1$};
\end{tikzpicture}, 
   \end{aligned}$$
where the dot is on the $j$th strand, and the crossing is between $i$th and $(i+1)$st strands. Indeed, the relations \eqref{ss}--\eqref{xx} are satisfied by the images under $\phi$ of the corresponding generators thanks to \eqref{symmetric}--\eqref{braid} and \eqref{dotmovecrossing}. The remaining relation \eqref{xs} under $\phi$ clearly holds by the monoidal structure of $\AW$. 

Now the isomorphism follows since $\phi$ sends the well-known basis elements in  $\{wx_1^{k_1}\ldots x_m^{k_m}\mid w\in \mathfrak S_m, k_i\in \N, 1\le i\le m\}$ to the corresponding  basis elements of  $\End_{\AW}(1^m)$ in Theorem~\ref{basisAW}.
\end{proof}

\subsection{Centers of endomorphism algebras in $\AW$}

There are rich combinatorial structures inside $\AW$.

\begin{conjecture}
\label{con:centerAW}
For any $\lambda\in\Lambda_{\text{st}}(m)$, the center of the endomorphism algebra $\End_{\AW}(\lambda)$ is isomorphic to the ring of symmetric functions in $m$ variables.
\end{conjecture}

Let us provide some supporting evidence below. 

\begin{example}
\begin{enumerate}
    \item    
For $\la=(1^m)$, we have $\End_{\AW}(1^m) \cong \widehat {\mathcal H}_m$ by Proposition~ \ref{isotodeaffine}. It is well known that the center of $\widehat {\mathcal H}_m$ is isomorphic to $\kk[x_1,x_2,\ldots, x_m]^{\mathfrak S_m}$, cf. \cite{Kle05}. 
\item 
For $\la=(m)$, the conjecture holds by Theorem~ \ref{thm:Da}. 
\item
Let $\lambda=(2,1)$. One can check that the center of $\End_{\AW}(\lambda) $  is freely generated by 
\[z_1=~
\begin{tikzpicture}[baseline = 3pt, scale=0.4, color=\clr]
\draw[-,line width=1.5pt](0,-.5) to (0,1.2);
\draw(0,.5) \bdot;
%\draw(-.55,0)node {$\scriptstyle \omega_{2}$};
\draw[-,line width=1pt](.5,-.5) to (0.5,1.2);
\draw(0,-.8) node {$\scriptstyle 2$};
\draw(0.5,-.8) node {$\scriptstyle 1$};
\end{tikzpicture}
+
\begin{tikzpicture}[baseline = 3pt, scale=0.4, color=\clr]
\draw[-,line width=1.5pt](0,-.5) to (0,1.2);
\draw(0.5,.5) \bdot;
%\draw(-.55,0)node {$\scriptstyle \omega_{2}$};
\draw[-,line width=1pt](.5,-.5) to (0.5,1.2);
\draw(0,-.8) node {$\scriptstyle 2$};
\draw(0.5,-.8) node {$\scriptstyle 1$};
\end{tikzpicture}
-
\begin{tikzpicture}[baseline = 3pt, scale=0.4, color=\clr]
\draw[-,line width=1.5pt](0,-.5) to (0,1.2);
%\draw(0.5,.5) \bdot;
%\draw(-.55,0)node {$\scriptstyle \omega_{2}$};
\draw[-,line width=1pt](.5,-.5) to (0.5,1.2);
\draw(0,-.8) node {$\scriptstyle 2$};
\draw(0.5,-.8) node {$\scriptstyle 1$};
\end{tikzpicture}
,~\quad
z_2=
\begin{tikzpicture}[baseline = 3pt, scale=0.4, color=\clr]
\draw[-,line width=1.5pt](0,-.5) to (0,1.2);
\draw(0,.5) \bdot;
\draw(-.65,0.5)node {$\scriptstyle \omega_{2}$};
\draw[-,line width=1pt](.5,-.5) to (0.5,1.2);
\draw(0,-.8) node {$\scriptstyle 2$};
\draw(0.5,-.8) node {$\scriptstyle 1$};
\end{tikzpicture}
+
\begin{tikzpicture}[baseline = 3pt, scale=0.4, color=\clr]
\draw[-,line width=1.5pt](0,-.5) to (0,1.2);
\draw(0.5,.5) \bdot;
\draw(0,.5) \bdot;
%\draw(-.55,)node {$\scriptstyle \omega_{2}$};
\draw[-,line width=1pt](.5,-.5) to (0.5,1.2);
\draw(0,-.8) node {$\scriptstyle 2$};
\draw(0.5,-.8) node {$\scriptstyle 1$};
\end{tikzpicture}
-
\begin{tikzpicture}[baseline = 3pt, scale=0.4, color=\clr]
\draw[-,line width=1.5pt](0,-.5) to (0,1.2);
\draw(0.5,.5) \bdot;
%\draw(-.55,0)node {$\scriptstyle \omega_{2}$};
\draw[-,line width=1pt](.5,-.5) to (0.5,1.2);
\draw(0,-.8) node {$\scriptstyle 2$};
\draw(0.5,-.8) node {$\scriptstyle 1$};
\end{tikzpicture}
,~\quad
z_3= 
\begin{tikzpicture}[baseline = 3pt, scale=0.4, color=\clr]
\draw[-,line width=1.5pt](0,-.5) to (0,1.2);
\draw(0,.5) \bdot;
\draw(0.5,.5) \bdot;
\draw(-.65,0.5)node {$\scriptstyle \omega_{2}$};
\draw[-,line width=1pt](.5,-.5) to (0.5,1.2);
\draw(0,-.8) node {$\scriptstyle 2$};
\draw(0.5,-.8) node {$\scriptstyle 1$};
\end{tikzpicture} .
\]
%and it is isomorphic to $\kk[x_1,x_2,x_3]^{\mathfrak S_3}$ by sending $z_1$ $z_2$, $z_3$ to $x_1+x_2+x_3$,  $x_1x_2+x_1x_3+x_2x_3$ and  $x_1x_2x_3$, respectively.
\end{enumerate} 
\end{example}

\begin{rem}
For $\lambda =(\la_1, \ldots, \la_l) \in\Lambda_{\text{st}}$, $\End_{\AW}(\lambda_1) \otimes \ldots \otimes \End_{\AW}(\lambda_l)$ is a  commutative subalgebra of the endomorphism algebra $\End_{\AW}(\lambda)$ by Theorem~\ref{thm:Da}. We expect that it is a maximally commutative subalgebra.
\end{rem}
\section{Cyclotomic web categories and finite $W$-algebras}
 \label{sec:WSchur}

In this section, we introduce the cyclotomic web categories and establish a functor from a cyclotomic web category to the module category of an appropriate finite $W$-algebra of type A. 
We also establish precise connections between the cyclotomic web categories and the $W$-Schur algebras arising from the Schur duality for finite $W$-algebras introduced in \cite{BK08}. 

By default we work over an arbitrary commutative ring $\kk$ in this section unless otherwise specified. We specify $\kk=\C$ when working with finite $W$-algebras. 

\subsection{The cyclotomic web categories}
Recall from \eqref{def-gau} the elements 
\[
g_{r,u} =\sum_{0\le i\le r}(-1)^i\prod_{0\le j\le i-1}(u+j)\:\:
 \begin{tikzpicture}[baseline = -1mm,scale=.7,color=\clr]
\draw[-,line width=1.5pt] (0.08,-.6) to (0.08,.5);
\node at (.08,-.8) {$\scriptstyle r$};
\draw(0.08,0) \bdot;
\draw(.65,0)node {$\scriptstyle \omega_{r-i}$};
\end{tikzpicture}\in \End_{\AW}(r), 
\]
for any $r\ge 1$ and $u\in\kk$. 
\begin{definition}
 \label{def:cycWeb}
    Suppose  $\ell\ge 1$. For any $\bfu=(u_1,\ldots, u_\ell) \in \kk^\ell$, the cyclotomic web category $\W_{\bfu}$ is the quotient category of $\AW$ by the right tensor ideal generated by $\prod_{1\le j\le \ell }g_{r,u_j}$, for $r\ge 1$. 
\end{definition}

Recall from \eqref{dottedreduced} the set $\PMat_{\lambda,\mu}$ of all elementary chicken foot diagrams from $\mu$ to $\la$.
For any matrix  $ P=(\nu_{ij})_{l(\lambda)\times l(\mu)}$ of partitions, denote 
\[
l(P)=\max\{l(\nu_{ij})\mid 1\le i\le l(\lambda), 1\le j\le l(\mu)\}.
\]
We define
\begin{equation}
\label{Def-spansetof-cycweb}    
\PMat_{\lambda,\mu}^\ell:=\{(A,P)\in\PMat_{\lambda,\mu}\mid l(P)\le \ell-1\},
\end{equation}
and regard it as the set of all level $\ell$ elementary chicken foot diagrams from $\mu$ to $\la$. Note that, for $\ell=1$, $\PMat_{\lambda,\mu}^\ell =\Mat_{\lambda,\mu}.$

\begin{lemma}
\label{lem:cycweb-span}
For any $\lambda,\mu \in\Lambda_{\text{st}}(m)$, $\Hom_{\W_{\bfu}}(\mu,\lambda)$ is spanned  by $ \PMat_{\lambda,\mu}^\ell$.    
\end{lemma}
\begin{proof}
Recall $D_a^k$ from \eqref{Dak}. We make the following
    
{\bf Claim}. For any $\nu\in \Par_a$ with $l(\nu)\ge\ell$, we have
\begin{equation}
\begin{tikzpicture}[baseline = 3pt, scale=0.5, color=\clr]
\draw[-,line width=1.5pt] (0,-.2) to[out=up, in=down] (0,1.2);
\draw(0,0.5) \bdot; \node at (0.6,0.5) {$ \scriptstyle \nu$};
\node at (0,-.4) {$\scriptstyle a$};
\end{tikzpicture}\in D_{a}^{\ell-1}.
\end{equation}
We prove the claim by induction on $\min\{\nu_j\mid 1\le j\le \ell \}$ and on degree. 
Suppose $\nu_\ell=r$ and write $\bar \nu=(\nu_1,\nu_2,\ldots,\nu_{\ell-1})$. If $r=a$, then we must have $\nu=(r,r,\ldots,r)$ and the claim holds since 
 \begin{equation}
 \label{cycweb-polyr}
   \prod_{1\le j\le \ell}g_{r,u_j}=0  
 \end{equation}
 in $\W_\bfu$. Suppose $r<a$. 
 Let $\Lambda_{r,\ell-1}(\nu)$ be the set of pair of compositions 
 $(\nu',\nu'')$ such that $l(\nu'),l(\nu'')\le \ell-1$, $\nu'_j+\nu''_j=\nu_j$, and $0\le \nu_j'\le r$ for all $1\le j\le \ell-1$.  
 we have 
\begin{align*}
\begin{tikzpicture}[baseline = 3pt, scale=0.5, color=\clr]
\draw[-,line width=1.5pt] (0,-.2) to[out=up, in=down] (0,1.2);
\draw(0,0.5) \bdot; \node at (0.4,0.5) {$ \scriptstyle \nu$};
\node at (0,-.4) {$\scriptstyle a$};
\end{tikzpicture}
&\overset{\eqref{splitmerge}}=
\begin{tikzpicture}[baseline = 3pt, scale=0.7, color=\clr]
\draw[-,line width=1.5pt] (0,.5) to[out=up, in=down] (0,1.2);
\draw(0,.8) \bdot; 
\node at (0.4,.8) {$ \scriptstyle \bar\nu$};
\draw[-,line width=1pt](0,-.5) to [out=left,in=left] (0,0.5);
\draw[-,line width=1pt](0,-.5) to [out=right,in=right] (0,0.5);
\draw[-,line width=1.5pt] (0,-.8) to (0,-.5);
\draw(-.3,0) \bdot; 
\node at (-.7,0) {$ \scriptstyle \omega_r$};
\node at (-.3,-.6) {$\scriptstyle r$};
\node at (.6,-.5) {$\scriptstyle {a-r}$};
\end{tikzpicture}
\overset{Lem\; \ref{dotmovefreely}(3)}\equiv
\sum_{(\nu',\nu'')\in\Lambda_{r,\ell-1}(\nu) }
\begin{tikzpicture}[baseline = 3pt, scale=0.7, color=\clr]
\draw[-,line width=1.5pt] (0,.5) to[out=up, in=down] (0,1);
%\draw(0,.8) \bdot; 
%\node at (0.4,.8) {$ \scriptstyle \bar\nu$};
\draw[-,line width=1pt](0,-.5) to [out=left,in=left] (0,0.5);
\draw[-,line width=1pt](0,-.5) to [out=right,in=right] (0,0.5);
\draw[-,line width=1.5pt] (0,-.8) to (0,-.5);
\draw(-.3,.2) \bdot;
\node at (-.7,.2) {$ \scriptstyle \nu'$};
\draw(.3,.1) \bdot;
\node at (.7,.1) {$ \scriptstyle \nu''$};
\draw(-.3,-.2) \bdot; 
\node at (-.7,-.2) {$ \scriptstyle \omega_r$};
\node at (-.3,-.6) {$\scriptstyle r$};
%\node at (.55,-.5) {$\scriptstyle {a-r}$};
\end{tikzpicture}\\
&
\overset{\eqref{cycweb-polyr}}\equiv
\sum_{\overset{(\nu',\nu'')\in\Lambda_{r,\ell-1}(\nu)}{\nu'\neq (r,r,\ldots,r)}}
\begin{tikzpicture}[baseline = 3pt, scale=0.7, color=\clr]
\draw[-,line width=1.5pt] (0,.5) to[out=up, in=down] (0,1);
%\draw(0,.8) \bdot; 
%\node at (0.4,.8) {$ \scriptstyle \bar\nu$};
\draw[-,line width=1pt](0,-.5) to [out=left,in=left] (0,0.5);
\draw[-,line width=1pt](0,-.5) to [out=right,in=right] (0,0.5);
\draw[-,line width=1.5pt] (0,-.8) to (0,-.5);
\draw(-.3,.2) \bdot;
\node at (-.7,.2) {$ \scriptstyle \nu'$};
\draw(.3,.1) \bdot;
\node at (.7,.1) {$ \scriptstyle \nu''$};
\draw(-.3,-.2) \bdot; 
\node at (-.7,-.2) {$ \scriptstyle \omega_r$};
\node at (-.3,-.6) {$\scriptstyle r$};
%\node at (.55,-.5) {$\scriptstyle {a-r}$};
\end{tikzpicture}
\end{align*}
 which belongs to $D_a^{\ell-1}$ by the inductive assumption and Corollary~ \ref{cor:blambdamu}.
 Now the claim follows by applying the above argument repeatedly.

Note that the morphism spaces of $\W_{\bfu}$ are spanned by all elementary chicken foot diagrams. By \eqref{dotmovecrossing} and \eqref{symmetric}, 
we have 
\begin{equation}
\label{movedottofirststrand}
  \begin{tikzpicture}[baseline = -1mm,scale=0.7,color=\clr]
	\draw[-,thick] (0.2,-.6) to (0.2,.6);
	\draw[-,thick] (-0.2,-.6) to (-0.2,.6);
        \node at (0.2,-.75) {$\scriptstyle b$};
        \node at (-0.2,-.75) {$\scriptstyle a$};
        \draw (.2, 0) \bdot;
        \node at (0.6,0) {$\scriptstyle \omega_r$};
\end{tikzpicture}
\equiv
\begin{tikzpicture}[baseline = -1mm,scale=0.7,color=\clr]
	\draw[-,thick] (0.28,0) to[out=90,in=-90] (-0.28,.6);
	\draw[-,thick] (-0.28,0) to[out=90,in=-90] (0.28,.6);
	\draw[-,thick] (0.28,-.6) to[out=90,in=-90] (-0.28,0);
	\draw[-,thick] (-0.28,-.6) to[out=90,in=-90] (0.28,0);
        \node at (0.3,-.75) {$\scriptstyle b$};
        \draw (-.28, 0) \bdot;
        \node at (-0.7,0) {$\scriptstyle \omega_r$};
        \node at (-0.3,-.75) {$\scriptstyle a$};
\end{tikzpicture}.
\end{equation}
Then for any elementary CFD $(A,P)$ with  $l(\nu_{ij})\ge \ell$ for some $\nu_{ij}$, we first use \eqref{movedottofirststrand} to move the dot $\omega_{\nu_{ij}}$ to the first strand of the diagram (up to lower degree terms) and then applying the above claim to see that $(A,P)$ can be written as a linear combination of some elementary CFD's $(\tilde A,\tilde P)$ with $ l(\tilde\nu_{ij})<\ell$ up to lower degree terms. Then the result follows by applying this procedure repeatedly and induction on degrees.  
\end{proof}

\begin{theorem}
\label{basis-cyc-web} 
Suppose that ${\bfu} =(u_1, \ldots, u_\ell) \in\kk^\ell$. Then $\Hom_{\W_{\bfu}}(\mu,\lambda)$ has a basis given by $\PMat_{\lambda,\mu}^\ell$, for any $\lambda,\mu \in\Lambda_{\text{st}}(m)$.    
\end{theorem}

\begin{proof}
By Lemma \ref{lem:cycweb-span}, it remains to prove the linear independence of the spanning set  $\PMat_{\lambda,\mu}^\ell$ for $\Hom_{\W_{\bfu}}(\mu,\lambda)$. 
A direct lengthy proof of this using representations of $\W_\bfu$ on $\mathfrak{gl}_N$-mod is possible but will be skipped. Instead, we shall enlarge affine/cyclotomic web categories to affine/cyclotomic Schur categories denoted $\ASch$ and $\Sch_\bfu$ in \cite{SW2} and establish a basis theorem for $\Sch_\bfu$ in \cite[Theorem~E or 5.17]{SW2}. Then the linear independence of $\PMat_{\lambda,\mu}^\ell$ in $\Hom_{\W_{\bfu}}(\mu,\lambda)$ follows readily via a functor $\W_\bfu \rightarrow \Sch_\bfu$ (see \cite[\S 5.5]{SW2}%\cite[\S\ref{subsec:proofbasis}]{SW2}
), which maps $\la$ to $(\emptyset^\ell, \la)$ and the spanning set $\PMat_{\lambda,\mu}^\ell$ for $\Hom_{\W_{\bfu}}(\mu,\lambda)$ to a basis for $\Hom_{\Sch_{\bfu}}((\emptyset^\ell,\mu), (\emptyset^\ell,\lambda))$. 
\end{proof}

\begin{proposition}
The cyclotomic web category $\W_{(0)}$ for $\ell=1$ and  $\bfu=(0)$ is isomorphic  to the web category $\W$ as $\kk$-linear categories.  In particular, $\W$ is both a subcategory and a quotient category of $\AW$.
\end{proposition}

\begin{proof} 
 By Lemma \ref{lem:cycweb-span}, the morphism spaces in $\W_{(0)}$ are generated by reduced chicken foot diagrams (without dots). These elements are linearly independent over $\C$ (and over $\Z$ and hence over $\kk$ by base change) by considering the functor $\mathcal F_M$ with $M$ being $\C$ which factors through $\W_{(0)}$.
 We conclude that each morphism space of $\W_{(0)}$ has a basis given by reduced chicken foot diagrams, the same as $\W$.
 Now composing the obvious functor $\W\rightarrow \AW$ with the quotient functor $\AW\rightarrow \W_{(0)}$
gives us the required isomorphism $\W\stackrel{\cong}{\rightarrow} \W_{(0)}$. The proposition follows. 
\end{proof}

Recall \cite{BEEO} that $\W$ (a cyclotomic web category for $\ell=1$) is isomorphic to the Schur category and the usual Schur algebras are Morita equivalent to the path algebras for $\W$ (see Proposition~\ref{cor-isom-schur}). Hence, it is natural to ask about possible generalizations for cyclotomic web categories, for $\ell \ge 1$. This will be answered affirmatively later in this section. We first need some preparation.

\subsection{A representation of $\W_\bfu$}

We will define a new representation of $\AW$, which factors through a quotient of $\AW$ which is a cyclotomic web category. This functor will also establish a connection between the cyclotomic web categories and finite $W$-algebras in the next subsections.

Fix 
%$(u_1, \ldots, u_\ell) \in \kk^\ell$ and 
a strict composition $\Lambda=(\lambda_1,\ldots, \lambda_n)$ of $N$ such that 
 $\lambda_1\le \ldots \le \lambda_n$, and  $l(\Lambda')=\ell$ (i.e.,  $l(\Lambda)=n$ and  $\lambda_n=\ell$), where $\Lambda'$ is the transpose partition of $\Lambda$.
We number the rows of the Young  diagram of $\Lambda$ by
$1,\ldots,n$ from top to bottom, and columns by $1,\ldots, \ell$ from left to right. Moreover, we number the boxes $1,2,\ldots,N$ in the Young diagram of $\Lambda$ along the rows starting at the top row (This differs from the numbering convention along the columns in \cite{BK08} but will be convenient when we consider adding more rows of $\ell$ boxes to the bottom of $\Lambda$ later on). 
%\footnote{Note that we use different numbering from \cite{BK08}, in which they number the boxes down columns starting from the first column. The finite $W$-algebras and $W$-finite Schur algebras are isomorphic for the two different numberings.  We use the new numbering so as to define an infinite version of $W$-Schur algebra later.}.
For example, suppose $\Lambda=(1,2,3,3)\in \Lambda_{\text{st}}(9)$, the Young diagram for $\Lambda$ is 
\begin{align}
\label{pyramid-lambda}
   \begin{ytableau}
    1 & \none & \none  \\
    2 & 3  & \none \\
    4 & 5  & 6 \\
7& 8& 9\\
\end{ytableau}   
\end{align}
The column and row numbers of the $i$th box are denoted by $\text{col}(i)$ and $\text{row}(i)$.
Let $V$ be the free $\kk$-module with basis $\{v_i\mid i\in [N]\}$, where $[N]=\{1,2,\ldots, N\}$. 
Let $e\in \End_\kk(V)$ be the nilpotent matrix that maps $v_i$ to $v_{j}$ such that  the $j$-box is immediately to the left of the $i$th box of $\Lambda$
if the $i$th box is not the leftmost box in its row, or to zero otherwise.
For example, for \eqref{pyramid-lambda}, 
$e=e_{2,3}+e_{4,5}+e_{5,6}+e_{7,8}+e_{8,9}$.

Let 
$\kk$-mod$^\wedge$ be the category consisting of 
$\{\bigwedge^\mu V\mid \mu \in \Lambda_{\text{st}}\}$ with $\kk$-linear maps as morphisms.
Recall we have already a functor $\Phi: \W \rightarrow \kk$-mod$^\wedge$ from  Proposition~ \ref{functorofweb}.
We aim at defining a functor 
\[
\hat\Phi: \AW \longrightarrow \kk \text{-mod}^\wedge  
\]
 (which is \emph{not} a monoidal functor).
 On the object level, $ \hat \Phi$ is the same as $\Phi$.
 They are also the same on the generating  morphisms (as a $\kk$-linear category)
$1_* \splits 1_*$, $1_* \merge 1_*$, $1_* \crossing 1_* $. It remains to define the action of  $1_\mu \wkdotr 1_\nu$ on $\bigwedge ^\mu V\otimes \bigwedge^r V\otimes \bigwedge^\nu V$.

We first define the action of 
$1_\mu \dotgen 1_\nu$ (under $\hat\Phi$) on 
$\bigwedge^\mu V\otimes V\otimes \bigwedge^\nu V$ for $\mu, \nu \in \Lambda_{\text{st}}$. Write the natural basis elements  of $\bigwedge ^\mu V$ as 
\[ 
v_{[\mathbf i]}:=v_{\mathbf i_1}\otimes \ldots \otimes v_{\mathbf i_h},
\]
where $ h=l(\mu)$ and $ v_{\mathbf i_j}=v_{i_{j,1}}\wedge \ldots\wedge v_{i_{j,\mu_j}}$ with $\mathbf i_j\in [N]^{\mu_j}$ and $i_{j,1}<\ldots< i_{j,\mu_j}$.
Similarly, we have a natural basis for $ \bigwedge ^\nu V$.

Let $\bfc=(c_1,\ldots,c_\ell)\in \kk^\ell$ be a fixed $\ell$-tuple.
As in \cite[Lemma 3.4]{BK08}, let $\bfu=(u_1,\ldots,u_\ell)\in \kk^\ell$ be given by 
\begin{align}  \label{BKui}
u_i=(c_i+\lambda'_i-n), \qquad 1\leq i\le \ell.
\end{align}
Recall the action of a dot on $V^{\otimes d}$ in \cite[Lemma 3.3]{BK08} is given as 
\begin{equation}
\label{equ:actionofcycheckev}
\begin{aligned}
    v_{\mathbf i} \cdot  1_{1^{k-1}} \dotgen 1_{1^{d-k}} 
& =
v_{i_1}\otimes \ldots \otimes v_{i_{k-1}}\otimes e(v_{i_k})\otimes v_{i_{k+1}}\otimes \ldots\otimes  v_{i_d}+ u_{\text{col}(i_k)}v_{\mathbf i}\\
&+ \sum_{\overset{1\le j<k}{\text{col}(i_j)\ge \text{col}(i_k)}}v_{\mathbf i \cdot (j,k)} - \sum_{\overset{k\le j\le d}{\text{col}(i_j)< \text{col}(i_k)}}v_{\mathbf i \cdot (j,k)}.
\end{aligned}
\end{equation}
By restricting the above action of a dot to the subspace  $\bigwedge^{\mu}V\otimes  V \otimes \bigwedge ^\nu V$  of $V^{\otimes d}$ with $d=m+n+1$, we have the following result which can be checked directly.

\begin{lemma}  \label{lem:dot}
The subspace $\bigwedge^{\mu}V\otimes  V \otimes \bigwedge ^\nu V$ is preserved by $1_{\mu}\dotgen 1_{\nu}$. More explicitly,  
$1_{\mu} \dotgen 1_{\nu}$ sends 
$v_{[\mathbf i],k, [\mathbf j]}:= v_{[\mathbf i]}\otimes v_k \otimes v_{[\mathbf j]}$ to 
\begin{equation}
\label{actionofdotonv}
 v_{[\mathbf i]}\otimes e(v_k) \otimes v_{[\mathbf j]}+u_{\text{col}(k)} v_{[\mathbf i],k, [\mathbf j]} +\sum_{(\mathbf i',k')\in S_{\mathbf i,k} } v_{[\mathbf i'],k',[\mathbf j] }- \sum_{(r',\mathbf j')\in S_{k,\mathbf j} } v_{[\mathbf i],r' ,[\mathbf j'] }    
\end{equation}  
where $S_{\mathbf i, k}$ is the set consisting of $(\mathbf i',k')$ such that $k'=i_{a,b}$ with $ \text{col}(i_{a,b})\ge \text{col}(k)$ for some $1\le a\le h$ and $1\le b\le \mu_a$, and $\mathbf i'$ is obtained from $\mathbf i=(\mathbf i_1, \ldots, \mathbf i_h)$ by substituting $i_{a,b}$ with $k$. Similarly, $S_{k,\mathbf j}$ consists of $(r',\mathbf j')$ such that $r'=j_{c,d}$ with $ \text{col}(j_{c,d})< \text{col}(k)$ for some $1\le c\le l(\nu)$ and $1\le d\le \nu_c$, and $\mathbf j'$ is obtained from $\mathbf j=(\mathbf j_1, \ldots, \mathbf j_t)$ by substituting $j_{c,d}$ with $k$.
\end{lemma}
 We then define the action of $1_\mu \dotgen 1_\nu$ on
$\bigwedge^{\mu}V\otimes  V \otimes \bigwedge ^\nu V$
by Lemma \ref{lem:dot} and \eqref{actionofdotonv}.

% Recall from \eqref{equ:r=0andr=a} that $\wkdota$ for $a\ge r$ is obtained as a composition of a split, $\wkdotr$ and a merge. Hence to define the action of $1_\mu \wkdota 1_\nu$ on $\bigwedge ^\mu V\otimes \bigwedge^a V\otimes \bigwedge^\nu V$, it suffices to define the action of $ 1_{\mu} \wkdotr 1_\nu $ on $\bigwedge ^\mu V\otimes \bigwedge^r V\otimes \bigwedge^\nu V$.
 
Now we are ready to define the action of $ 1_{\mu} \wkdotr 1_\nu $ on $\bigwedge ^\mu V\otimes \bigwedge^r V\otimes \bigwedge^\nu V$. Let $x_1,\ldots, x_r$ be the diagram of the image of $\phi$ in Proposition \ref{isotodeaffine} in $\End_{\AW}(1^r)$, i.e., $x_j$ is obtained from the identity morphism of $1^r$ by adding a dot on the $j$th strand. In \eqref{actionofdotonv} (or in \cite[(4.14)]{BCNR}), we already have the action of $1_{1^{m}}x_j1_{1^n}$ on $V^{\otimes m+r+n}$. We define the action of $ 1_{\mu} \wkdotr 1_\nu $ on $\bigwedge ^\mu V\otimes \bigwedge^r V\otimes \bigwedge^\nu V $ for any 
$\mu\in \Lambda_{\text{st}}(m) $ and 
$\nu\in \Lambda_{\text{st}}(n)$ to be the restriction of the action of $1_{1^{m}}\prod_{1\le j\le r}x_j1_{1^n}$ to the subspace $\bigwedge ^\mu V\otimes \bigwedge^r V\otimes \bigwedge^\nu V$.
One checks  that this is a well-defined  linear operator on $\bigwedge ^\mu V\otimes \bigwedge^r V\otimes \bigwedge^\nu V$ by using  the fact that $\prod_{1\le j\le r}x_j$ is central in $\End_{\AW}(1^r)$. 
% since 
% $\bigwedge ^\mu V\otimes \bigwedge^r V\otimes \bigwedge^\nu V=V^{\otimes m+r+n} \sum_{w\otimes \mathfrak S_\mu\times \mathfrak S_{r}\times \mathfrak S_{\nu}}w$
%  and 
%  \begin{align*}
%  \bigwedge ^\mu V\otimes \bigwedge^r V\otimes \bigwedge^\nu V (1_{1^{m}}\prod_{1\le j\le r}x_j1_{1^n})
%  &=
%  V^{\otimes m+r+n} \sum_{w\in \mathfrak S_\mu\times \mathfrak S_{r}\times \mathfrak S_{\nu}}w
%  (1_{1^{m}}\prod_{1\le j\le r}x_j1_{1^n})\\
% & = 
%  V^{\otimes m+r+n} (1_{1^{m}}\prod_{1\le j\le r}x_j1_{1^n}) \sum_{w\in \mathfrak S_\mu\times \mathfrak S_{r}\times \mathfrak S_{\nu}}w\\
% & \subset \bigwedge ^\mu V\otimes \bigwedge^r V\otimes \bigwedge^\nu V, 
%  \end{align*} 
%  where the second equality follows from

\begin{proposition}
\label{th:actionV}
The functor $\hat \Phi: \AW \rightarrow \kk \text{-mod}^\wedge$  above is well defined and  it factors through $\W_{\bfu}$ to a functor $\hat \Phi : \W_\bfu\rightarrow \kk \text{-mod}^\wedge$. 
\end{proposition}

\begin{proof}
It suffices to check that $\hat \Phi$ respects the relations between the generating morphisms
$1_* \splits 1_*$, $1_* \merge 1_*$, $1_* \crossing 1_* $ and $
1_* \dotgen 1_*
$ 
for $\AW$ as a $\kk$-linear category which involves the defining relations of $\AW$ as a monoidal category (i.e., \eqref{webassoc}--\eqref{intergralballon}) and the following additional  commuting relations 
\begin{equation}
\label{equ:commuting-relations}
   (1_\mathbf a\otimes g\otimes 1_{\bfc\otimes \mathbf d_1\otimes \mathbf e })\circ (1_{\mathbf a\otimes \mathbf b_1\otimes c}\otimes h\otimes 1_\mathbf e)= (1_{\mathbf a\otimes \mathbf b_2\otimes \bfc}\otimes h\otimes 1_{\mathbf  e })\circ (1_{\mathbf a}\otimes g\otimes 1_{\mathbf  c\otimes \mathbf d_1\otimes \mathbf e}) 
\end{equation} 
for all objects $\mathbf a,\bfc, \mathbf e\in \Lambda_{\text{st}}$ and all generating morphisms $g: \mathbf b_1\rightarrow \mathbf b_2$ and $h: \mathbf d_1\rightarrow \mathbf d_2$. We refer to \cite[\S 2.6]{BCNR} for general remarks on relations for a monoidal category without referring to the monoidal structure. 

The relations involving \eqref{webassoc}--\eqref{equ:r=0} and the commuting relations which do not involve the dots follow from those for the functor $\Phi$ given in  Proposition \ref{functorofweb}. The relation \eqref{intergralballon} follows by definition. 
Note that the action of $ 1_\mu \dotgen 1_\nu$ 
given in \eqref{actionofdotonv} and Lemma~\ref{lem:dot} is the restriction of 
$1_{1^m}\dotgen 1_{1^n}$ 
from  \cite[Lemma~ 3.3]{BK08} to the subspace $\bigwedge^\mu V\otimes V\otimes \bigwedge^\nu V$.
This implies the relation involving \eqref{dotmovecrossingC} over $\C$ since it holds in a larger space $V^{\otimes m+n+2}$ as proved in \cite[\S 3.4]{BK08}. Since the above action of generators are defined over the polynomial ring  $\Z[\tilde u_1,\ldots,\tilde u_\ell]$ in $\ell$ variables, the additional relations  \eqref{equ:dotmovecrossing-simplify}--\eqref{dotmovesplits+merge-simplify} follows from the simplified relations over $\C$ by the base change arguments.
By the same reasoning, the commuting relations involving two dots 
also follow from \cite[\S 3.4]{BK08}.
Finally, the commuting relations involving one dot and one split (or merge, or crossing) can be checked directly from the definition of the action given in \eqref{actionofdotonv}.

Write $g=\prod_{1\le j\le \ell}g_{1,u_j} =\prod_{1\le j\le \ell} (\dotgen -u_j )$.
It is proved in \cite[Lemma 3.4]{BK08} that  $g\otimes 1_{1^n}$ is annihilated by $\hat\Phi$. By restriction to $V\otimes \bigwedge ^\nu V$, we see that $\hat \Phi$ annihilates the right tensor ideal generated by $g$. Since $\hat \Phi(g) =\hat \Phi (\prod_{1\le j\le \ell}g_{1,u_j})=0$, we have 
$\hat \Phi(\prod_{1\le j\le \ell  }g_{r,u_j}) =0$, for any $r\ge 1$, over $\C$ by Lemma \ref{cyc-web-C}, and then over $\Z[\tilde u_1,\ldots,\tilde u_\ell]$ by a standard base change argument as above since $ \hat\Phi(\prod_{1\le j\le \ell  }g_{r,u_j} ) $ is also defined over the polynomial ring $\Z[\tilde u_1,\ldots,\tilde u_\ell]$. Therefore, $\hat \Phi$ factors through $\W_\bfu$. 
\end{proof}

\subsection{Finite $W$-algebras}

We assume that $\kk=\mathbb C$  in this subsection.

We begin by recalling some notations in \cite{BK08}. Recall $\mathfrak g:=\mathfrak {gl}_{N}(\kk)$ and the nilpotent matrix $e\in \mathfrak g$ from  the previous section; see, e.g., \eqref{pyramid-lambda}. Let $\mathfrak g_e$ be the centralizer of $e$ in $\mathfrak g$.  The universal enveloping algebra $U(\mathfrak g_e)$ admits a certain filtered deformation $W(\Lambda)$, the finite $W$-algebra associated with the nilpotent $e$.
More explicitly, set $\mathfrak p=\oplus_{r\ge 0}\mathfrak g_r$ and $\mathfrak m= \oplus_{r< 0}\mathfrak g_r$, where the $\Z$-grading of $\mathfrak g$ is defined by declaring $\deg (e_{i,j})=\text{col}(j)-\text{col}(i)$. Let $\eta$ be the algebra automorphism of $U(\mathfrak p)$ defined by sending $e_{i,j}$ to $e_{i,j}+\delta_{i,j}(n-q_{\text{col}(j)}-q_{\text{col}(j)+1}-\ldots-q_{\ell})$ for any $e_{i,j}\in\mathfrak p$, where we denote $\Lambda'=(q_1,\ldots, q_\ell)$. Then $W(\Lambda)$ is defined to be a subalgebra of  $U(\mathfrak p)$:
\[
W(\Lambda):= \{w\in U(\mathfrak p)\mid [x,\eta(y)]\in U(\mathfrak g)I_\chi \text{ for all } x\in\mathfrak m \}, 
\]
where  $I_\chi$ is the kernel of the homomorphism $\chi: U(\mathfrak m)\mapsto \C$ given by $x\mapsto (x,e)$ and $(\cdot,\cdot)$ is the trace form on $\mathfrak g$.

For  $d\in \N$, let $V_{\bfc}^{\otimes d}$ (identified as $V^{\otimes d}$ as vector spaces)  be the $U(\mathfrak p)$-module  with the action obtained by twisting the usual action by the automorphism $\eta_\bfc$ of $U(\mathfrak p)$ such that 
$\eta_{\bfc}(e_{i,j})=e_{i,j}+\delta_{i,j} c_{\text{col}(i)}$.
Similarly, we have the $U(\mathfrak p)$-submodule $\bigwedge^\mu V_{\bfc}$ (identified with $\bigwedge^\mu V$ as vector spaces) with the twisting action for any $\mu\in \Lambda_{\text{st}}$.
By restriction to $W(\Lambda)$, both $V_{\bfc}^{\otimes d}$
and $\bigwedge^\mu V_{\bfc}$ are $W(\Lambda)$-modules.

 Let $W(\Lambda)$-mod be the category of finite-dimensional $W(\Lambda)$-modules.
Hence, we have a subcategory of $W(\Lambda)$-mod
with the same objects as $\kk$-mod$^\wedge$ (i.e., by identifying the objects $\bigwedge^\mu V_{\bfc}$ with $\bigwedge^\mu V$).

\begin{proposition}
 \label{prop:Phihat}
Suppose $\kk=\mathbb C$. The functor $\hat \Phi$ in Proposition \ref{th:actionV}  gives rise to a functor  \begin{align}  \label{hPhi}
 \hat \Phi : \W_\bfu \longrightarrow W(\Lambda)\text{-mod}.
\end{align}
\end{proposition}

\begin{proof}
  The actions of splits, merges, and crossing commute with the action of $W(\Lambda)$ on $\bigwedge^\nu V_{\bfc}$ since they commute with the action of $\mathfrak g$. For the action of dots, it is explained in \cite[\S 3.4]{BK08} that they commute with the action of $W(\Lambda)$
 on $V_{\bfc}^{\otimes n}$. Then the same holds in the subspace $\bigwedge ^\nu V_{\bfc}$ since it is invariant under the actions of dots in $\W_\bfu$ and $ W(\Lambda)$.
\end{proof}

We shall establish in Theorem~\ref{thm:sameHom} a more precise relation between Hom-spaces in $\W_{\bfu}$ and Hom-spaces for $W(\Lambda)$-modules.

\subsection{The special idempotent $e_{\bf S}$ and Schur functor}

In this subsection, we review and rework some constructions from \cite{BK08} (where $\kk=\C$ is assumed) over any commutative ring $\kk$.

Recall the degenerate cyclotomic Hecke algebra $\mathcal{H}_{m,\bfu}$ with parameter $\bfu$ is the quotient of the affine Hecke algebra $\widehat {\mathcal{H}}_m$ (see \S\ref{affinehecke}) by the two-sided ideal generated by $\prod_{1\le i\le \ell}(x_1-u_i)$.

Let $\text{Tab}^m(\Lambda)$ be the set of tableaux obtained from the Young diagram of $\Lambda$ by filling non-negative integers in each box such that the sum of entries is
$m$.
Associated to each $A\in \text{Tab}^m(\Lambda)$, there is an $\ell$-multicomposition $\mu_A=(\mu_A^{(1)}, \ldots, \mu_A^{(\ell)})$ of $m$ such that 
$\mu_A^{(i)}$ is the $i$th column reading (from top to bottom).

A tableau $A\in \text{Tab}^m(\Lambda)$ is called an idempotent tableau if all entries of $A$ except the rightmost entries in each row are zero. Let $\text{Idem}^m(\Lambda)$ denote the set of all idempotent tableaux.

For any $\mathbf i,\mathbf j\in [N]^m$, let 
$e_{\mathbf i,\mathbf j}:=e_{i_1,j_1}\otimes \ldots\otimes e_{i_m,j_m}\in \End(V^{\otimes m})$.
Suppose that  $A\in\text{Idem}^m(\Lambda)$.
Let 
$\mathbf i(A)=(1^{a_1}, \ldots, N^{a_N})\in [N]^m$, where $(a_1,\ldots,a_N)$ is the row reading of $A$ starting from the top row.
Moreover, 
there is an idempotent $e_A$ (cf. \cite[(6.1)]{BK08}) of $\End(V^{\otimes m})$ defined as 
\begin{align}
    e_{A}=\sum_{\overset{\mathbf i\in [N]^m}{\text{row}(\mathbf i)\sim \text{row}(\mathbf i(A))} }e_{\mathbf i,\mathbf i}
\end{align}
where $\text{row}(\mathbf i)=(\text{row}(i_1),\ldots,\text{row}(i_m))$ and $\sim$ means lying in the same $\mathfrak S_m$-orbit.
For example, given 
\begin{align*}
   \begin{ytableau}
    0 & \none   \\
    0 & 1   \\
    0 & 1   \\
\end{ytableau}
\qquad
\in \text{Idem}^2(\Lambda) ,
\end{align*}
with $\Lambda=(1,2,2)$, we have $\mathbf i(A)=(3,5)$ and $e_A=e_{2,4}+e_{4,2}+e_{2,5}+e_{5,2}+e_{4,3}+e_{3,4}+e_{3,5}+e_{5,3}$.

It is clear that  
$\{e_A\mid A\in\text{Idem}^m(\Lambda)\}$ are mutually orthogonal idempotents summing to 1.
Moreover, one can check via \eqref{equ:actionofcycheckev} that $e_A$ commutes with the action of $\mathcal H_{m,\bfu}$ on $V^{\otimes m}_\mathbf c$.
%and hence $e_A\in W_{m,\bfu}(\Lambda)$.

Suppose that $\Lambda'_\ell\ge m$. We denote 
\[
\text{Idem}_\ell^m(\Lambda):=\{A\in \text{Idem}^m(\Lambda)\mid \mu_A^{(\ell)}\in \Lambda_{\text{st}}(m)\},
\]
that is, $\mu_A^{(i)}=\emptyset$ ($1\le i\le \ell-1$), for any $A\in\text{Idem}_\ell^m(\Lambda) $.

\begin{lemma}
\label{lem:ea=permu1}
  Suppose that $\Lambda'_\ell\ge m$.  We have $e_AV^{\otimes m}_\mathbf c= v_{\mathbf i(A)}\mathcal H_{m,\bfu}$, for any $A\in \text{Idem}_\ell^m(\Lambda)$.
\end{lemma}

\begin{proof}
First note that 
$e_AV^{\otimes m}_\mathbf c$ has a basis 
\[ \{v_{\mathbf i}\mid \mathbf i\in [N]^m, \text{row}(\mathbf i)\sim \text{row}(\mathbf i(A))\}. \]
 In particular, $v_{\mathbf i(A)}\in  e_AV^{\otimes m}_\mathbf c$  and hence 
 $v_{\mathbf i(A)}\mathcal H_{m,\bfu}\subset e_AV^{\otimes m}_\mathbf c$ since 
 $e_A\in W_{m,\bfu}(\Lambda)$ implies that  $e_AV^{\otimes m}_\mathbf c$ is a right $\mathcal H_{m,\bfu} $-module.

  We may check the converse inclusion by induction on the degree of $v_{\mathbf i}$, where $\deg (v_k)=\ell-\text{col}(k)$. The minimal degree (of degree 0) elements are of the form $v_{\mathbf i(A)w}$ for some $w\in \mathfrak S_m$ and hence are contained in $v_{\mathbf i(A)}\mathcal H_{m,\bfu}$. In general, suppose $v_\mathbf i\in e_AV^{\otimes m}_\mathbf c$ with $\deg (v_{\mathbf i})>0$. By permuting the positions we may assume that 
 $\text{col}(i_1)<\ell$. Let $\mathbf i'=(i_1+1,i_2,\ldots, i_m)$. Then
 $v_{\mathbf i'}\in e_AV^{\otimes m}_\mathbf c$
 with degree $\deg(v_\mathbf i)-1$.
 Using \eqref{equ:actionofcycheckev} we see that 
 $v_\mathbf i=v_{\mathbf i'}x_1+\sum v_\mathbf j$ for some $\mathbf j$ obtained from $\mathbf i'$ by place permutation. By induction on the degree, we have $v_{\mathbf i'}$ and all $v_\mathbf j$'s are contained in $v_{\mathbf i(A)}\mathcal H_{m,\bfu}$ and hence $v_\mathbf i\in v_{\mathbf i(A)}\mathcal H_{m,\bfu}$. This completes the proof.
\end{proof}

Let ${\bf S} \in \text{Idem}_\ell^m (\Lambda)$ be the special idempotent tableau with associated $\ell$-multicomposition 
\[
\mu_{\bf S}=(\emptyset^{\ell-1},(1^m)).
\]

\begin{lemma}
\label{lem:isomorm}
 Suppose $\Lambda'_\ell\ge m$.   We have $e_{\bf S} V^{\otimes m}_\mathbf c\cong \mathcal H_{m,\bfu}$ as right $\mathcal H_{m,\bfu}$-modules. In particular, 
    \begin{align}
    \label{equ:isomorphisofcychekce}
      \End_{\mathcal H_{m,\bfu}}(e_{\bf S}V^{\otimes m}_\mathbf c)\cong \End_{\mathcal H_{m,\bfu}}(\mathcal H_{m,\bfu})\cong \mathcal H_{m,\bfu}.  
    \end{align}
\end{lemma}

\begin{proof}
    Consider the right $\mathcal H_{m,\bfu}$-module homomorphism 
\begin{equation}
\label{equ:isoofrighthmodule}
  \begin{aligned} \phi: \mathcal H_{m,\bfu} &\longrightarrow e_{\bf S} V^{\otimes m}_\mathbf c\\
        1& \mapsto v_{\mathbf i({\bf S})}
        \end{aligned}
\end{equation}
  By Lemma \ref{lem:ea=permu1}, the above map is surjective. Considering the highest degree term of $v_{\mathbf i({\bf S})} wx_1^{n_1}\ldots x_m^{n_m}$ for any basis element of $\mathcal H_{m,\bfu}$, we see that $\phi$ is also injective and hence an isomorphism. 
\end{proof}
Via the isomorphism \eqref{equ:isomorphisofcychekce}, $\mathcal H_{m,\bfu}$ also acts on the left of $ e_{\bf S} V^{\otimes m}_\mathbf c$ and it becomes as an
$(\mathcal H_{m,\bfu},\mathcal H_{m,\bfu})$-bimodule.
This allows us to upgrade the isomorphism in \eqref{equ:isoofrighthmodule} as follows. 

\begin{lemma}
\label{lem:bimoduiso}
    Suppose $\Lambda'_\ell\ge m$.   We have $e_{\bf S} V^{\otimes m}_\mathbf c\cong \mathcal H_{m,\bfu}$ as  $\mathcal H_{m,\bfu}$-bimodules.
\end{lemma}
\begin{proof}
    It suffices to show that $\phi$ in \eqref{equ:isoofrighthmodule} is a left $\mathcal H_{m,\bfu}$-homomorphism.
    For any $h',h\in\mathcal H_{m,\bfu}$, we have 
    \[h'\phi(h)=h'(v_{\mathbf i({\bf S})} h)=(h'v_{\mathbf i({\bf S})}) h
   = (v_{\mathbf i({\bf S})}h')h=\phi(h'h) \]
   where the third equality follows since the left action of $\mathcal H_{m,\bfu}$ on $e_{\bf S} V^{\otimes m}_\mathbf c$ is given by 
   $\End_{\mathcal H_{m,\bfu}}(e_{\bf S}V^{\otimes m}_\mathbf c)$ via the isomorphism in \eqref{equ:isomorphisofcychekce}, i.e., left multiplication of $h'$ corresponds to sending $v_{\mathbf i({\bf S})}$ to $v_{\mathbf i({\bf S})}h'$.  The lemma is proved.
\end{proof}

 The permutation module $M^{\mu_A}$ is by definition (cf. \cite{DJM98}) the right ideal $\mathrm{x_{\mu_A}}\mathcal H_{m,\bfu} $ of $\mathcal H_{m,\bfu}$, where $\x_{\mu_A}=\sum_{w\in \mathfrak S_{\mu_A}}w$ for any  $A\in\text{Idem}_\ell^m(\Lambda) $.

Note that any left $\mathcal H_{m,\bfu}$-module can be regarded as a right $\mathcal H_{m,\bfu}$-module via the anti-involution 
 $*: \mathcal H_{m,\bfu}\rightarrow \mathcal H_{m,\bfu}$, $x_j\mapsto x_j, s_i\mapsto s_i$. This is how we view the module $e_{\bf S} \bigwedge^{\mu_A} V_\bfc$ as a right module in the lemma below. 
 
\begin{lemma} 
\label{lem:isopermutationv}
Suppose that $\Lambda'_\ell\ge m$.
    For any $A\in \text{Idem}_\ell^m(\Lambda)$, 
  we have   $e_{\bf S} \bigwedge^{\mu_A} V_\bfc \cong M^{\mu_A}$  as right $\mathcal H_{m,\bfu}$-modules.
\end{lemma}

\begin{proof}
Recall the action of $s_j$ on $V^{\otimes m}$ is given by the sign permutation and hence $V^{\otimes m}_\bfc \x_{\mu_{A}}= \bigwedge\nolimits^{\mu_A} V_{\bfc}$.
This together with the bimodule isomorphism in 
Lemma \ref{lem:bimoduiso}
implies the isomorphism of left $\mathcal H_{m,\bfu}$-modules
\[ 
\mathcal H_{m,\bfu} \x_{\mu_{A}} \cong e_{\bf S} V^{\otimes m}_\bfc \x_{\mu_{A}}=  e_{\bf S} \bigwedge\nolimits^{\mu_A} V_{\bfc}.
\]
Thus, $ e_{\bf S} \bigwedge^{\mu_A} V_{\bfc}\cong  \x_{\mu_{A}} \mathcal H_{m,\bfu}= M^{\mu_A}$ as right 
$\mathcal H_{m,\bfu}$-modules since $ \x_{\mu_A}$ is fixed under $*$.
\end{proof}

 The {\em $W$-Schur algebra} (also called the higher level Schur algebra in \cite[\S 3.7]{BK08}, where $\kk=\C$ was assumed) is defined as 
\begin{align}  \label{WSalg}
W_{m,\bfu}(\Lambda):= \End_{\mathcal H_{m,\bfu}}(V_\bfc^{\otimes m}).
\end{align}

Suppose that $\Lambda'_\ell\ge m$. It follows from the isomorphism in \eqref{equ:isomorphisofcychekce} and the isomorphism $e_{\bf S}W_{m,\bfu}(\Lambda)e_{\bf S} =\End_{\mathcal H_{m,\bfu}}(e_{\bf S}V^{\otimes m}_\mathbf c)$ that 
\begin{equation}
\label{equ:isoidemwalge}
e_{\bf S} W_{m,\bfu}(\Lambda) e_{\bf S}\cong \mathcal  H_{m,\bfu}.    
\end{equation}
See also \cite[Lemma 6.8]{BK08} over $\C$.
The Schur functor (which gives the equivalence of categories \cite[Theorem~ 5.10, Lemma 6.8]{BK08}) is given by the idempotent truncation functor associated to $e_{\bf S}$:
\begin{equation}
\label{equ:truncationfun}
  e_{\bf S}: W_{m,\bfu}(\Lambda)\text{-mod}\longrightarrow \mathcal H_{m,\bfu}\text{-mod},\quad  M\mapsto e_{\bf S}M  
\end{equation} 
via the isomorphism \eqref{equ:isoidemwalge},
where $\mathcal H_{m,\bfu}\text{-mod}$ is the category of left $\mathcal H_{m,\bfu}$-modules.

\subsection{Asymptotically faithfulness}

We are back to finite $W$-algebra $W(\Lambda)$, over $\kk=\C$, in this subsection.
The following result (generalizing the case for $\ell=1$ in \cite{CKM, BEEO}) will not be used elsewhere in this paper and its proof uses some facts which follow directly through connections to cyclotomic Schur category in \cite{SW2}.

\begin{theorem}
\label{thm:sameHom}
Suppose that $\kk=\C$. Then the functor $\hat \Phi:  \W_\mathbf u\rightarrow W(\Lambda)\text{-mod}$ in \eqref{hPhi} is asymptotically faithful in the sense that it induces an isomorphism
  $\Hom_{\W_\bfu}(\mu,\nu)\cong \Hom_{W(\Lambda)}(\bigwedge\nolimits^\mu V_\bfc, \bigwedge\nolimits^ \nu V_\bfc)$, for any $\mu,\nu \in \Lambda_{\text{st}}(m)$ with $m \le \Lambda'_\ell$.
\end{theorem}

\begin{proof}
Let the functor $\mathcal F'$ be the composition of the functors $\W_\bfu \stackrel{\mathcal F}{\rightarrow} \Sch_\bfu \stackrel{\mathcal G}{\rightarrow}  \Sc_\bfu$, where $\mathcal F$ from $\W_\bfu$ to the cyclotomic Schur category $\Sch_\bfu$ is fully faithful (see \cite[\S5.5]{SW2}) and $\mathcal G$ from $\Sch_\bfu$ to the category of cyclotomic Schur algebras is an isomorphism (see \cite[Theorem~4.5]{SW2}).  %{\cite[Theorem \ref{thm:G}]{SW2}}. 
Let $A, B\in \text{Idem}_\ell^m(\Lambda)$ be such that $\mu_A^{(\ell)}=\mu$ and $\mu_B^{(\ell)}=\nu$. We claim to have the following commutative diagram 
    \begin{displaymath}
    \xymatrix{
        \Hom_{\W_\bfu}(\mu,\nu)  \ar[dr]^{\mathcal F'} \ar[r]^{\hat \Phi~~\quad} & \Hom_{W(\Lambda)}(\bigwedge\nolimits^\mu V_\bfc, \bigwedge\nolimits^ \nu V_\bfc) \ar[d]^{e_{\bf S}} \\
     & \Hom_{\mathcal H_{m,\bfu}}(M^{\mu_A}, M^{\mu_B}) ,   
    }
\end{displaymath}
where $\mathcal F'$ is an isomorphism by the results cited above. 
Indeed, the commutativity of the diagram can be checked directly on generators 
$1_* \splits 1_*$, $1_* \merge 1_*$,  $
1_* 
\begin{tikzpicture}[baseline = 3pt, scale=0.5, color=\clr]
\draw[-,thick] (0,0) to[out=up, in=down] (0,1.4);
\draw(0,0.6) \bdot;
\node at (0,-.3) {$\scriptstyle 1$};
\end{tikzpicture}
1_*$. 
For example, the image of $1_* \merge 1_*$ for $a=\mu_i, b=\mu_{i+1}$ under $\mathcal F'$ and $e_{\bf S}\circ \hat\Phi$ is given by sending $\x_{\mu_A}$ to $\x_{\mu_A^{\vartriangle_i}}$. 

On the other hand, by Lemma \ref{lem:isopermutationv}, the map $e_{\bf S}$ in the vertical arrow is also an isomorphism since the functor $e_{\bf S}$ is an equivalence of categories. Hence $\hat\Phi$ is also an isomorphism. 
\end{proof}

Recall the finite $N$ version of the affine web category, $\AW_N$ from Remark~ \ref{rem:finiteaffineweb}. Let $\W_{\bfu,N}$ be the quotient category of  $\AW_N$  by the right tensor ideal generated by $\prod_{1\le j\le \ell }g_{r,u_j}$, for $r\ge 1$.

\begin{conjecture}
Suppose that $\kk=\mathbb C$.
   The functor $\hat \Phi:  \W_\bfu\longrightarrow W(\Lambda)\text{-mod}$ is full. Moreover, it induces a fully faithful functor from  $\W_{\bfu,N}$ to $W(\Lambda)\text{-mod}$.
\end{conjecture}
Note that the conjecture holds for $\ell=1$ 
%(without the assumption that $\Lambda'_{\ell}\ge m$) 
and for any field $\kk$; cf., e.g., \cite{BEEO}.

\subsection{Cyclotomic webs and $W$-Schur algebras}

In this subsection we work over any commutative ring $\kk$ with 1, and continue to assume that $\Lambda'_\ell\ge m$.

Then we have an idempotent algebra $e_\ell W_{m,\bfu}(\Lambda)e_\ell$ of $W_{m,\bfu}(\Lambda)$ with 
\[
e_\ell:= \sum_{A\in\text{Idem}_\ell^m(\Lambda) } e_A.
\]

For $A\in \text{Idem}_\ell^m(\Lambda)$, define the right $\mathcal H_{m,\bfu}$-module:
\[
N^{\mu_A}:={\rm y}_{\mu_A}\mathcal H_{m,\bfu}, 
\qquad
\text{ where }\; {\rm y}_{\mu_A}=\sum_{w\in \mathfrak S_{\mu_A}}(-1)^{l(w)}w.
\]

\begin{lemma}
\label{lem:permusign}
For any $A\in \text{Idem}_\ell^m(\Lambda)$, we have 
$e_A V^{\otimes m}_\mathbf c\cong N^{\mu_A}$ as 
 right $\mathcal H_{m,\bfu}$-modules.
\end{lemma}

\begin{proof}
Let $\tilde{\mathcal H}_{m,\bfu}$ be the cyclotomic Hecke algebra with the same generators as $\mathcal H_{m,\bfu}$ and with relation $x_{i+1}s_i=s_ix_i+1$ in replace of \eqref{xs}. There is an isomorphism    $\sigma: \mathcal H_{m,\bfu} \rightarrow \tilde{\mathcal H}_{m,\bfu} $ sending $s_i$ to $-s_i$ and fixing $x_j$'s. Moreover, $\sigma(N^{\mu_A})=\tilde M^{\mu_A}$, the permutation module of $\tilde{\mathcal H}_{m,\bfu}$ generated by $\rm x_{\mu_A}$.
  Under this isomorphism and applying the basis 
  $\{{\rm x}_{\mu_A} wx_1^{r_1}\ldots x_m^{r_m}\mid w\in (\mathfrak S_{\mu_A}/ \mathfrak S_m)_{min} , 0\le r_i\le \ell-1\}$ (e.g. \cite[Theorem 6.9]{BK08}) of $\tilde M^{\mu_A}$, we have the following basis of $N^{\mu_A}$:
  \[
  \mathcal M:=\{
  {\rm y}_{\mu_A} wx_1^{r_1}\ldots x_m^{r_m}\mid w\in (\mathfrak S_{\mu_A}/ \mathfrak S_m)_{\min}, 0\le r_i\le \ell-1\}.
  \]
  This implies that $v_{\mathbf i(A)}\mathcal H_{m,\bfu}$ has a spanning set $\{ v_{\mathbf i(A)}h \mid h\in \mathcal M\}$, which must be a basis of $v_{\mathbf i(A)}\mathcal H_{m,\bfu}$ by rank consideration.  Furthermore, the  obvious linear map 
  sends $v_{\mathbf i(A)}h $ to $v_{\mathbf i({\bf S})}{\rm y}_{\mu_A} h$ is an injective right $\mathcal H_{m,\bfu}$-module homomorphism. Composing this morphism with  
  the isomorphism in Lemma \ref{lem:isomorm}, we see that the image is $N^{\mu_A}$. The lemma follows.
\end{proof}

The locally unital algebra associated to the cyclotomic web category $\W_\bfu$ is 
\[
\mathfrak{Web}_{\bfu}:= \bigoplus_{\mu,\nu\in\Lambda_{\text{st}} }\Hom_{\W_\bfu}(\mu,\nu)= \bigoplus_{m\in \N}   \mathfrak{Web}_{m,\bfu}
\]
where 
\begin{align}
   \label{algWebu}
\mathfrak{Web}_{m, \bfu}:=\bigoplus_{\mu,\nu\in\Lambda_{\text{st}}(m) }\Hom_{\W_\bfu}(\mu,\nu) 
\end{align}
is an idempotent subalgebra of $\mathfrak{Web}_\bfu$.
Recall the isomorphism $\End_{\AW}(1^m)\cong \hat {\mathcal H}_m$ in Proposition \ref{isotodeaffine}.
This isomorphism induces an isomorphism
\[\End_{\W_\bfu}(1^m)\cong \mathcal H_{m,\bfu}\]
by the basis theorem of $\W_\bfu$ in Theorem~\ref{basis-cyc-web} and the well-known basis for $\mathcal H_{m,\bfu}$.
Then we see that $1_{1^m}\mathfrak {Web}_{m,\bfu}1_{1^m}\cong \mathcal H_{m,\bfu}$ and the idempotent truncation with respect to $1_{1^m}$ is the Schur functor.
\begin{theorem}  
\label{moritatheorem}
    Suppose that $\Lambda'_\ell\ge m$. Then 
    $e_\ell W_{m,\bfu}(\Lambda)e_\ell \cong  \mathfrak{Web}_{m,-\bfu}$. Furthermore, if $\kk=\C$, then the algebra $ \mathfrak{Web}_{m,-\bfu}$ is Morita equivalent to 
   the $W$-Schur algebra $W_{m,\bfu}(\Lambda)$.
\end{theorem}
(The sign difference in parameters is due to different conventions on slider relations here and affine Hecke relations in literature.)

\begin{proof}
There is an isomorphism between $\mathcal H_{m,-\bfu}$ and
$\tilde{\mathcal H}_{m,\bfu}$ sending $x_i$ to $-x_i$ and fixing $s_j$'s. Under this isomorphism, the permutation module $M^{\mu_A}$ corresponds to $\tilde M^{\mu_A}$. Then  
 we have 
\begin{align*}
   e_\ell W_{m,\bfu}(\Lambda)e_\ell
   &= \oplus_{A, B\in\text{Idem}_\ell^m(\Lambda)    }\Hom_{\mathcal{H}_{m,\bfu}}(e_AV^{\otimes r}_\bfc,e_B V^{\otimes r}_\bfc)\\
   & \cong \oplus _{A, B\in\text{Idem}_\ell^m(\Lambda)    }\Hom_{\mathcal{H}_{m,\bfu}}(N^{\mu_A}, N^{\mu_B}) \quad \text{ by Lemma \ref{lem:permusign}}\\
   & \cong \oplus_{A, B\in\text{Idem}_\ell^m(\Lambda)}
   \Hom_{\tilde {\mathcal H}_{m,\bfu}}(\tilde M^{\mu_A}, \tilde M^{\mu_B})\\
   & \cong \oplus_{A, B\in\text{Idem}_\ell^m(\Lambda)    }\Hom_{ \mathcal{H}_{m,-\bfu}}( M^{\mu_A},  M^{\mu_B})\\
   &\cong  \mathfrak{Web}_{m,-\bfu}, 
\end{align*}
  where the last equality follows from the basis theorem of $\W_{-\bfu}$ and its more detailed description in  \cite[\S5.5]{SW2}. %[\S \ref{subsec:proofbasis}]{SW2}.

  Note that $\mathcal H_{m,\bfu}\cong \mathcal H_{m,-\bfu}$
  via $s_i\mapsto -s_i$ and $x_i\mapsto -x_i$. Then 
the second statement follows from \eqref{equ:isoidemwalge} and  the fact in \cite[Theorem 5.10, Lemma~ 6.8]{BK08} that $W_{m,\bfu}(\Lambda)$ is Morita equivalent to $\mathcal H_{m,\bfu}\cong \mathcal H_{m,-\bfu}$ and $1_{1^m}\mathfrak {Web}_{m,-\bfu}1_{1^m}\cong \mathcal H_{m,-\bfu}$.
\end{proof}

\begin{rem}
For $\ell=1$, $\mathcal{H}_{m,\bfu}$ is the group algebra of symmetric group and 
$W_{m,\bfu}(\Lambda)$ is the usual Schur algebra. Theorem~\ref{moritatheorem} for $\ell=1$ specializes to the fact \cite{BEEO} that 
%the web category is isomorphic to the Schur category and hence 
the path algebras associated to the web category over $\C$ are Morita equivalent to the Schur algebras.  
\end{rem}

\subsection{The $W$-Schur category}
Recall $\Lambda =(\la_1, \ldots, \la_n)$ with $\la_n=\ell$. 
For any $a\in \N$, let $\Lambda^a=(\lambda_1,\ldots,\lambda_n,\ell^a)$ which is a composition of $N(a)=N+a\ell$.
Associated to Young diagram of  $\Lambda^a$, we have the associated idempotent $e(a)\in \End(V(a)^{\otimes m})$, where $V(a)$ is the natural module of $\mathfrak {gl}_{N(a)}$ with basis $\{v_i\mid i\in [N(a)]\}$. Moreover, we have the associated finite $W$-algebra $W(\Lambda^a)$ (over $\C$).

Since $\bfu$ defined in \eqref{BKui} remains unchanged when $\Lambda$ is replaced by $\Lambda^a$ for any $a$,
the same cyclotomic Hekce algebra $\mathcal H_{m,\bfu}$ acts on $V(a)^{\otimes  m}_\bfc$, for all $a$, via \eqref{equ:actionofcycheckev}.
Moreover, we have  the $W$-Schur algebra associated to $\Lambda^a$
\[W_{m,\bfu}(\Lambda^a):= \End_{\mathcal H_{m,\bfu}}(V(a)^{\otimes m}_\bfc ). \]
Define the idempotent  $\mathfrak e_a:=\sum_Ae_A$ of $W_{m,\bfu}(\Lambda^a)$,
where the sum is over all $A\in \text{Idem}^m(\Lambda^a)$ such that all entries in the last $a$ rows of $A$ are zero.
It was proved in \cite[Lemma 6.3]{BK08}
that $\mathfrak e_a W_{m,\bfu}(\Lambda^a)\mathfrak e_a\cong W_{m,\bfu}(\Lambda)$.

We want to formulate an infinite version of $W_{m,\bfu}(\Lambda^a)$, i.e, for $a=\infty$. 
Note that  we also have the action of $\mathcal H_{m,\bfu}$ on $V(\infty)^{\otimes m}_{\bfc}$ via \eqref{equ:actionofcycheckev} (by considering the action restricting to subspaces $V(a)^{\otimes m}_{\bfc}$).
Then we have the $W$-Schur algebra
$W_{m,\bfu}(\Lambda^\infty):=\End_{\mathcal H_{m,\bfu}}( V(\infty)^{\otimes m}_{\bfc})$.
Let $ W_{m,\bfu}(\Lambda^\infty)$-mod be the category of finite-dimensional modules of $ W_{m,\bfu}(\Lambda^\infty)$.

Associated to the locally unital algebra $\bigoplus_{m\in \N}e_\ell^m W_{m,\bfu}(\Lambda^\infty)e_\ell^m$, where 
\[
e_\ell^m :=\sum_{A\in\text{Idem}_\ell^m(\Lambda^\infty) } e_A,
\]
we define a {\em $W$-Schur category} $\WSch_\bfu$ as follows.  
The object set of $\WSch_\bfu$ is $\Lambda_{\text{st}}$. For any $\mu\in \Lambda_{\text{st}}(m)$ and $\nu \in \Lambda_{\text{st}}(m')$,
the morphism space $\Hom_{\WSch_\bfu}(\mu,\nu)=0$ unless $m=m'$; in case $m=m'$, we have 
\begin{align}
  \label{HomWSu}
    \Hom_{\WSch_\bfu}(\mu,\nu) =e_BW_{m,\bfu}(\Lambda^\infty) e_A, \qquad \text{for } \mu_A=(\emptyset^{\ell-1},\mu), \;\; \mu_B=(\emptyset^{\ell-1},\nu). 
\end{align}
The following result follows from Theorem \ref{moritatheorem}.

\begin{theorem}
\label{thm:CycWeb=WSchur}
The category $\WSch_\bfu$ is isomorphic to $\W_{-\bfu}$. Furthermore, if $\kk=\C$, the category $ W_{m,\bfu}(\Lambda^\infty)$-mod is  equivalent to $\mathfrak {Web}_{m,-\bfu}$-mod,  for any $m\ge 1$.
\end{theorem}

Note that for $\ell=1$ the category $\WSch_\bfu$ is the Schur category $\Sch$, and hence the above result specializes to the isomorphism $\W \cong \Sch$ \cite[Theorem~ 4.10]{BEEO}.

\begin{conjecture}
\label{conj:Morita}
The Morita equivalence in 
 Theorem~ \ref{moritatheorem} and the Morita equivalence in Theorem~\ref{thm:CycWeb=WSchur} remain valid over any field $\kk$ of arbitrary characteristic. 
\end{conjecture}

The conjecture holds when $\ell=1$; cf. \cite{BEEO}.
Note that the usual Schur algebra (i.e., for $\ell=1$) is not Morita equivalent to Hecke algebra over a field $\kk$ of positive characteristic. 

\section{Affine web category over $\C$}
\label{sec:C}

Throughout this section, we set $\kk$ to be any field of characteristic zero, e.g., $\C$. Under such an assumption on $\kk$, we greatly simplify the presentations of $\AW$ and its cyclotomic quotients. This also helps us to understand better the origins of some more involved general relations in the earlier presentation of $\AW$  (for general $\kk$, e.g., $\kk =\Z$). 

\subsection{A reduced presentation for $\AW$ over $\C$}
 \label{affineweboverC}

Definition~ \ref{def-affine-web} is simplified as follows.

\begin{definition} 
\label{def-affine-webC}
Let $\kk$ to be any field of characteristic zero. 
The affine web category $\AWC$ is the strict $\kk$-linear monoidal category generated by objects $a\in \mathbb Z_{\ge 1}$. The morphisms are generated by 
\begin{align}
\label{merge+split+crossing C}
\begin{tikzpicture}[baseline = -.5mm,color=\clr]
	\draw[-,line width=1pt] (0.28,-.3) to (0.08,0.04);
	\draw[-,line width=1pt] (-0.12,-.3) to (0.08,0.04);
	\draw[-,line width=1.5pt] (0.08,.4) to (0.08,0);
        \node at (-0.22,-.4) {$\scriptstyle a$};
        \node at (0.35,-.4) {$\scriptstyle b$};\node at (0,.55){$\scriptstyle a+b$};\end{tikzpicture} 
&:(a,b) \rightarrow (a+b),&
\begin{tikzpicture}[baseline = -.5mm,color=\clr]
	\draw[-,line width=1.5pt] (0.08,-.3) to (0.08,0.04);
	\draw[-,line width=1pt] (0.28,.4) to (0.08,0);
	\draw[-,line width=1pt] (-0.12,.4) to (0.08,0);
        \node at (-0.22,.5) {$\scriptstyle a$};
        \node at (0.36,.5) {$\scriptstyle b$};
        \node at (0.1,-.45){$\scriptstyle a+b$};
\end{tikzpicture}
&:(a+b)\rightarrow (a,b),&
\begin{tikzpicture}[baseline=-.5mm,color=\clr]
	\draw[-,line width=1pt] (-0.3,-.3) to (.3,.4);
	\draw[-,line width=1pt] (0.3,-.3) to (-.3,.4);
        \node at (0.3,-.42) {$\scriptstyle b$};
        \node at (-0.3,-.42) {$\scriptstyle a$};
         \node at (0.3,.55) {$\scriptstyle a$};
        \node at (-0.3,.55) {$\scriptstyle b$};
\end{tikzpicture}
&:(a,b) \rightarrow (b,a),
\end{align}
(called the merges, splits and crossings, respectively), and 
\begin{equation}
\label{dotgeneratorC}
 \begin{tikzpicture}[baseline = 3pt, scale=0.5, color=\clr]
\draw[-,thick] (0,0) to[out=up, in=down] (0,1.4);
\draw(0,0.6) \bdot;
\node at (0,-.3) {$\scriptstyle 1$};
\end{tikzpicture} \;,
\end{equation}
subject to the following relations \eqref{webassoc C}--\eqref{dotmovecrossingC}, for $a,b,c,d \in \Z_{\ge 1}$ with $d-a=c-b$:
\begin{align}
\label{webassoc C}
\begin{tikzpicture}[baseline = 0,color=\clr]
	\draw[-,thick] (0.35,-.3) to (0.08,0.14);
	\draw[-,thick] (0.1,-.3) to (-0.04,-0.06);
	\draw[-,line width=1pt] (0.085,.14) to (-0.035,-0.06);
	\draw[-,thick] (-0.2,-.3) to (0.07,0.14);
	\draw[-,line width=1.5pt] (0.08,.45) to (0.08,.1);
        \node at (0.45,-.41) {$\scriptstyle c$};
        \node at (0.07,-.4) {$\scriptstyle b$};
        \node at (-0.28,-.41) {$\scriptstyle a$};
\end{tikzpicture}
&=
\begin{tikzpicture}[baseline = 0, color=\clr]
	\draw[-,thick] (0.36,-.3) to (0.09,0.14);
	\draw[-,thick] (0.06,-.3) to (0.2,-.05);
	\draw[-,line width=1pt] (0.07,.14) to (0.19,-.06);
	\draw[-,thick] (-0.19,-.3) to (0.08,0.14);
	\draw[-,line width=1.5pt] (0.08,.45) to (0.08,.1);
        \node at (0.45,-.41) {$\scriptstyle c$};
        \node at (0.07,-.4) {$\scriptstyle b$};
        \node at (-0.28,-.41) {$\scriptstyle a$};
\end{tikzpicture}\:,
\qquad
\begin{tikzpicture}[baseline = -1mm, color=\clr]
	\draw[-,thick] (0.35,.3) to (0.08,-0.14);
	\draw[-,thick] (0.1,.3) to (-0.04,0.06);
	\draw[-,line width=1pt] (0.085,-.14) to (-0.035,0.06);
	\draw[-,thick] (-0.2,.3) to (0.07,-0.14);
	\draw[-,line width=1.5pt] (0.08,-.45) to (0.08,-.1);
        \node at (0.45,.4) {$\scriptstyle c$};
        \node at (0.07,.42) {$\scriptstyle b$};
        \node at (-0.28,.4) {$\scriptstyle a$};
\end{tikzpicture}
=\begin{tikzpicture}[baseline = -1mm, color=\clr]
	\draw[-,thick] (0.36,.3) to (0.09,-0.14);
	\draw[-,thick] (0.06,.3) to (0.2,.05);
	\draw[-,line width=1pt] (0.07,-.14) to (0.19,.06);
	\draw[-,thick] (-0.19,.3) to (0.08,-0.14);
	\draw[-,line width=1.5pt] (0.08,-.45) to (0.08,-.1);
        \node at (0.45,.4) {$\scriptstyle c$};
        \node at (0.07,.42) {$\scriptstyle b$};
        \node at (-0.28,.4) {$\scriptstyle a$};
\end{tikzpicture}\:,
\\
\label{mergesplitC}
\begin{tikzpicture}[baseline = 7.5pt,scale=.9, color=\clr]
	\draw[-,line width=1pt] (0,0) to (.275,.3) to (.275,.7) to (0,1);
	\draw[-,line width=1pt] (.6,0) to (.315,.3) to (.315,.7) to (.6,1);
        \node at (0,1.13) {$\scriptstyle b$};
        \node at (0.63,1.13) {$\scriptstyle d$};
        \node at (0,-.1) {$\scriptstyle a$};
        \node at (0.63,-.1) {$\scriptstyle c$};
\end{tikzpicture}
&=
\sum_{\substack{0 \leq s \leq \min(a,b)\\0 \leq t \leq \min(c,d)\\t-s=d-a}}
\begin{tikzpicture}[baseline = 7.5pt,scale=.9, color=\clr]
	\draw[-,thick] (0.58,0) to (0.58,.2) to (.02,.8) to (.02,1);
	\draw[-,thick] (0.02,0) to (0.02,.2) to (.58,.8) to (.58,1);
	\draw[-,thin] (0,0) to (0,1);
	\draw[-,line width=1pt] (0.61,0) to (0.61,1);
        \node at (0,1.13) {$\scriptstyle b$};
        \node at (0.6,1.13) {$\scriptstyle d$};
        \node at (0,-.1) {$\scriptstyle a$};
        \node at (0.6,-.1) {$\scriptstyle c$};
        \node at (-0.1,.5) {$\scriptstyle s$};
        \node at (0.77,.5) {$\scriptstyle t$};
\end{tikzpicture},
\\
\label{splitmergeC}
\begin{tikzpicture}[baseline = -1mm,scale=.7,color=\clr]
	\draw[-,line width=1.5pt] (0.08,-.8) to (0.08,-.5);
	\draw[-,line width=1.5pt] (0.08,.3) to (0.08,.6);
\draw[-,thick] (0.1,-.51) to [out=45,in=-45] (0.1,.31);
\draw[-,thick] (0.06,-.51) to [out=135,in=-135] (0.06,.31);
        \node at (-.33,-.05) {$\scriptstyle a$};
        \node at (.45,-.05) {$\scriptstyle b$};
\end{tikzpicture}
&= 
\binom{a+b}{a}\:\:
\begin{tikzpicture}[baseline = -1mm,scale=.6,color=\clr]
	\draw[-,line width=1.5pt] (0.08,-.8) to (0.08,.6);
        \node at (.08,-1.2) {$\scriptstyle a+b$};
\end{tikzpicture},
\\
\label{dotmovecrossingC}
\begin{tikzpicture}[baseline = 7.5pt, scale=0.4, color=\clr]
\draw[-,thick] (0,-.2) to  (1,2.2);
\draw[-,thick] (1,-.2) to  (0,2.2);
\draw(0.2,1.6)\bdot;
\node at (0, -.5) {$\scriptstyle 1$};
\node at (1, -.5) {$\scriptstyle 1$};
\end{tikzpicture} 
&=
\begin{tikzpicture}[baseline = 7.5pt, scale=0.4, color=\clr]
\draw[-, thick] (0,-.2) to (1,2.2);
\draw[-,thick] (1,-.2) to(0,2.2);
\draw(.8,0.3)\bdot;
\node at (0, -.5) {$\scriptstyle 1$};
\node at (1, -.5) {$\scriptstyle 1$};
\end{tikzpicture} 
+
\begin{tikzpicture}[baseline = 7.5pt, scale=0.4, color=\clr]
\draw[-, thick] (0,.5) to (0,2.2);
\draw[-, thick] (0,-.2) to (0,.5);
\draw[-, thick]   (1,1.8) to (1,2.2); 
\draw[-, thick] (1,1.8) to (1,-.2); 
\node at (0, -.5) {$\scriptstyle 1$};
\node at (1, -.5) {$\scriptstyle 1$};
\end{tikzpicture},
\qquad
\begin{tikzpicture}[baseline = 7.5pt, scale=0.4, color=\clr]
\draw[-, thick] (0,-.2) to (1,2.2);
\draw[-,thick] (1,-.2) to(0,2.2);
\draw(.2,0.2)\bdot;
\node at (0, -.5) {$\scriptstyle 1$};
\node at (1, -.5) {$\scriptstyle 1$};
\end{tikzpicture}
= 
\begin{tikzpicture}[baseline = 7.5pt, scale=0.4, color=\clr]
\draw[-,thick] (0,-.2) to  (1,2.2);
\draw[-,thick] (1,-.2) to  (0,2.2);
\draw(0.8,1.6)\bdot;
\node at (0, -.5) {$\scriptstyle 1$};
\node at (1, -.5) {$\scriptstyle 1$};
\end{tikzpicture}
+
\begin{tikzpicture}[baseline = 7.5pt, scale=0.4, color=\clr]
\draw[-,thick] (0,-.3) to (0,2.2);
\draw[-,thick] (1,2.2) to (1,-.3); 
\node at (0, -.6) {$\scriptstyle 1$};
\node at (1, -.6) {$\scriptstyle 1$};
\end{tikzpicture}.
\end{align}
\end{definition}

\begin{theorem}
\label{thm:AWisom}
   The monoidal categories $\AWC$ and $\AW$ are isomorphic.
\end{theorem}  

\begin{proof}
To relate $\AWC$ to $\AW$, we define the following morphisms  in $\AWC$ which correspond to the generators $\wkdotr$, for $ r\in\N$, in $\AW$:
\begin{equation}
\label{newgendot}
% \begin{split}
\wkdotr :=\frac{1}{r!}
\begin{tikzpicture}[baseline = 1.5mm, scale=.5, color=\clr]
\draw[-, line width=1.2pt] (0.5,2) to (0.5,2.5);
\draw[-, line width=1.2pt] (0.5,0) to (0.5,-.4);
\draw[-,thin]  (0.5,2) to[out=left,in=up] (-.5,1)
to[out=down,in=left] (0.5,0);
\draw[-,thin]  (0.5,2) to[out=left,in=up] (0,1)
 to[out=down,in=left] (0.5,0);      
\draw[-,thin] (0.5,0)to[out=right,in=down] (1.5,1)
to[out=up,in=right] (0.5,2);
\draw[-,thin] (0.5,0)to[out=right,in=down] (1,1) to[out=up,in=right] (0.5,2);
\node at (0.5,.7){$\scriptstyle \cdots$};
\draw (-0.5,1) \bdot; 
%\node at (-.8,1) {$\scriptstyle \omega_1$}; 
\draw (0,1) \bdot; 
%\node at (0.3,1) {$\scriptstyle \omega_1$};
\node at (0.5,-.6) {$\scriptstyle r$};
\draw (1,1) \bdot;
\draw (1.5,1) \bdot; 
%\node at (1.8,1) {$\scriptstyle \omega_1$};
\node at (-.22,0) {$\scriptstyle 1$};
\node at (1.2,0) {$\scriptstyle 1$};
\node at (.3,0.3) {$\scriptstyle 1$};
\node at (.7,0.3) {$\scriptstyle 1$};
\end{tikzpicture} \; .
% \\
% \wkdota &:=
% \begin{tikzpicture}[baseline = -1mm,scale=1,color=\clr]
% \draw[-,line width=2pt] (0.08,-.3) to (0.08,0);
% \draw[-,line width=2pt] (0.08,.8) to (0.08,1.1);
% \draw[-,thick] (0.1,-.01) to [out=45,in=-45] (0.1,.81);
% \draw[-,thick] (0.06,-.01) to [out=135,in=-135] (0.06,.81);
% \draw(-.1,0.5) \bdot;
% \draw(-.4,0.5)node {$\scriptstyle \omega_r$};
% \node at (-.3,.15) {$\scriptstyle r$};
% \node at (.6,.15) {$\scriptstyle a-r$};
% \end{tikzpicture}  \qquad\qquad (a\ge r).
% \end{split}
\end{equation}
There is a natural functor $\Psi: \AW \rightarrow \AWC$, matching the generating objects and generating morphisms in the same notation.
%for $\AW$ and $\AWC$. 

The only difference now is that there are weaker defining relations \eqref{webassoc C}--\eqref{dotmovecrossingC} for $\AWC$ than defining relations \eqref{webassoc}--\eqref{intergralballon} for $\AW$. 
Indeed, the relations \eqref{webassoc C}--\eqref{splitmergeC} for $\AWC$ are the same as \eqref{webassoc}--\eqref{equ:r=0} for $\AW$. Note also that the relations \eqref{dotmovecrossingC} in $\AWC$ are special cases of \eqref{equ:dotmovecrossing-simplify} in $\AW$.

To show that $\Psi$ is an isomorphism, it suffices to verify that the remaining defining relations \eqref{equ:dotmovecrossing-simplify}--\eqref{intergralballon} hold for $\AWC$.
%(as relations \eqref{webassoc}--\eqref{mergesplit} already hold). 

The relation \eqref{intergralballon} holds for $\AWC$ by definition \eqref{newgendot}. 

The remaining relations \eqref{equ:dotmovecrossing-simplify}--\eqref{dotmovesplits+merge-simplify}  in $\AWC$ are established in Lemma~\ref{definadotred} and Lemma~\ref{lem:dotmovemersplitC} below, respectively. 

This proves that $\Psi$ is an isomorphism.
\end{proof}

\subsection{Additional relations over $\C$}

We establish relations in Lemma~\ref{definadotred} and Lemma~\ref{lem:dotmovemersplitC} as required in the proof of Theorem~\ref{thm:AWisom} above. 

\begin{lemma}  
\label{definadotred}
The relations \eqref{equ:dotmovecrossing-simplify} hold in $\AWC$, 
that is,
\begin{equation*}
 .
\end{equation*}
\end{lemma}

\begin{proof}
  We only prove the first relation, as the second one is obtained by symmetry.
 We proceed by induction on $b$ and $a$. The case for  $a=1=b$ is \eqref{dotmovecrossingC}.
Suppose that $b=1$ and  $a\ge 2$. 
We have 
\begin{equation}
\label{r=1}
\begin{aligned}
%
,
\quad \text{by inductive assumption on $a-1$}.
    \end{aligned}
\end{equation}

Suppose $b\ge 2$. We have 
\begin{align*}
 %
.
\end{align*}
This completes the proof of the relation.
\end{proof}

\begin{lemma}
\label{lem:dotmovemersplitC}
The relations \eqref{dotmovesplits+merge-simplify} hold in $\AWC$, that is, 
\begin{equation*}
%
.
\end{equation*}
\end{lemma}

\begin{proof}
    We only prove the first relation and the second one follows by symmetry. 
We proceed by induction on $a+b$. For the (first nontrivial) basic case when $a=b=1$, we have 
\begin{align*}
    %
.
\end{align*}
Suppose $a+b\ge 2$. Below we simplify the notation $%
  \quad \text{ by induction on $a+b-1$}\\
&
\overset{\eqref{dotmovecrossing}\eqref{newgendot}}{\underset{\eqref{webassoc C}}{=}}
\frac{a}{r}~
%
 \; ,
\end{align*}
where $r=a+b$.
The lemma is proved. 
\end{proof}

\begin{rem}  [An alternative definition of $\AWC$]
    It is possible to define the monoidal category $\AWC$ using only the merges, splits and the dotted strand as generating morphisms (i.e., no crossings; cf. Remark \ref{rem:cross} for $\W$) together with relations \eqref{replcecross1}--\eqref{replacecross2} below in place of \eqref{dotmovecrossingC}:
\begin{align}
\label{replcecross1}
%
.
\end{align} 
\end{rem}

\subsection{Cyclotomic web categories over $\C$}
\label{subsec:cycwebC}

\begin{definition}
Let $\kk$ be any field of characteristic zero. For any $\mathbf u=(u_1,\ldots, u_\ell) \in \kk^\ell$ and $\ell\ge 1$,  let $\W'_{\mathbf u}$ be the quotient category of $\AW$ by the right tensor ideal generated by $\prod_{1\le j\le \ell }g_{1,u_j}$. 
\end{definition}

The following result shows that the above definition simplifies the definition of cyclotomic web category $\W_\mathbf u$ over $\mathbb C$ (or any field of characteristic zero).
\begin{theorem}  \label{th:WebuSame}
    The categories $\W'_{\mathbf u}$ and $\W_{\mathbf u}$ are isomorphic.
\end{theorem}
The proof of Theorem \ref{th:WebuSame} follows by definitions of $\W'_{\mathbf u}$ and $\W_{\mathbf u}$ together with the following.

\begin{lemma}
\label{cyc-web-C}
Suppose $\kk=\C$. In $\AW$, we have
\begin{equation}
\label{equ:cycdoymove}
r! \prod_{1\le j\le \ell  }g_{r,u_j}= \begin{tikzpicture}[baseline = 5pt, scale=.5, color=\clr]
\draw[-, line width=1.2pt] (0.5,2) to (0.5,2.3);
\draw[-, line width=1.2pt] (0.5,0) to (0.5,-.3);
\draw[-,thin](0.5,2) to[out=left,in=up] (-.5,1) to[out=down,in=left] (0.5,0);
\draw[-,thin]  (0.5,2) to[out=left,in=up] (0,1) to[out=down,in=left] (0.5,0);   
\draw[-,thin] (0.5,0)to[out=right,in=down] (1.5,1)to[out=up,in=right] (0.5,2);
\draw[-,thin] (0.5,0)to[out=right,in=down] (1,1)
 to[out=up,in=right] (0.5,2);
\node at (0.5,.7){$\scriptstyle \cdots$};
\draw (-0.5,1) \bdot; 
 \node at (-.75,1) {$\scriptstyle f$}; 
\draw (0,1) \bdot; 
\node at (0.3,1) {$\scriptstyle f$};
\draw (1,1) \bdot;
\draw (1.5,1) \bdot; 
\node at (1.8,1) {$\scriptstyle f$};
\node at (-.4,0) {$\scriptstyle 1$};
 \node at (.2,0.3) {$\scriptstyle 1$};
\node at (.7,0.3) {$\scriptstyle 1$};
\node at (1.2,0) {$\scriptstyle 1$};
\draw (.5,-.5) node{$\scriptstyle {r}$};
\end{tikzpicture}    
\end{equation}
for any  $\ell\ge 1$, $u_1,\ldots, u_\ell\in\C$ and $r\in \Z_{>0}$, 
where $f=\prod_{1\le j\le \ell}g_{1,u_j}$.
    \end{lemma}
    
\begin{proof}
    Using \eqref{dotmovesplits+merge-simplify} repeatedly, we have 
\[ 
\begin{tikzpicture}[baseline = 1.5mm, scale=.4, color=\clr]
%\draw[-, line width=1.5pt] (0.5,2) to (0.5,2.5);
\draw[-, line width=1.2pt] (0.5,0) to (0.5,-.8);
\draw (0.5,-.51) \bdot;
\node at (1.2,-.5) {$\scriptstyle \omega_r$};
\node at (0.5,-1.1) {$\scriptstyle r$};
\draw[-,thin]  (-0.5,1.5) to (-.5,1)
to[out=down,in=left] (0.5,0);
\draw[-,thin]  (0,1.5) to (0,1)
 to[out=down,in=left] (0.5,0);      
\draw[-,thin] (0.5,0)to[out=right,in=down] (1.5,1)
to (1.5,1.5);
\draw[-,thin] (0.5,0)to[out=right,in=down] (1,1) to (1,1.5);
\node at (0.5,.7){$\scriptstyle \cdots$};
%\draw (-0.5,1) \bdot; 
%\draw (0,1) \bdot; 
%\draw (1,1) \bdot;
%\draw (1.5,1) \bdot; 
\node at (-.5,1.8) {$\scriptstyle 1$};
\node at (0,1.8) {$\scriptstyle 1$};
\node at (1,1.8) {$\scriptstyle 1$};
\node at (1.5,1.8) {$\scriptstyle 1$};
\end{tikzpicture}
=
\begin{tikzpicture}[baseline = 1.5mm, scale=.4, color=\clr]
%\draw[-, line width=2pt] (0.5,2) to (0.5,2.5);
\draw[-, line width=1.2pt] (0.5,0) to (0.5,-.8);
%\draw (0.5,-.51) \bdot;
%\node at (1,-.5) {$\scriptstyle \omega_r$};
\node at (0.5,-1.1) {$\scriptstyle r$};
\draw[-,thin]  (-0.5,1.5) to (-.5,1)
to[out=down,in=left] (0.5,0);
\draw[-,thin]  (0,1.5) to (0,1)
 to[out=down,in=left] (0.5,0);      
\draw[-,thin] (0.5,0)to[out=right,in=down] (1.5,1)
to (1.5,1.5);
\draw[-,thin] (0.5,0)to[out=right,in=down] (1,1) to (1,1.5);
\node at (0.5,.7){$\scriptstyle \cdots$};
\draw (-0.5,1) \bdot; 
\draw (0,1) \bdot; 
\draw (1,1) \bdot;
\draw (1.5,1) \bdot; 
\node at (-.5,1.8) {$\scriptstyle 1$};
\node at (0,1.8) {$\scriptstyle 1$};
\node at (1,1.8) {$\scriptstyle 1$};
\node at (1.5,1.8) {$\scriptstyle 1$};
\end{tikzpicture}.
\]
There is an obvious isomorphism $\sigma_u$ of $\AW$ which fixes all splits, merges and crossings, and sends 
\begin{tikzpicture}[baseline = 3pt, scale=0.4, color=\clr]
\draw[-,thick] (0,0) to[out=up, in=down] (0,1.4);
\draw(0,0.6) \bdot;
\node at (0,-.3) {$\scriptstyle 1$};
\end{tikzpicture} to 
$
\begin{tikzpicture}[baseline = 3pt, scale=0.4, color=\clr]
\draw[-,thick] (0,0) to[out=up, in=down] (0,1.4);
\draw(0,0.6) \bdot;
\node at (0,-.3) {$\scriptstyle 1$};
\end{tikzpicture}
-u \begin{tikzpicture}[baseline = 3pt, scale=0.4, color=\clr]
\draw[-,thick] (0,0) to[out=up, in=down] (0,1.4);
%\draw(0,0.6) \bdot;
\node at (0,-.3) {$\scriptstyle 1$};
\end{tikzpicture}$ 
for any $u\in\kk$.
Applying $\sigma_u$ to the above equation yields
\begin{equation}
\label{omegaovertor}
  \begin{tikzpicture}[baseline = 1.5mm, scale=.4, color=\clr]
%\draw[-, line width=2pt] (0.5,2) to (0.5,2.5);
\draw[-, line width=1.2pt] (0.5,0) to (0.5,-.8);
\draw (0.5,-.51) \bdot;
\node at (1.4,-.5) {$\scriptstyle g_{r,u}$};
\node at (0.5,-1.1) {$\scriptstyle r$};
\draw[-,thin]  (-0.5,1.5) to (-.5,1)
to[out=down,in=left] (0.5,0);
\draw[-,thin]  (0,1.5) to (0,1)
 to[out=down,in=left] (0.5,0);      
\draw[-,thin] (0.5,0)to[out=right,in=down] (1.5,1)
to (1.5,1.5);
\draw[-,thin] (0.5,0)to[out=right,in=down] (1,1) to (1,1.5);
\node at (0.5,.7){$\scriptstyle \cdots$};
%\draw (-0.5,1) \bdot; 
%\draw (0,1) \bdot; 
%\draw (1,1) \bdot;
%\draw (1.5,1) \bdot; 
\node at (-.5,1.8) {$\scriptstyle 1$};
\node at (0,1.8) {$\scriptstyle 1$};
\node at (1,1.8) {$\scriptstyle 1$};
\node at (1.5,1.8) {$\scriptstyle 1$};
\end{tikzpicture}
=
\begin{tikzpicture}[baseline = 1.5mm, scale=.4, color=\clr]
%\draw[-, line width=2pt] (0.5,2) to (0.5,2.5);
\draw[-, line width=1.2pt] (0.5,0) to (0.5,-.8);
%\draw (0.5,-.51) \bdot;
%\node at (1,-.5) {$\scriptstyle \omega_r$};
\node at (0.5,-1.1) {$\scriptstyle r$};
\draw[-,thin]  (-0.5,1.5) to (-.5,1)
to[out=down,in=left] (0.5,0);
\draw[-,thin]  (0,1.5) to (0,1)
 to[out=down,in=left] (0.5,0);      
\draw[-,thin] (0.5,0)to[out=right,in=down] (1.5,1)
to (1.5,1.5);
\draw[-,thin] (0.5,0)to[out=right,in=down] (1,1) to (1,1.5);
\node at (0.5,.5){$\scriptstyle \cdots$};
\draw (-0.5,1) \bdot; 
\node at (-1,1) {$\scriptstyle g$};
\draw (0,1) \bdot; 
\node at (.25,1) {$\scriptstyle g$};
\draw (1,1) \bdot;
%\node at (.75,1) {$\scriptstyle f$};
\draw (1.5,1) \bdot; 
\node at (1.8,1) {$\scriptstyle g$};
\node at (-.5,1.8) {$\scriptstyle 1$};
\node at (0,1.8) {$\scriptstyle 1$};
\node at (1,1.8) {$\scriptstyle 1$};
\node at (1.5,1.8) {$\scriptstyle 1$};
\end{tikzpicture},  
\qquad\qquad \text{ where } g=g_{1,u}.
\end{equation}
Then \eqref{equ:cycdoymove} follows by using Lemma \ref{lem:gru} and \eqref{omegaovertor} repeatedly.
\end{proof}

\bibliographystyle{alpha}
\bibliography{affineWeb}

\end{document}